\documentclass{elsarticle}

\usepackage{amssymb}

\usepackage{amsthm}

\usepackage{amsmath}
\usepackage{graphicx,tabularx}
\usepackage{psfrag}

\usepackage{cprotect}

\makeatletter
\def\ps@pprintTitle{%
 \let\@oddhead\@empty
 \let\@evenhead\@empty
 \def\@oddfoot{\centerline{\thepage}}%
 \let\@evenfoot\@oddfoot}
\makeatother

\newtheorem{thm}{Theorem}[section]
\newtheorem{lemma}[thm]{Lemma}
\newtheorem{cor}[thm]{Corollary}
\newtheorem{prop}[thm]{Proposition}
\newtheorem{exa}[thm]{Example}

\theoremstyle{remark}
\newtheorem{remark}[thm]{Remark}
\newtheorem{openproblem}[thm]{Open Problem}

\newcounter{myenumi}

\newsavebox{\cmm}
\savebox{\cmm}{\indent}
\newenvironment{myenumerate}[1]{
\begin{list}{
{\bf #1~\themyenumi}. } {\labelwidth=0pt
\labelsep=0pt\leftmargin=0pt\usecounter{myenumi}} }{\end{list}}

\newenvironment{enumerateB}{
\begin{enumerate}
  \setlength{\itemsep}{3pt}
  \setlength{\parskip}{0pt}
}{\end{enumerate}}

\begin{document}

\begin{frontmatter}

\title{A family of non-Cayley cores based on vertex-transitive or strongly regular self-complementary graphs}

\author{Marko Orel}

\address{University of Primorska, FAMNIT, Glagolja\v{s}ka 8, 6000 Koper, Slovenia}
\address{University of Primorska, IAM, Muzejski trg 2, 6000 Koper, Slovenia}
\address{IMFM, Jadranska 19, 1000 Ljubljana, Slovenia}
\ead{marko.orel@upr.si}

\begin{abstract}
Given a finite simple graph $\Gamma$ on $n$ vertices its complementary prism is the graph~$\Gamma\bar{\Gamma}$ that is obtained from~$\Gamma$ and its complement $\bar{\Gamma}$ by adding a perfect matching, where each its edge connects two copies of the same vertex in $\Gamma$ and~$\bar{\Gamma}$. It generalizes the Petersen graph, which is obtained if $\Gamma$ is the pentagon. The automorphism group of $\Gamma\bar{\Gamma}$ is described for arbitrary graph~$\Gamma$. In particular, it is shown that the ratio between the cardinalities of the automorphism groups of $\Gamma\bar{\Gamma}$ and $\Gamma$ can attain only values $1$, $2$, $4$, and~$12$. It is shown that the Cheeger number of $\Gamma\bar{\Gamma}$ equals either 1 or $1-\frac{1}{n}$, and the two corresponding classes of~graphs are fully determined. It is proved that $\Gamma\bar{\Gamma}$ is vertex-transitive if and only if $\Gamma$ is vertex-transitive and self-complementary. In this case the complementary prism is Hamiltonian-connected whenever $n>5$, and is not a Cayley graph whenever $n>1$. The main results involve endomorphisms of graph $\Gamma\bar{\Gamma}$. It is shown that $\Gamma\bar{\Gamma}$  is a core, i.e. all its endomorphisms are automorphisms, whenever $\Gamma$ is strongly regular and self-complementary. The same conclusion is obtained for many vertex-transitive self-complementary graphs. In particular, it is shown that if there exists a vertex-transitive self-complementary graph~$\Gamma$ such that $\Gamma\bar{\Gamma}$  is not a core, then $\Gamma$ is neither a core nor its core is a complete graph.
\end{abstract}


\end{frontmatter}

\section{Introduction}

The existence of a homomorphism $\varphi : \Gamma_1\to\Gamma_2$ between two graphs provides information on the relation between the two values of various graph invariants for $\Gamma_1$ and $\Gamma_2$. For example, the clique number and the chromatic number of $\Gamma_1$ cannot exceed the corresponding value for $\Gamma_2$. Since many of these invariants are difficult to compute, the study of graph homomorphisms gain on importance in the last couple of decades, and some systematic treatment appeared in the literature~\cite{nesetril,godsil_knjiga,hahn}. On the other hand it is also not surprising that exploring graph homomorphisms between two given graphs is often a hard task to do. Related to this research is a notion of a core, which is a graph with the property that each its endomorphism is an automorphism. Each graph $\Gamma$ has a subgraph~$\Gamma'$, which is a core and such that there exists some homomorphism $\varphi: \Gamma\to \Gamma'$. It is referred as the core of $\Gamma$ and it is unique up to an isomorphism. The investigation of cores is mostly focused on graphs that posses either some degree of `symmetry' or some nice combinatorial properties. Often their definitions rely on some algebraic or geometric structure. It turns out that many such graphs are either cores or their cores are complete graphs. In fact, Cameron and Kazanidis~\cite{cameron} proved that this property is shared by all graphs with the automorphism group that acts transitively on (unordered) pairs of nonadjacent vertices. The same kind of a result was obtained by Godsil and Royle~\cite{godsil2011} for connected regular graphs with the automorphism group acting transitively on (unordered) pairs of vertices at distance two. Recently Roberson proved that each primitive strongly regular graph is either a core or its core is complete~\cite{roberson-strgreg}. There are other examples of graphs with the same property, which show that the union of the graph families in~\cite{cameron,godsil2011,roberson-strgreg} is not optimal (cf. \cite{LAA,JACO,E-JC,JCTA}). In these cases it is often very difficult to decide, which of the two possibilities regarding the core is true. In fact, for certain classes of graphs this question translates to some of longstanding open problems in finite geometry~\cite{cameron,JCTA}. For some other results related to cores see~\cite{roberson-universal,roberson-colorings,mancinska,nesetril-samal,roberson-vertex}.

A famous core is the Petersen graph (cf.~\cite{godsil_knjiga}). Its importance (cf.~\cite{petersen}) resulted in several generalizations of this graph. Perhaps the most known is the cubic version that was introduced in~\cite{coxeter,watkins2}. Yet there exist other generalizations with combinatorial/algebraic flavor. Among them the most studied family consists of Kneser graphs $K(n,r)$, which are cores whenever $2r<n$~\cite[Theorem~7.9.1]{godsil_knjiga}. If we consider a graph with the vertex set consisting of all invertible $n\times n$ Hermitian matrices over the field with four elements, where a pair of matrices $\{A, B\}$ form an edge if and only if $\textrm{rank}(A-B)=1$, we obtain another family of cores~\cite{E-JC}. It generalizes the Petersen graph, which is obtained if $n=2$. In this paper we study the complementary prism $\Gamma\bar{\Gamma}$ of a graph $\Gamma$. This is another generalization of the Petersen graph, which is obtained if $\Gamma$ is the pentagon. Since the later graph is vertex-transitive, strongly regular, and self-complementary,  the author's initial aim was to the study the core of~$\Gamma\bar{\Gamma}$ under the assumption of (a combination of) these three graph properties. During the research other motivation to study $\Gamma\bar{\Gamma}$ appeared, as we explain in the sequel.

The complementary prism was introduced in~\cite{haynes2007}. Despite it was studied in several papers (see for example \cite{cp1,cp2,cp3,carvalho2018,cp4,cp6,cp5,cp7,cp13,cp8,cp9,cp10,haynes2007,cp14,cp11,cp15,cp12}) it turns out that many important graph invariants of a complementary prism are not determined yet. One such invariant is the automorphism group $\textrm{Aut}(\Gamma\bar{\Gamma})$, which we describe for arbitrary finite simple graph~$\Gamma$. Recall that each automorphism of $\Gamma$ is also an automorphism of the complement $\bar{\Gamma}$. Hence, it induces a natural automorphism of $\Gamma\bar{\Gamma}$. Similarly, each antimorphism of $\Gamma$, if it exists, induces a natural member of $\textrm{Aut}(\Gamma\bar{\Gamma})$. Are there any other automorphisms of the complementary prism? For some families of graphs the answer is positive, and the most difficult part of determining $\textrm{Aut}(\Gamma\bar{\Gamma})$ in this paper was to identify these families. The automorphism group  $\textrm{Aut}(\Gamma\bar{\Gamma})$ is by definition related to our study of graph homomorphisms. Yet, its examination yields other benefits. With it we are able to partition all finite simple graphs into four classes with respect to the fraction $|\textrm{Aut}(\Gamma\bar{\Gamma})|/|\textrm{Aut}(\Gamma)|$, which can attain only values 1, 2, 4, and 12. Moreover, we prove that $\Gamma\bar{\Gamma}$ is vertex-transitive if and only if $\Gamma$ is vertex-transitive and self-complementary, while it is not a Cayley graph if $\Gamma$ has more than one vertex. Consequently, the study of $\Gamma\bar{\Gamma}$ joins two seemingly unrelated research areas, which are the study of vertex-transitive self-complementary graphs, and the investigation of non-Cayley vertex-transitive graphs. Self-complementary graphs provided several applications in other subfields of graph theory, including some bounds on the Ramsey numbers~(see for example~\cite{sc-survey} and the references therein). Non-Cayley vertex-transitive graphs are much more illusive than their Cayley counterparts and got  much attention in the past. We postpone the `state-of-the-art' of these two areas to the end of Section~\ref{section-auto}, where our corresponding results are presented in more details.

Since the graph spectrum is one of the key ingredients to study the core of $\Gamma\bar{\Gamma}$, we investigate also two other concepts of the complementary prisms that are related to the spectrum. One of them is represented by Hamiltonian properties of the complementary prism, the other is its Cheeger number (a.k.a. isoperimetric number) $h(\Gamma\bar{\Gamma})$. Since it can be difficult to compute the Cheeger number of a graph from a complexity point of view (cf.~\cite{mohar2,golovach}), it is rather unexpected that we are able to determine $h(\Gamma\bar{\Gamma})$ for each finite simple graph $\Gamma$.

The paper is organized as follows. In Section~2 we recall several results from various subareas of graph theory, which we apply in the proofs. In Section~3 we determine the automorphism group of a complementary prims (Theorem~\ref{p3} and Propositions~\ref{p4},~\ref{p5}), and present the applications that are mentioned above. Section~4 is devoted to Hamiltonian aspects of $\Gamma\bar{\Gamma}$ and to its Cheeger number. In particular, we prove that graph $\Gamma\bar{\Gamma}$ is Hamiltonian-connected whenever $\Gamma$ is a self-complementary graph on more than five vertices, which is either strongly regular or vertex-transitive (Corollary~\ref{p9}). The Cheeger number~$h(\Gamma\bar{\Gamma})$ is computed in Proposition~\ref{cheeger}. In Section~5 we study the core of a complementary prism. In particular, we prove that $\Gamma\bar{\Gamma}$ is a core whenever $\Gamma$ is self-complementary and strongly regular (Theorem~\ref{thm-strg}). Here, the proof relies on the properties of the Lov\'{a}sz theta function. Despite Section~5 provides substantial information on the core of $\Gamma\bar{\Gamma}$ for arbitrary finite simple graph $\Gamma$, the rest of the section is mainly focused on the case, where $\Gamma\bar{\Gamma}$ is vertex-transitive. In particular, Theorem~\ref{thm-vt} proves that $\Gamma\bar{\Gamma}$ is a core for each vertex-transitive self-complementary graph $\Gamma$, which is either a core or has a complete core.

Despite the topics in Sections 3-5 are interconnected, and some lemmas from Section~2 are applied to prove multiple results from  Sections 3-5, it turns out that from a technical point of view the proofs in Sections 3-5 are pairwise independent. Hence, the reader interested only in some aspects of this paper, should be able to read just his favorite Section among 3-5.

Lastly we mention that the paper states also two open problems. Open Problem~\ref{op} questions the existence of a vertex-transitive self-complementary graph $\Gamma$ such that $\Gamma\bar{\Gamma}$ is not a core. Another open problem is about the existence of a self-complementary graph, which is vertex-transitive (or at least regular) and has a `large' second eigenvalue of the adjacency matrix (see Remark~\ref{opomba2}).

\section{Preliminaries}\label{preliminaries}

This section is divided into five subsections. Subsections 2.2-2.5 are devoted to self-complementary graphs, graph homomorphisms, lexicographic product, and complementary prism, respectively. Subsection~2.1 contains the rest. While all of the results in this section are needed as lemmas for Sections~3-5, some of them seem to be new as we did not find them in the literature. For example, Lemma~\ref{sr-1wr} proves the equivalence between strong regularity and 1-walk regularity in the class of self-complementary graphs. Lemma~\ref{hamconnected} proves that each self-complementary graph on more than five vertices, which is either strongly regular or vertex-transitive, is Hamiltonian-connected.

\subsection{General graph theory}

All graphs in this paper are simple and finite unless otherwise stated. We assume that the vertex set $V(\Gamma)$ of a graph $\Gamma$ is nonempty unless otherwise stated in which case we refer to $\Gamma$ as the \emph{null graph}. In the rare occasions, where we allow the vertex set to be empty, we use the symbol $\Lambda$ for the graph. That is, $|V(\Gamma)|\geq 1$ and $|V(\Lambda)|\geq 0$, where $|.|$ denotes the cardinality. If two vertices $u,v\in V(\Gamma)$ form an edge, we write $u\sim_{\Gamma} v$ or simply $u\sim v$. Hence, $E(\Gamma)=\{\{u,v\} : u,v\in V(\Gamma), u\sim v\}$, $N_{\Gamma}(v)=\{u\in V(\Gamma) : u\sim v\}$, and $N_{\Gamma}[v]=N_{\Gamma}(v)\cup \{v\}$ are the edge set, the neighborhood of the vertex $v$, and the closed neighborhood of the vertex $v$, respectively. The degree $\delta_{\Gamma}(v)=|N_{\Gamma}(v)|$ of a vertex $v\in V(\Gamma)$ is often abbreviated as $\delta(v)$. Given a connected graph~$\Gamma$, we denote by $\textrm{diam}(\Gamma)$ its diameter, while $d_{\Gamma}(u,v)$ is the distance between vertices $u,v\in V(\Gamma)$. We use $K_n$, $C_n$, and $P_n$ to denote the complete graph, the cycle, and the path on $n$ vertices, respectively. The complement of a graph~$\Gamma$ is denoted by $\bar{\Gamma}$. A subset $C\in V(\Gamma)$ is a \emph{clique} if it induces a complete graph in~$\Gamma$. Similarly, a subset $I\in V(\Gamma)$ is \emph{independent} if it induce a complete graph in the complement $\bar{\Gamma}$. The \emph{clique number} $\omega(\Gamma)$ and the \emph{independence number}~$\alpha(\Gamma)$ are the orders of the largest clique and independent set in $\Gamma$, respectively. We use $\chi(\Gamma)$ do denote the chromatic number of a graph $\Gamma$. It is well known that $\chi(\Gamma)\geq \omega(\Gamma)$ and $\chi(\Gamma)\geq \frac{|V(\Gamma)|}{\alpha(\Gamma)}$ (cf.~\cite{brouwer-haemers}). The \emph{vertex-connectivity}~$\kappa(\Gamma)$ of a non-complete connected graph $\Gamma$ is defined as the minimum number of vertices, removal of which disconnects the graph. The \emph{Cheeger number} $h(\Gamma)$ (or the \emph{isoperimetric constant} or the \emph{edge expansion constant}) of a graph $\Gamma$ is the minimum value of the fraction $e(S,T)/|S|$, where $\{S,T\}$ runs over all partitions of~$V(\Gamma)$ such that $1\leq |S|\leq |T|$, and $e(S,T)$ denotes the number of edges meeting both $S$ and $T$ (cf.~\cite{brouwer-haemers,mohar2}).

If $|V(\Gamma)|=n$, then a sequence $(v_1,v_2,\ldots,v_n)$ of distinct vertices form a \emph{Hamiltonian path} if $v_i\sim v_{i+1}$ for all $1\leq i\leq n-1$. It is a \emph{Hamiltonian cycle} if in addition $v_n\sim v_1$. A graph $\Gamma$ is \emph{Hamiltonian-connected} if for each pair of vertices $u,v\in V(\Gamma)$ there exists a Hamiltonian path in $\Gamma$ that joins $u$ and $v$. Clearly, each Hamiltonian-connected graph has a Hamiltonian cycle.

Given graphs $\Gamma_1, \Gamma_2$, a map $\varphi: V(\Gamma_1)\to V(\Gamma_2)$ is a \emph{graph homomorphism} if it obeys the implication
\begin{equation}\label{e104}
\{u,v\}\in E(\Gamma_1)\Longrightarrow \{\varphi(u),\varphi(v)\}\in E(\Gamma_2)
\end{equation}
for all $u,v\in V(\Gamma_1)$. In particular, $\varphi(u)\neq \varphi(v)$ for the two endpoints of each edge in $\Gamma_1$. If in addition $\varphi$ is bijective, and the implication in~\eqref{e104} is replaced by the equivalence $\Longleftrightarrow$, then $\varphi$ is a \emph{graph isomorphism} and graphs $\Gamma_1, \Gamma_2$ are \emph{isomorphic}, which is denoted by $\Gamma_1\cong \Gamma_2$. The sets of all graph homomorphisms/isomorphisms from $\Gamma_1$ to $\Gamma_2$ are denoted by $\textrm{Hom}(\Gamma_1,\Gamma_2)$ and  $\textrm{Iso}(\Gamma_1,\Gamma_2)$, respectively. Similarly, we use
$\textrm{End}(\Gamma)=\textrm{Hom}(\Gamma,\Gamma)$ and $\textrm{Aut}(\Gamma)=\textrm{Iso}(\Gamma,\Gamma)$ to denote the sets of all \emph{graph endomorphisms} and \emph{automorphisms}, respectively. Clearly, $\textrm{Aut}(\Gamma)$ forms a group, where the composition of maps is the group operation, which we denote by $(\varphi_2\circ \varphi_1)(v)=\varphi_2\big(\varphi_1(v)\big)$ for all $v\in V(\Gamma)$.

A subgroup $G$ in $\textrm{Aut}(\Gamma)$ \emph{acts transitively} on $V(\Gamma)$ if for each pair of vertices $u,v\in V(\Gamma)$ there exists $\varphi\in G$ such that $\varphi(u)=v$. A graph $\Gamma$ is \emph{vertex-transitive} if $\textrm{Aut}(\Gamma)$ acts transitively on $V(\Gamma)$. Similarly, $\Gamma$ is \emph{edge-transitive} or \emph{arc-transitive} if $\textrm{Aut}(\Gamma)$ acts transitively on $E(\Gamma)$ or on the set of all arcs in $\Gamma$, respectively. Let $G$ be a group and $S$ its subset, which does not contain the identity element and $S=S^{-1}:=\{s^{-1} : s\in S\}$. Then the \emph{Cayley graph} $\textrm{Cay}(G,S)$ has the group $G$ as the vertex set, and two vertices $g_1,g_2\in G$ form an edge if and only if $g_1^{-1}g_2\in S$. It is well known that each Cayley graph is vertex-transitive. A graph $\Gamma$ is \emph{$k$-regular} if $\delta_{\Gamma}(v)=k$ for all $v\in V(\Gamma)$. It is \emph{regular} if it is $k$-regular for some $k$. A graph $\Gamma$ is \emph{strongly regular} with parameters $(n,k,\lambda,\mu)$ if it is a $k$-regular graph on $n$ vertices such that
$$|N_{\Gamma}(u)\cap N_{\Gamma}(v)|=\left\{\begin{array}{ll} \lambda &\textrm{if}\ u\sim v,\\
\mu &\textrm{if}\ u\nsim v, u\neq v,
 \end{array}\right.$$
for all distinct $u,v\in V(\Gamma)$. Given a graph $\Gamma$ with the vertex set $\{v_1,\ldots,v_n\}$ let $A$ be its \emph{adjacency matrix}, i.e. a $n\times n$ real matrix with $(i,j)$-th entry equal to 1 if $v_i\sim v_j$ and 0 otherwise. In this paper its eigenvalues are referred as the eigenvalues of $\Gamma$, and we number them in the decreasing order $\lambda_1\geq \cdots\geq \lambda_n$. If $\Gamma$ is $k$-regular, then $k=\lambda_1$ (cf.~\cite{brouwer-haemers}). A graph $\Gamma$ is \emph{1-walk regular} if for each positive integer $i$ there are constants $a_i, b_i$ such that $A^i\bullet I=a_i I$, $A^i\bullet A=b_i A$, where $I$ is the identity matrix and $\bullet$ is the entry-wise product (cf.~\cite{roberson-thetas,roberson-colorings}). In other words, a graph is 1-walk regular if for each $i$ both (a) the number of walks
of length $i$ starting and ending at a vertex does not depend on the choice of vertex and (b) the number of walks of length $i$ between the endpoints of an edge does not depend on the edge. Clearly, 1-walk regular graphs include strongly regular graphs and graphs, which are both vertex-transitive and edge-transitive.

Lemma~\ref{l3} is well known. We provide a short proof for reader's convenience.
\begin{lemma}\label{l3}
Let $\Gamma$ be any graph. Then $\Gamma$ or $\bar{\Gamma}$ is connected.
\end{lemma}
\begin{proof}
Suppose that $\Gamma$ is not connected. Let $\Gamma_1,\ldots,\Gamma_m$ be its connected components, where $m\geq 2$. Pick arbitrary vertices $v$ and $w$ in $V(\bar{\Gamma})$. Then there exist $i,j\in \{1,\ldots,m\}$ such that $v\in V(\Gamma_i)$ and $w\in V(\Gamma_j)$. If $i\neq j$, then $v\sim_{\bar{\Gamma}} w$. If $i=j$, then we can pick arbitrary $l\in\{1,\ldots,m\}\backslash\{i\}$ and arbitrary vertex $u\in V(\Gamma_l)$. Then $v\sim_{\bar{\Gamma}} u\sim_{\bar{\Gamma}} w$. Hence, $\bar{\Gamma}$ is connected.
\end{proof}

Lemma~\ref{clique-coclique} is applied extensively throughout the paper. The vertex-transitive part can be found in~\cite[Corollaries~2.1.2, 2.1.3]{godsil_cambridge}, where it is stated in more general settings. The claim for 1-walk-regular graphs is proved in \cite[Corollary~4.1]{roberson-thetas}.

\begin{lemma}\label{clique-coclique}
If a graph $\Gamma$ is vertex-transitive or 1-walk-regular, then $\alpha(\Gamma)\omega(\Gamma)\leq |V(\Gamma)|$. If the equality holds in the case of a vertex-transitive graph, then $|C\cap I|=1$  for each clique $C$ and each independent set $I$ that attain the equality.
\end{lemma}

The next result is the well known Sabidussi's Theorem~\cite[Lemma~4]{sabidussi}, which we apply in Corollary~\ref{c2}.
\begin{lemma}\label{sabid}
A graph $\Gamma$ is a Cayley graph if and only if there exists a subgroup $G$ in $\textrm{Aut}(\Gamma)$ that acts transitively on $V(\Gamma)$ and $|G|=|V(\Gamma)|$.
\end{lemma}

In Example~\ref{ex:misteriozen} we rely on the clique number and on the chromatic number of the Kneser graph and its complement. Recall that the \emph{Kneser graph} $K(n,r)$, for $2r<n$, has the vertex set formed by all $r$-subsets in $\{1,2,\ldots,n\}$, where two $r$-subsets form and edge if and only if they are disjoint. It is obvious that $\omega(K(n,r))=\lfloor \frac{n}{r}\rfloor$. On the other hand the Erd\H{o}s-Ko-Rado theorem states that $\omega(\overline{K(n,r)})=\alpha(K(n,r))=\binom{n-1}{r-1}$ (cf.~\cite[Theorem~7.8.1]{godsil_knjiga}). The famous Lov\'{a}sz proof of the Kneser conjecture~\cite{lovasz1978} implies that $\chi(K(n,r))=n-2r+2$ (see also \cite[Theorem~7.11.4]{godsil_knjiga}). The chromatic number of the complement is computed in \cite[Corollary~4, p.~93]{baranyai} and equals
$$\chi(\overline{K(n,r)})=\left\lceil\frac{\binom{n}{r}}{\lfloor\frac{n}{r}\rfloor}\right\rceil.$$
In particular,
$$\omega(K(10,4))=2,\ \chi(K(10,4))=4,\  \omega(\overline{K(10,4)})=84,\ \chi(\overline{K(10,4)})=105.$$
Moreover, if $V=\{S\subseteq \{1,2,\ldots,10\} : |S|=4 \}$ is the vertex set of $K(10,4)$, then a proper $4$-vertex-coloring $c: V\to \{1,2,3,4\}$ can be obtained by defining $c(S)=\min_{s\in S} s$ whenever $\min_{s\in S}\leq 3$ and $c(S)=4$ otherwise. Hence, if we add the edge $e=\{\{1,2,3,4\},\{2,3,4,5\}\}$ to graph $K(10,4)$, then the newly obtained graph $K(10,4) + e$ satisfies
\begin{equation}\label{e86}
\chi(K(10,4) + e)=4\ \textrm{and}\ \omega(K(10,4) + e)\geq 2,
\end{equation}
whereas the graph $\overline{K(10,4)} - e$ that is obtained by deleting the edge $e$ satisfies
\begin{equation}\label{e87}
\chi(\overline{K(10,4)} - e)\geq 104 \ \textrm{and}\ \omega(\overline{K(10,4)} - e)\leq 84.
\end{equation}

Lemma~\ref{DAM} is proved in \cite[Theorem~2.6]{DAM2008}.
\begin{lemma}\label{DAM}
If $\Gamma$ and $\bar{\Gamma}$ are connected graphs with minimum degrees $\delta(\Gamma)$ and $\delta(\bar{\Gamma})$, respectively, then
$\kappa(\Gamma)+\kappa(\bar{\Gamma})\geq \min\{\delta(\Gamma),\delta(\bar{\Gamma})\}+1$.
\end{lemma}

Lemma~\ref{connectivity_hamilton} is the well-known Chv\'{a}tal-Erd\H{o}s Theorem~\cite[Theorem~3]{chvatal1972}.

\begin{lemma}\label{connectivity_hamilton}
If $\alpha(\Gamma)<\kappa(\Gamma)$, then $\Gamma$ is Hamiltonian-connected.
\end{lemma}

We apply Lemmas~\ref{DAM} and~\ref{connectivity_hamilton} along Lemma~\ref{sr-1wr} to prove Lemma~\ref{hamconnected}, which is subsequently applied to prove Corollary~\ref{p9}.

\subsection{Self-complementary graphs}\label{subsection-sc}

An isomorphism between the graph $\Gamma$ and its complement $\bar{\Gamma}$ is an \emph{antimorphism} (or a \emph{complementing permutation}) of $\Gamma$. The set $\textrm{Iso}(\Gamma,\bar{\Gamma})=\textrm{Iso}(\bar{\Gamma},\Gamma)$ of all such maps is shortly denoted by $\overline{\textrm{Aut}(\Gamma)}$. If it is nonempty, then $\Gamma$ is a \emph{self-complementary} graph. In this case $\textrm{Aut}(\Gamma)\cup \overline{\textrm{Aut}(\Gamma)}$ is a group with composition of maps as the group operation, and $\textrm{Aut}(\Gamma)$ is its subgroup of index two, i.e. $|\textrm{Aut}(\Gamma)|=|\overline{\textrm{Aut}(\Gamma)}|$ (see \cite{gibbs} or \cite{sc-survey}). Moreover, a composition $\varphi_1\circ\cdots \circ \varphi_m$ of members of $\textrm{Aut}(\Gamma)\cup \overline{\textrm{Aut}(\Gamma)}$ is an antimorphism whenever $\varphi_i\in \overline{\textrm{Aut}(\Gamma)}$ for an odd number of indices $i$. Otherwise, $\varphi_1\circ\cdots \circ \varphi_m\in \textrm{Aut}(\Gamma)$ (cf. \cite{sc-survey}). Each self-complementary graph is connected, as it follows from Lemma~\ref{l3}.

The following result is well known and follows immediately from the results of Sachs~\cite{sachs} or Ringel~\cite{ringel} (see also~\cite{gibbs,sc-survey}). We apply it in Corollary~\ref{c2}.
\begin{lemma}\label{order_antim}
The order of an antimorphism of a self-complementary graph with more than one vertex is divisible by four.
\end{lemma}

If a graph on $n\geq 1$ vertices is $(\frac{n-1}{2})$-regular, then $n$ must be odd. Moreover,  it follows from the hand-shaking lemma that $n=4m+1$ for some integer $m\geq 0$. In particular, this is true for all regular self-complementary graphs. By a result of Sachs~\cite{sachs} or Ringel~\cite{ringel}, each cycle in an antimorphism has the length divisible by four, except for one cycle of length one, in the case the order of the graph equals 1 modulo 4 (a proof in English can be found in~\cite[p.~12]{sc-survey}). Lemma~\ref{fiksna_tocka} is a special case of this fact, and is crucial in the proofs of Theorem~\ref{thm-vt} and Propositions~\ref{prop-lex1}, \ref{prop-lex2}.
\begin{lemma}\label{fiksna_tocka}
If $\sigma$ is an antimorphism of a regular self-complementary graph $\Gamma$, then there exists a unique vertex $v\in V(\Gamma)$ such that $\sigma(v)=v$.
\end{lemma}

The smallest two (non-null) regular self-complementary graphs are $K_1$ and the pentagon $C_5$. In total there are 36 (nonisomorphic) self-complementary graphs on 9 vertices. They were firstly described in~\cite{morris}. Among them only 4 are regular (see Figure~\ref{f8}). Their adjacency matrices are listed in Figure~\ref{f9} in the same  order.
\begin{figure}
\centering
\begin{tabular}{ccc}
\includegraphics[width=0.4\textwidth]{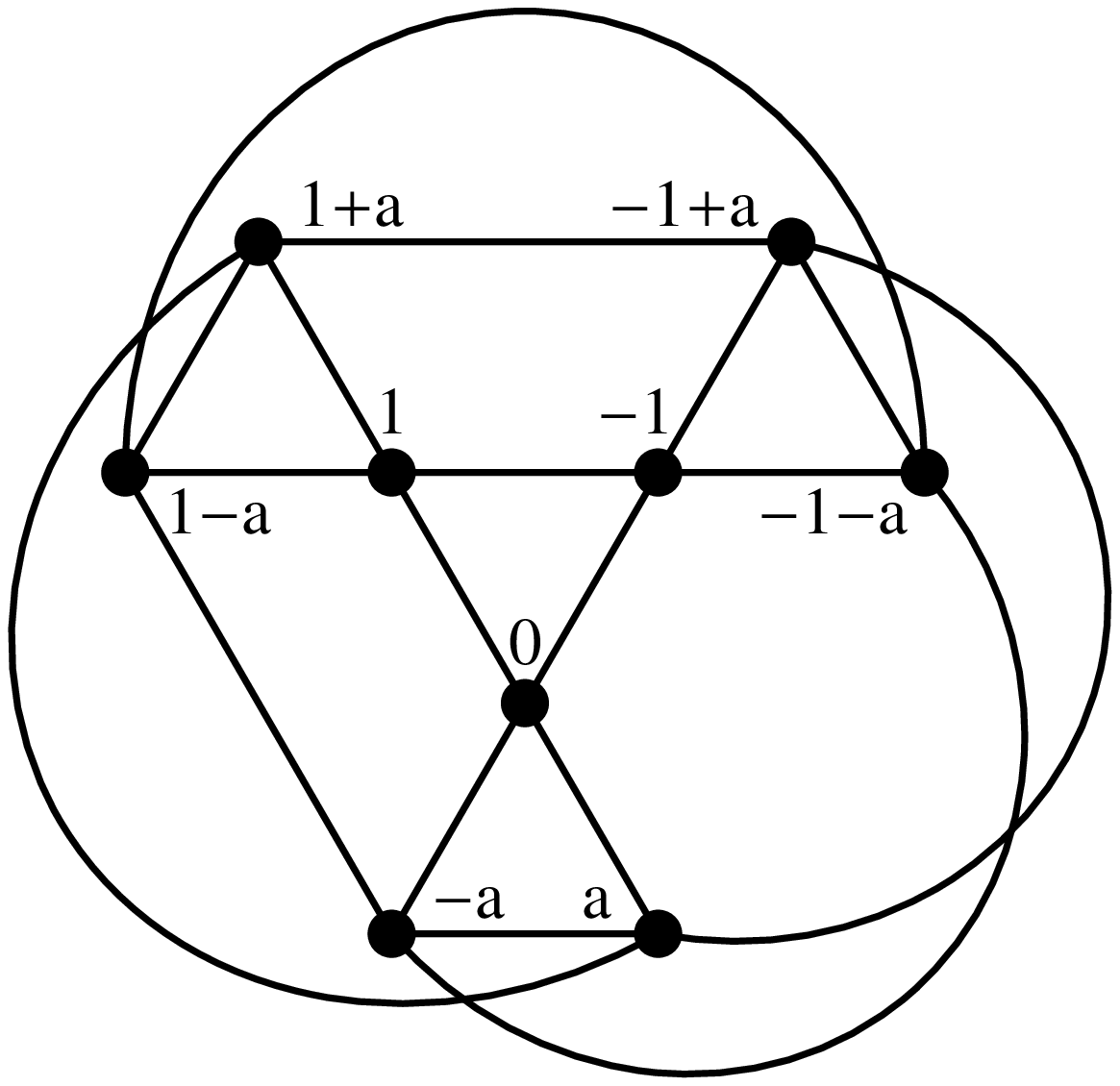}& &
\includegraphics[width=0.4\textwidth]{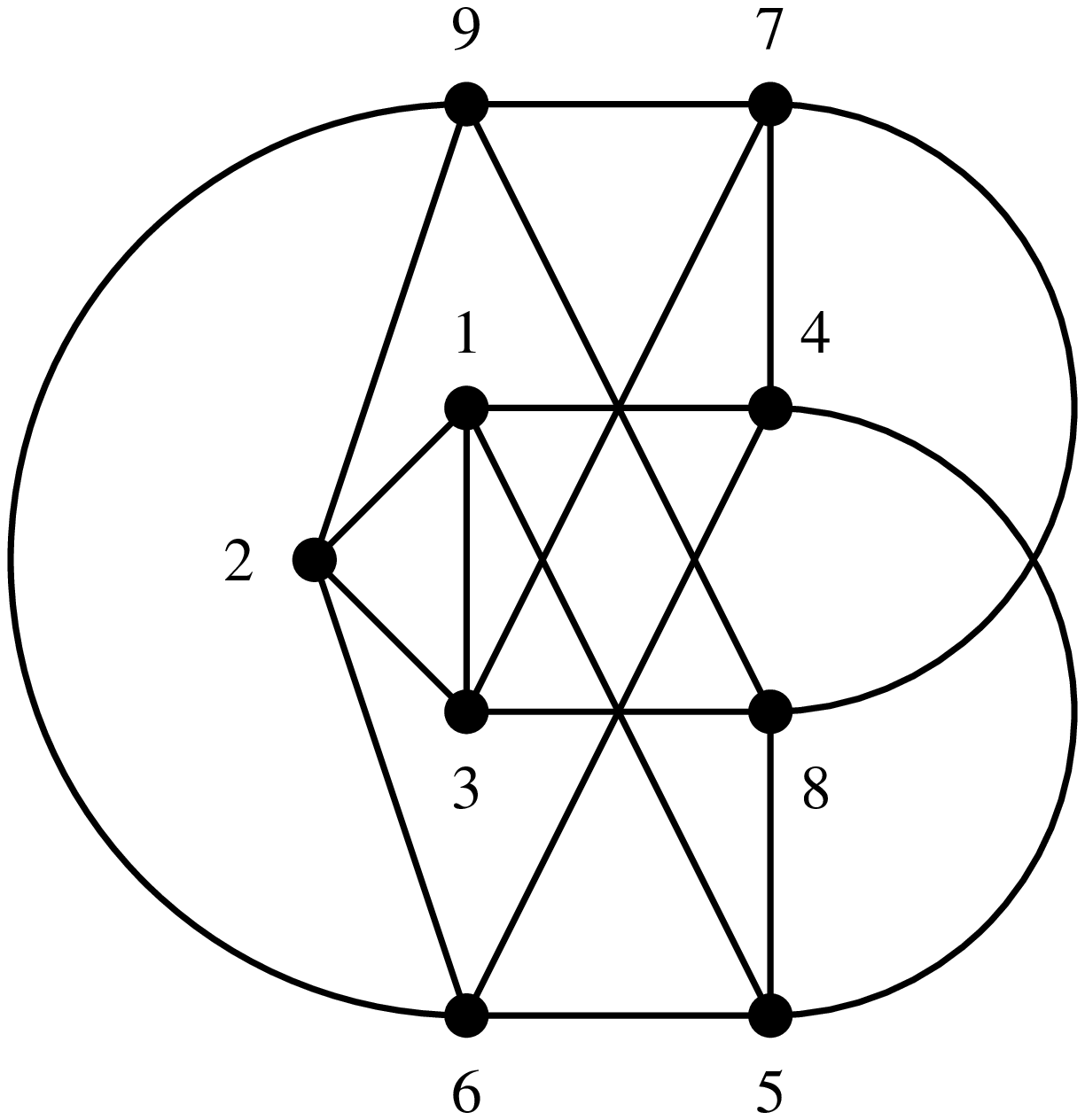}\\
\includegraphics[width=0.4\textwidth]{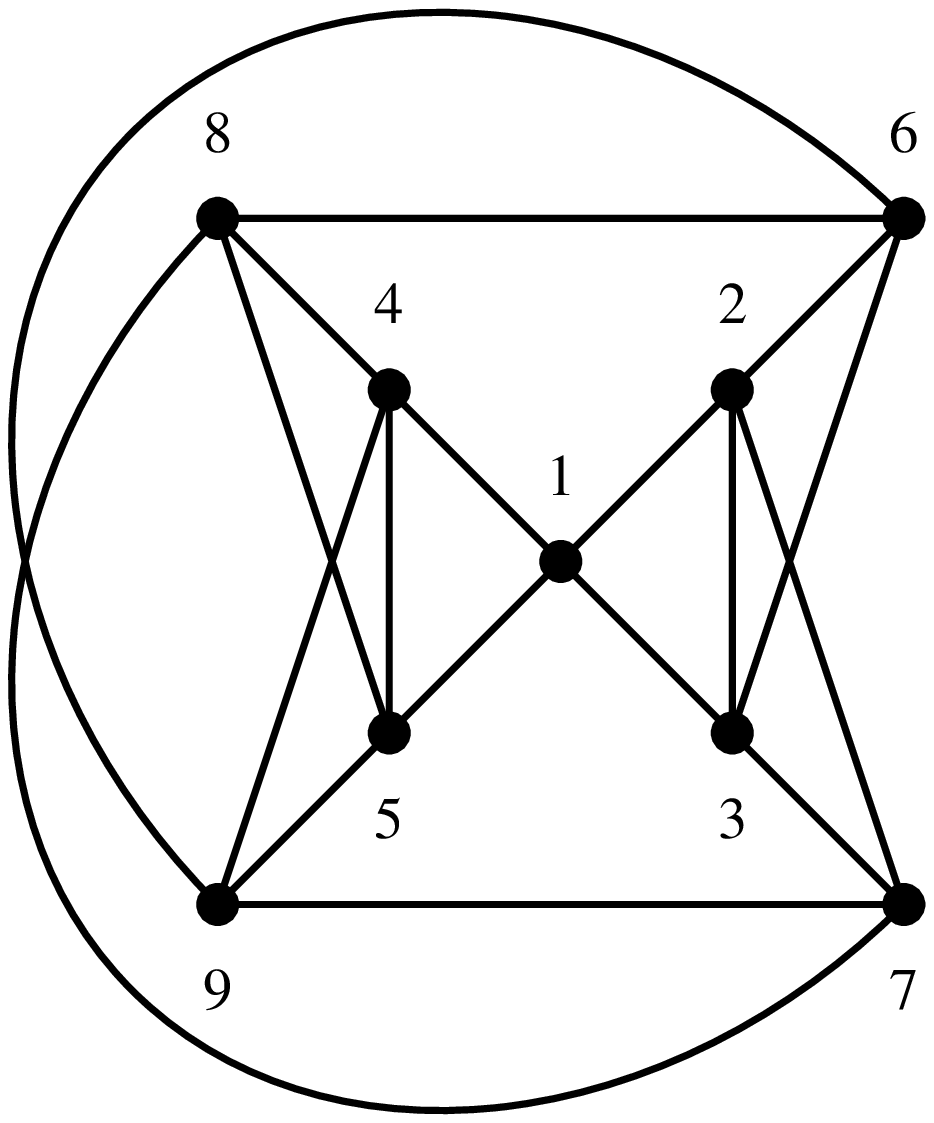}& &
\includegraphics[width=0.4\textwidth]{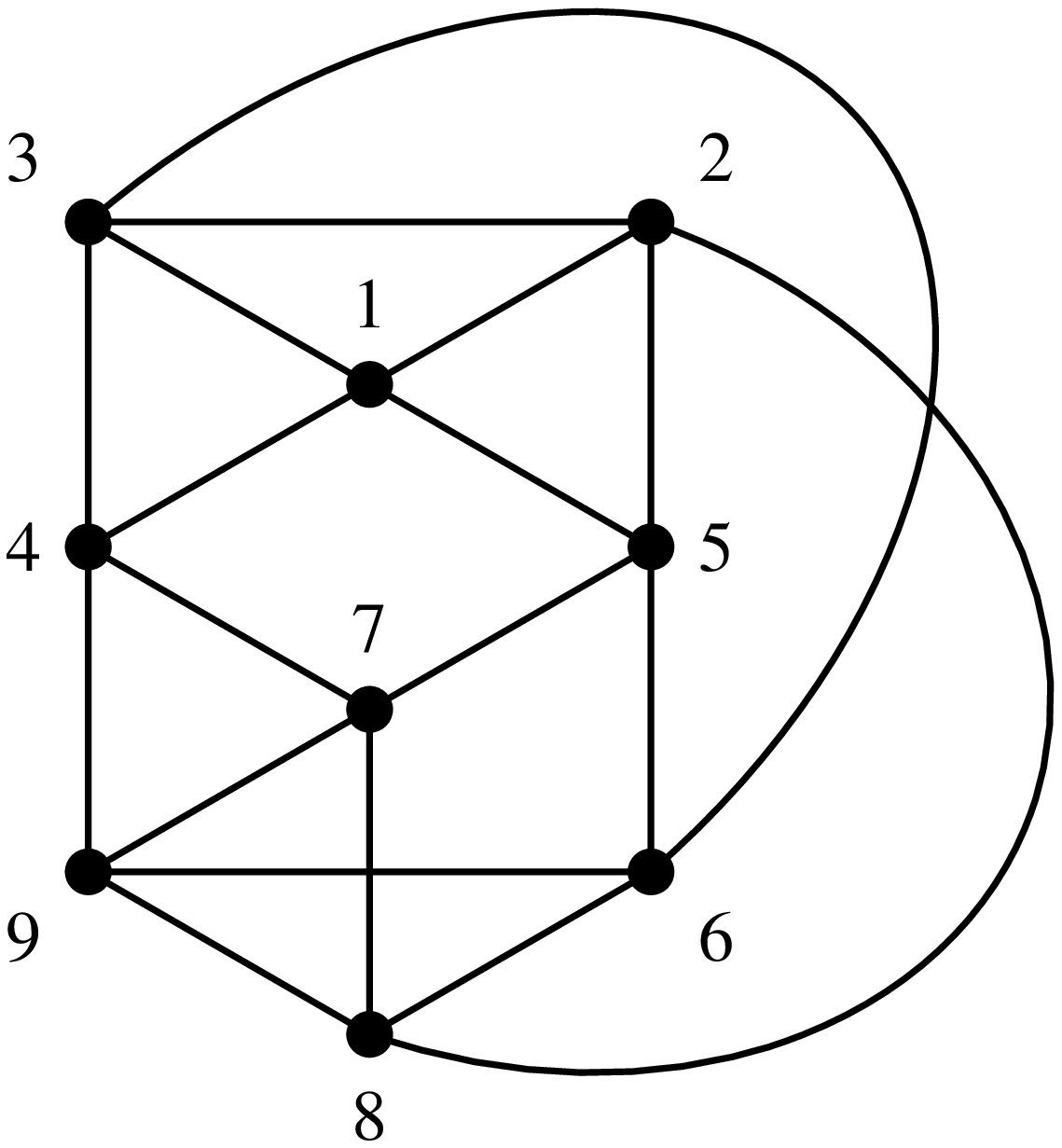}\\
\end{tabular}
\caption{Regular self-complementary graphs on 9 vertices. The Paley graph is in the top-left.}
\label{f8}
\end{figure}
\begin{figure}
\centering
{\tiny
\begin{align*}
A_1&=\left(
\begin{array}{ccccccccc}
 0 & 1 & 1 & 1 & 1 & 0 & 0 & 0 & 0 \\
 1 & 0 & 1 & 0 & 0 & 1 & 1 & 0 & 0 \\
 1 & 1 & 0 & 0 & 0 & 0 & 0 & 1 & 1 \\
 1 & 0 & 0 & 0 & 1 & 1 & 0 & 1 & 0 \\
 1 & 0 & 0 & 1 & 0 & 0 & 1 & 0 & 1 \\
 0 & 1 & 0 & 1 & 0 & 0 & 1 & 1 & 0 \\
 0 & 1 & 0 & 0 & 1 & 1 & 0 & 0 & 1 \\
 0 & 0 & 1 & 1 & 0 & 1 & 0 & 0 & 1 \\
 0 & 0 & 1 & 0 & 1 & 0 & 1 & 1 & 0
\end{array}
\right),\quad A_2=\left(
\begin{array}{ccccccccc}
 0 & 1 & 1 & 1 & 1 & 0 & 0 & 0 & 0 \\
 1 & 0 & 1 & 0 & 0 & 1 & 0 & 0 & 1 \\
 1 & 1 & 0 & 0 & 0 & 0 & 1 & 1 & 0 \\
 1 & 0 & 0 & 0 & 1 & 1 & 1 & 0 & 0 \\
 1 & 0 & 0 & 1 & 0 & 1 & 0 & 1 & 0 \\
 0 & 1 & 0 & 1 & 1 & 0 & 0 & 0 & 1 \\
 0 & 0 & 1 & 1 & 0 & 0 & 0 & 1 & 1 \\
 0 & 0 & 1 & 0 & 1 & 0 & 1 & 0 & 1 \\
 0 & 1 & 0 & 0 & 0 & 1 & 1 & 1 & 0
\end{array}
\right)\\
&\\
A_3&=\left(
\begin{array}{ccccccccc}
 0 & 1 & 1 & 1 & 1 & 0 & 0 & 0 & 0 \\
 1 & 0 & 1 & 0 & 0 & 1 & 1 & 0 & 0 \\
 1 & 1 & 0 & 0 & 0 & 1 & 1 & 0 & 0 \\
 1 & 0 & 0 & 0 & 1 & 0 & 0 & 1 & 1 \\
 1 & 0 & 0 & 1 & 0 & 0 & 0 & 1 & 1 \\
 0 & 1 & 1 & 0 & 0 & 0 & 0 & 1 & 1 \\
 0 & 1 & 1 & 0 & 0 & 0 & 0 & 1 & 1 \\
 0 & 0 & 0 & 1 & 1 & 1 & 1 & 0 & 0 \\
 0 & 0 & 0 & 1 & 1 & 1 & 1 & 0 & 0
\end{array}
\right),\quad A_4=\left(
\begin{array}{ccccccccc}
 0 & 1 & 1 & 1 & 1 & 0 & 0 & 0 & 0 \\
 1 & 0 & 1 & 0 & 1 & 0 & 0 & 1 & 0 \\
 1 & 1 & 0 & 1 & 0 & 1 & 0 & 0 & 0 \\
 1 & 0 & 1 & 0 & 0 & 0 & 1 & 0 & 1 \\
 1 & 1 & 0 & 0 & 0 & 1 & 1 & 0 & 0 \\
 0 & 0 & 1 & 0 & 1 & 0 & 0 & 1 & 1 \\
 0 & 0 & 0 & 1 & 1 & 0 & 0 & 1 & 1 \\
 0 & 1 & 0 & 0 & 0 & 1 & 1 & 0 & 1 \\
 0 & 0 & 0 & 1 & 0 & 1 & 1 & 1 & 0
\end{array}
\right).
\end{align*}
}
\caption{The adjacency matrix $A_1$ of $\textrm{Paley}(9)$ is indexed  with respect to the vertices in the following order: $0$, $-1$, $1$, $-a$, $a$, $-1-a$, $-1+a$, $1-a$, $1+a$.}
\label{f9}
\end{figure}
The adjacency matrices were obtained from the database~\cite{repository} with the help of the online converter~\cite{converter}. The top-left graph in Figure~\ref{f8} is an example of a Paley graph $\textrm{Paley}(q)$, which has a finite field $\mathbb{F}_q$ as its vertex set, where $q=1\ (\textrm{mod}\, 4)$, and two elements in $\mathbb{F}_q$ form an edge if any only if their difference is a nonzero square element in $\mathbb{F}_q$. It is well known that each Paley graph is self-complementary. In fact, if $c\in \mathbb{F}_q$ is any nonsquare element, then the map $\sigma$ on $\mathbb{F}_q$, defined by $\sigma(x)=cx$ for all $x\in \mathbb{F}_q$, is an example of an antimorphism. It is also well known that Paley graphs are arc-transitive (consequently also vertex/edge-transitive). In fact, if $(x_1,y_1)$ and $(x_2,y_2)$ are any two arcs in $\textrm{Paley}(q)$, then the map $\varphi$, defined by $\varphi(x)=\frac{y_2-x_2}{y_1-x_1}(x-x_1)+x_2$ for all $x\in \mathbb{F}_q$, is an automorphism that maps $(x_1,y_1)$ to $(x_2,y_2)$. Hence, it follows from~\cite[Corollary~1A]{watkins} that $\kappa(\textrm{Paley}(q))=\frac{q-1}{2}$. If $q$ is a square, then we have a decomposition $\mathbb{F}_q=\mathbb{F}_{\sqrt{q}}+a \mathbb{F}_{\sqrt{q}}$, where $\mathbb{F}_{\sqrt{q}}:=\{x\in \mathbb{F}_q : x^{\sqrt{q}}=x\}$ is a subfield with $\sqrt{q}$ elements and $a\in \mathbb{F}_q\backslash \mathbb{F}_{\sqrt{q}}$ satisfies $a^{\sqrt{q}}=-a$ (in the case of the graph $\textrm{Paley}(9)$ from Figure~\ref{f8} we deduce that $a^2=-1$).

Recall that given a regular graph with eigenvalues $k=\lambda_1\geq\lambda_2\geq \cdots \geq \lambda_n$, its complement has eigenvalues $n-k-1\geq -1-\lambda_n\geq\cdots \geq -1-\lambda_2$ (cf.~\cite{brouwer-haemers}). Consequently, the largest eigenvalue of a regular self-complementary graph with $n$ vertices equals $\frac{n-1}{2}$, while other eigenvalues occur in pairs $\{\lambda, -1-\lambda\}$ that are symmetric with respect to the value $-\frac{1}{2}$. More can be said in the case of   strong regularity. In fact, the parameters $(n,k,\lambda,\mu)$ of a strongly regular self-complementary graph are of the form $(n,\frac{n-1}{2},\frac{n-5}{4},\frac{n-1}{4})$ (see \cite[Lemma~5.1.2]{mullin} or~\cite[Theorem~2]{mathon}). The spectrum of the adjacency matrix of a strongly regular graph is determined by its parameters (cf.~\cite[Theorem~9.1.3]{brouwer-haemers}). Consequently, we have the following lemma, which is stated also in~\cite[p.~82]{sc-survey}, and applied in the proof of Theorem~\ref{thm-strg}.
\begin{lemma}\label{strg}
Let $\Gamma$ be a strongly regular self-complementary graph on $n>1$ vertices. Then $\Gamma$ has eigenvalues $\frac{n-1}{2}, \frac{\sqrt{n}-1}{2}, \frac{-\sqrt{n}-1}{2}$ with multiplicities $1, \frac{n-1}{2}, \frac{n-1}{2}$, respectively.
\end{lemma}

\begin{lemma}\label{sr-1wr}
A self-complementary graph is 1-walk regular if and only if it is strongly regular.
\end{lemma}
\begin{proof}
Recall that each strongly regular graph is 1-walk regular. Conversely, suppose that $\Gamma$ is a 1-walk regular self-complementary graph. We may assume that $n:=|V(\Gamma)|>1$. Let $A$ and $\bar{A}$ be the adjacency matrices of $\Gamma$ and $\bar{\Gamma}$, respectively. Since both are 1-walk regular, for each positive integer $i$ there are constants $a_i,b_i, \bar{a}_i, \bar{b}_i$ such that $A^i\bullet I=a_i I$, $A^i\bullet A=b_i A$, $\bar{A}^i\bullet I=\bar{a}_i I$, $\bar{A}^i\bullet \bar{A}=b_i \bar{A}$. In particular, the diagonal entries $[A^2]_{vv}$ all equal $a_2$, while $[A^2]_{vu}=b_2$ whenever vertices $v$ and $u$ are adjacent in $\Gamma$. Since the valency of $\Gamma$ is $\frac{n-1}{2}$, we can square the equality $\bar{A}=J-A-I$, where $J$ is the $n\times n$ all-one matrix, to deduce that $\bar{A}^2=A^2+2A+I-J$. If we multiply the later equality entry-wise by $\bar{A}$, we deduce that $\bar{b}_2\bar{A}=A^2\bullet \bar{A}+\bar{a}_1 I - \bar{A}$. Consequently, if distinct vertices $v$ and $u$ are nonadjacent in $\Gamma$, then $[A^2]_{vu}=1+\bar{b}_2$. Hence,
$$A^2=a_2 I+b_2 A+ (1+\bar{b}_2)\bar{A}=(b_2-\bar{b}_2-1)A+(a_2-\bar{b}_2-1)I+(1+\bar{b}_2)J.$$
It now follows from \cite[Theorem~9.1.2]{brouwer-haemers} that $\Gamma$ is strongly regular.
\end{proof}

The next result follows immediately from Rao's paper~\cite[Theorem~A.1.]{rao1979}. It is applied in Proposition~\ref{p8}.
\begin{lemma}\label{regularscHamiltonskost}
Each regular self-complementary graph on $n>1$ vertices has a Hamiltonian cycle.
\end{lemma}
\begin{remark}\label{opomba}
It is well known that Hamiltonicity provides some information on graph spectrum (cf. \cite{mohar,heuvel,brouwer-haemers}). If $\frac{n-1}{2}=\lambda_1\geq\cdots\geq \lambda_n$ are eigenvalues of a graph $\Gamma$ from Lemma~\ref{regularscHamiltonskost}, then the Hamiltonian cycle provides an induced cycle $C_{n}$ in the line graph of $\Gamma$, which has eigenvalues $\lambda_1+\frac{n-5}{2}\geq\cdots \geq \lambda_n+\frac{n-5}{2}\geq -2\geq \cdots \geq -2$~\cite[Subsection~1.4.5]{brouwer-haemers}. Since the $j$-th largest eigenvalue of $C_n$ equals $2\cos(\frac{2\pi\lfloor j/2 \rfloor}{n})$ for all $j$~\cite[Subsection~1.4.3]{brouwer-haemers}, the interlacing~\cite[Section~3.2]{brouwer-haemers} and equation $\lambda_i=-1-\lambda_{n-i+2}$  yield  $\lambda_i\leq \frac{n-7}{2}-2\cos(\frac{2\pi\lfloor (n-i+2)/2 \rfloor}{n})$ for $i\in \{2,\ldots,\frac{n+1}{2}\}$. In particular, $\lambda_2\leq \frac{n-7}{2}-2\cos(\frac{\pi(n-1)}{n})<\frac{n-3}{2}$.
\end{remark}
In~\cite{rao1979} Rao proposed a problem about the characterization of Hamiltonian connected self-complementary graphs. Lemma~\ref{hamconnected} provides some information in this direction and indicates the difficulty of the problem.
\begin{lemma}\label{hamconnected}
If a self-complementary graph $\Gamma$ on $n>5$ vertices is vertex-transitive or strongly-regular, then $\Gamma$ is Hamiltonian-connected.
\end{lemma}
\begin{proof}
By Lemmas~\ref{clique-coclique},~\ref{sr-1wr},  and self-complementarity we have $\alpha(\Gamma)\leq \lfloor\sqrt{n}\rfloor$.
Since $\Gamma$ is regular and self-complementary it follows from Lemmas~\ref{DAM} that
$\kappa(\Gamma)\geq \frac{1}{2}(\frac{n-1}{2}+1)=\frac{n+1}{4}$. Consequently, $\alpha(\Gamma)<\kappa(\Gamma)$ for $n\geq 12$. Since $n=1\, (\textrm{mod}~4)$, the claim for $n\neq 9$ follows from Lemma~\ref{connectivity_hamilton}. Recall that if $\Gamma$ is isomorphic to $\textrm{Paley}(9)$, then $\kappa(\Gamma)=\frac{9-1}{2}=4$ and the claim follows from Lemma~\ref{connectivity_hamilton} as above. There are three more regular self-complementary graphs on 9 vertices, but none of them is vertex-transitive or strongly regular (i.e. 1-walk regular). To see this we can compute the powers $A_2^4, A_3^3, A_4^3$ of matrices in Figure~\ref{f9} and observe that their $(1,1)$ and $(2,2)$ entries differ.
Hence, the numbers of closed walks of length four/three at vertices~1 and~2 are different.
\end{proof}

\begin{remark}
The statement in Lemma~\ref{hamconnected} can be generalized: Each regular self-complementary graph $\Gamma$ on $n>5$ vertices such that $\alpha(\Gamma)\leq \sqrt{n}$ is Hamiltonian-connected. The proof above confirms this statement for graphs that are not isomorphic to graphs with adjacency matrices $A_2,A_3,A_4$ in Figure~\ref{f9}. The statement for these three graphs was verified by a computer. Actually, it can be checked that graphs with adjacency matrices $A_2$ and $A_4$ both satisfy $3=\alpha(\Gamma)<\kappa(\Gamma)=4$, as do the graph $\textrm{Paley}(9)$. Hence, Hamiltonian-connectivity follows from Lemma~\ref{connectivity_hamilton}. The graph with the adjacency matrix $A_3$ is more complicated, since $\alpha(\Gamma)=3=\kappa(\Gamma)$. In fact, the deletion of the independent set $\{1,6,7\}$ or $\{1,8,9\}$ disconnects this graph.
\end{remark}

\subsection{Graph homomorphisms}

A graph $\Gamma$ is a \emph{core} if $\textrm{End}(\Gamma)=\textrm{Aut}(\Gamma)$. Simple examples include complete graphs, odd cycles, or more generally \emph{vertex-critical} graphs, i.e. graphs, where a removal of any vertex decreases the chromatic number (cf. Lemma~\ref{lemma-chromatic}). A~subgraph $\Gamma'$ in a graph $\Gamma$ is a \emph{core of} $\Gamma$, if $\Gamma'$ is a core and there exists some graph homomorphism from $\Gamma$ to $\Gamma'$. A core of a graph is an induced subgraph and is unique up to isomorphism~\cite[Lemma~6.2.2]{godsil_knjiga}. By $\textrm{core}(\Gamma)$ we denote any of the cores of $\Gamma$. There always exists a \emph{retraction} $\psi: \Gamma\to \textrm{core}(\Gamma)$, i.e. a graph homomorphism that fixes each vertex in $\textrm{core}(\Gamma)$. Namely, if $\varphi: \Gamma\to \textrm{core}(\Gamma)$ is any graph homomorphism, then its restriction to the vertices of $\textrm{core}(\Gamma)$ is a member of $\textrm{Aut}(\textrm{core}(\Gamma))$ and therefore $(\varphi|_{V(\textrm{core}(\Gamma))})^{-1}\circ \varphi$ is a retraction. The core of a graph $\Gamma$ is a complete graph if and only if $\chi(\Gamma)=\omega(\Gamma)$. Lemma~\ref{lemma-clique} is obvious. Lemma~\ref{lemma-chromatic} is also well known (cf.~\cite{handbook_produkti,hahn}).

\begin{lemma}\label{lemma-clique}
Let $\Gamma_1$ and $\Gamma_2$ be two graphs such that there exists a homomorphism from $\Gamma_1$ to $\Gamma_2$. Then $\omega(\Gamma_1)\leq \omega(\Gamma_2)$.
\end{lemma}

\begin{lemma}\label{lemma-chromatic}
Let $\Gamma_1$ and $\Gamma_2$ be two graphs such that there exists a homomorphism from $\Gamma_1$ to $\Gamma_2$. Then $\chi(\Gamma_1)\leq \chi(\Gamma_2)$.
\end{lemma}

In Corollary~\ref{lemma-preimage}, the claim for vertex-transitive graphs follows directly from the proof of \cite[Theorem~6.13.2]{godsil_knjiga}. We provide a short proof that works also for 1-walk regular graphs. Corollary~\ref{lemma-preimage} is applied in Propositions~\ref{prop-lex1} and~\ref{prop-lex2}.
\begin{cor}\label{lemma-preimage}
If a graph $\Gamma$ is vertex-transitive or 1-walk regular, and $\varphi: \Gamma\to K_{\omega(\Gamma)}$ is a graph homomorphism, then $\alpha(\Gamma)\omega(\Gamma)=|V(\Gamma)|$ and $|\varphi^{-1}(v)|=\alpha(\Gamma)$ for each vertex $v$ in $K_{\omega(\Gamma)}$.
\end{cor}
\begin{proof}
By Lemma~\ref{lemma-chromatic}, $\chi(\Gamma)\leq \chi(K_{\omega(\Gamma)})=\omega(\Gamma)$. Hence, $\omega(\Gamma)=\chi(\Gamma)\geq \frac{|V(\Gamma)|}{\alpha(\Gamma)}$. By Lemma~\ref{clique-coclique} it follows that $\alpha(\Gamma)\omega(\Gamma)=|V(\Gamma)|$. Since the preimages
$\varphi^{-1}(v)$, with  $v\in V(K_{\omega(\Gamma)})$, are independent sets in $\Gamma$ that partition $V(\Gamma)$, we have
$$|V(\Gamma)|=\sum_{v\in V(K_{\omega(\Gamma)})} |\varphi^{-1}(v)|\leq  \alpha(\Gamma)\omega(\Gamma)=|V(\Gamma)|.$$
Consequently,  $|\varphi^{-1}(v)|=\alpha(\Gamma)$ for all $v$.
\end{proof}

The Lov\'{a}sz theta function $\vartheta(\Gamma)$ of a graph $\Gamma$ was introduced in~\cite{lovasz1979}. If a graph $\Gamma$ has $n$ vertices and  $\nu<0$ is a real number, then consider the infinite graph $S(n,\nu)$ with the unit sphere $S^{n+1}:=\{{\bf x}\in \mathbb{R}^n : \langle{\bf x}, {\bf x}\rangle_n=1\}$ as the vertex set and the edge set given by $\{\{{\bf x}, {\bf y}\}: {\bf x},{\bf y}\in S^{n+1}, \langle{\bf x}, {\bf y}\rangle_n=\nu\}$. Here, $\langle{\bf x}, {\bf y}\rangle_n=\sum_{i=1}^n x_i y_i$ is the dot product of vectors ${\bf x}=(x_1,\ldots,x_n)$ and ${\bf y}=(y_1,\ldots,y_n)$. In~\cite[Theorem~8.2]{JACM} it was shown that the Lov\'{a}sz theta function of the complement $\vartheta({\bar \Gamma})$ is the same as the infimum among all values $1-\frac{1}{\nu}$, where $\nu$ ranges over all negative values such that there exists a homomorphism $\varphi_{n,\nu}: \Gamma\to  S(n,\nu)$. Consequently it was observed in~\cite{robersonJCTB} that $\vartheta({\overline \Gamma_1})\leq \vartheta({\overline \Gamma_2})$ whenever there exists a homomorphism $\varphi: \Gamma_1\to \Gamma_2$ (see~Lemma~\ref{lemma-theta} below). In fact, if $n_1:=|V(\Gamma_1)|$, $n_2:=|V(\Gamma_2)|$, and $\varphi_{n_2,\nu}$ is a homomorphism between $\Gamma_2$ and $S(n_2,\nu)$, then the image of the map $\varphi_{n_2,\nu}\circ\varphi$ spans a vector subspace $U\subseteq \mathbb{R}^{n_2}$ of dimension $\dim U\leq n_1$. If we choose any (linear) map $g: (U,\langle\cdot,\cdot\rangle_{n_2})\to (\mathbb{R}^{\dim U}, \langle\cdot,\cdot\rangle_{\dim U})$ that preserves the dot product, and the map $f: \mathbb{R}^{\dim U} \to \mathbb{R}^{n_1}$ that extends vectors by $n_1-\dim U$ zero entries, we deduce that the map $\varphi_{n_1,\nu}:=f\circ g\circ \varphi_{n_2,\nu}\circ\varphi$ is a homomorphism between $\Gamma_1$ and $S(n_1,\nu)$.

\begin{lemma}\label{lemma-theta}
Let $\Gamma_1$ and $\Gamma_2$ be two graphs such that there exists a homomorphism from $\Gamma_1$ to $\Gamma_2$. Then $\vartheta({\overline \Gamma_1})\leq \vartheta({\overline \Gamma_2})$.
\end{lemma}

Lemmas~\ref{lemma-thetaproduct}, \ref{lemma-thetabound} were proved by Lov\'{a}sz~\cite[Corollary~2,Theorem~9]{lovasz1979}.
\begin{lemma}\label{lemma-thetaproduct}
Let $\Gamma$ be a graph on $n$ vertices. Then $\vartheta(\Gamma)\vartheta(\bar{\Gamma})\geq n$.
\end{lemma}

\begin{lemma}\label{lemma-thetabound}
Let $\Gamma$ be a $k$-regular graph on $n$ vertices. Then $$\vartheta(\Gamma)\leq \frac{n}{1-\frac{k}{\lambda_{n}}},$$
where $\lambda_{n}$ is the smallest eigenvalue of the adjacency matrix of $\Gamma$.
\end{lemma}

The above results about $\vartheta$ are applied in the proof of Theorem~\ref{thm-strg}.

\begin{remark}
If we apply Lemmas~\ref{lemma-thetaproduct} and \ref{lemma-thetabound} at a regular self-complementary graph, we deduce that its second eigenvalue $\lambda_2=-1-\lambda_n$ satisfies $\lambda_2\geq\frac{\sqrt{n}-1}{2}$. By Lemma~\ref{strg}, strongly regular self-complementary graphs satisfy the equality.
\end{remark}

The following result can be found in \cite[Lemma~6.2.3]{godsil_knjiga}. It is an important tool in Section~\ref{cores}.

\begin{lemma}\label{lemma-izomorfnostjeder}
Let $\Gamma_1$ and $\Gamma_2$ be two graphs. Then there exist homomorphisms from $\Gamma_1$ to $\Gamma_2$ and from $\Gamma_2$ to $\Gamma_1$ if and only if $\textrm{core}(\Gamma_1)\cong \textrm{core}(\Gamma_2)$.
\end{lemma}
\begin{remark}\label{remark}
In particular, if $\psi$ is a retraction from $\Gamma$ onto its image, then $\Gamma$  and the graph induced by the image of $\psi$ have isomorphic cores.
\end{remark}

Lemma~\ref{lemma-core-povezanost} is well known but we did not find it in the literature.
\begin{lemma}\label{lemma-core-povezanost}
If graph $\Gamma$ is connected, then $\textrm{core}(\Gamma)$ is connected too.
\end{lemma}
\begin{proof}
Let $u$ and $v$ be arbitrary vertices in $\textrm{core}(\Gamma)$. Since $\textrm{core}(\Gamma)$ is a subgraph in the connected graph $\Gamma$, we can find a walk $$u=u_0\sim_{\Gamma} u_1\sim_{\Gamma}\cdots\sim_{\Gamma}u_{d-1}\sim_{\Gamma} u_d=v$$ in $\Gamma$ that joins $u$ and $v$. If $\psi$ is any retraction of $\Gamma$ onto $\textrm{core}(\Gamma)$, then $$u=u_0\sim_{\Gamma} \psi(u_1)\sim_{\Gamma}\cdots \sim_{\Gamma}\psi(u_{d-1})\sim_{\Gamma} u_d=v$$ is a walk in $\textrm{core}(\Gamma)$ that connects $u$ and $v$.
\end{proof}

Lemma~\ref{lemma-core-vt} is proved in~\cite[Theorem~3.2]{welzl}, where it is stated in an old terminology. Its proofs can be found also in~\cite{hahn,godsil_knjiga}.
\begin{lemma}\label{lemma-core-vt}
If graph $\Gamma$ is vertex-transitive, then $\textrm{core}(\Gamma)$ is vertex-transitive.
\end{lemma}
\begin{remark}
An analogous claim, where vertex-transitivity in Lemma~\ref{lemma-core-vt} is replaced by regularity, is not true. Consider for example the regular graph~$\Gamma$ in Figure~\ref{f11}.
\begin{figure}[h!]
\centering
\includegraphics[width=0.4\textwidth]{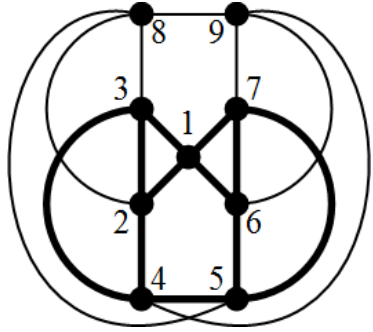}
\caption{A regular graph with a non-regular core.}
\label{f11}
\end{figure}
The map $\psi$ defined on the vertex set $\{1,\ldots,9\}$ by $\psi(8)=4$, $\psi(9)=5$, and $\psi(i)=i$ for $i\leq 7$, is a retraction onto the bold part of the graph. Since the bold part is a vertex-critical graph, it follows that it is a core of $\Gamma$.
\end{remark}

\subsection{Lexicographic product}

The results listed or proved in this subsection are applied exclusively in the last two results of this paper, that is, in Propositions~\ref{prop-lex1} and~\ref{prop-lex2}. The notion of a lexicographic product appears also in Proposition~\ref{p7}.

The lexicographic product of graphs $\Gamma_1$ and $\Gamma_2$ is the graph $\Gamma_1[\Gamma_2]$ with the vertex set $V(\Gamma_1)\times V(\Gamma_2)$, where $(u_1,u_2)\sim (v_1,v_2)$ if and only if either $u_1\sim_{\Gamma_1} v_1$ or $u_1=v_1$ and $u_2\sim_{\Gamma_2} v_2$. Consequently, $\overline{\Gamma_1[\Gamma_2]}=\bar{\Gamma}_1[\bar{\Gamma}_2]$. In particular, $\Gamma_1[\Gamma_2]$ is self-complementary whenever $\Gamma_1$ and $\Gamma_2$ are self-complementary. Analogous claim is true for vertex-transitivity (cf.~\cite[Theorem~10.14]{handbook_produkti}).

If for $i=1,2$ graphs $\Gamma_i$ and $\Gamma_i'$ are isomorphic, then we define
\begin{equation}\label{e99}
\textrm{Iso}(\Gamma_1,\Gamma_1')\wr \textrm{Iso}(\Gamma_2,\Gamma_2')
\end{equation}
as the set of all maps $(\varphi, \beta) : \Gamma_1[\Gamma_2]\to \Gamma_1'[\Gamma_2']$ that are of the form
\begin{equation}\label{e98}
(\varphi, \beta)(v_1,v_2):=\big(\varphi(v_1), \beta(v_1)(v_2)\big)
\end{equation}
for all $v_1\in V(\Gamma_1), v_2\in V(\Gamma_2)$, where $\varphi\in \textrm{Iso}(\Gamma_1,\Gamma_1')$ and $\beta: V(\Gamma_1)\to \textrm{Iso}(\Gamma_2,\Gamma_2')$ is a map. Clearly, maps~\eqref{e98} are isomorphisms and \eqref{e99} is a subset in $\textrm{Iso}(\Gamma_1[\Gamma_2], \Gamma_1'[\Gamma_2'])$. If $\Gamma_i'=\Gamma_i$ for $i=1,2$ or  $\Gamma_i'=\bar{\Gamma}_i$ for $i=1,2$, then we write
$\textrm{Aut}(\Gamma_1)\wr \textrm{Aut}(\Gamma_2)$ or $\overline{\textrm{Aut}(\Gamma_1)}\wr \overline{\textrm{Aut}(\Gamma_2)}$ instead of~\eqref{e99}, respectively.

Given a graph $\Gamma$ let $R_{\Gamma}=\{(u,v)\in V(\Gamma)\times V(\Gamma) : N_{\Gamma}(u)=N_{\Gamma}(v)\}$, $S_{\Gamma}=\{(u,v)\in V(\Gamma)\times V(\Gamma): N_{\Gamma}[u]=N_{\Gamma}[v]\}$, and $\triangle_{\Gamma}=\{(u,u)\in V(\Gamma)\times V(\Gamma) : u\in V(\Gamma)\}$. Sabidussi~\cite{sabidussi1959} proved the following result (see also~\cite[Theorem~10.13]{handbook_produkti}).

\begin{lemma}\label{lex-auto}
 For any graphs $\Gamma_1, \Gamma_2$ we have $\textrm{Aut}\big(\Gamma_1[\Gamma_2]\big)=\textrm{Aut}(\Gamma_1)\wr \textrm{Aut}(\Gamma_2)$ if and only if the following two assertions hold:
\begin{enumerate}
\item If $R_{\Gamma_1}\neq \triangle_{\Gamma_1}$, then $\Gamma_2$ is connected.
\item If $S_{\Gamma_1}\neq \triangle_{\Gamma_1}$, then $\bar{\Gamma}_2$ is connected.
\end{enumerate}
\end{lemma}

\begin{cor}\label{c6}
Suppose that $\Gamma_i\cong \Gamma_i'$ for $i=1,2$. If $\Gamma_2$ and $\bar{\Gamma}_2$ are both connected, then $\textrm{Iso}\big(\Gamma_1[\Gamma_2], \Gamma_1'[\Gamma_2']\big)=\textrm{Iso}(\Gamma_1,\Gamma_1')\wr \textrm{Iso}(\Gamma_2,\Gamma_2')$.
\end{cor}
\begin{proof}
Let $\Phi\in \textrm{Iso}\big(\Gamma_1[\Gamma_2], \Gamma_1'[\Gamma_2']\big)$. Pick any $\psi_1\in \textrm{Iso}(\Gamma_1,\Gamma_1')$, $\psi_2\in \textrm{Iso}(\Gamma_2,\Gamma_2')$, and define $\Psi \in \textrm{Iso}\big(\Gamma_1[\Gamma_2], \Gamma_1'[\Gamma_2']\big)$ by $\Psi(v_1,v_2)=(\psi_1(v_1),\psi_2(v_2))$ for all $v_1\in V(\Gamma_1)$ and $v_2\in V(\Gamma_2)$. Then $\Psi^{-1}\circ\Phi\in \textrm{Aut}\big(\Gamma_1[\Gamma_2]\big)$. By Lemma~\ref{lex-auto}, $\Psi^{-1}\circ\Phi=(\varphi,\beta)$ for some $\varphi\in \textrm{Aut}(\Gamma_1)$ and a map $\beta: V(\Gamma_1)\to \textrm{Aut}(\Gamma_2)$. Consequently, $\Phi=(\psi_1\circ\varphi,\gamma)$, where $\gamma : V(\Gamma_1)\to \textrm{Iso}(\Gamma_2,\Gamma_2')$ is defined by $\gamma(v_1)=\psi_2\circ \beta(v_1)$. Hence, $\Phi\in \textrm{Iso}(\Gamma_1,\Gamma_1')\wr \textrm{Iso}(\Gamma_2,\Gamma_2')$.
\end{proof}

Since self-complementary graphs are connected, we deduce the following.
\begin{cor}\label{c7}
Let $\Gamma_1$ and $\Gamma_2$ be self-complementary graphs. Then
\begin{enumerate}
\item $\textrm{Aut}\big(\Gamma_1[\Gamma_2]\big)=\textrm{Aut}(\Gamma_1)\wr \textrm{Aut}(\Gamma_2)$,
\item $\overline{\textrm{Aut}\big(\Gamma_1[\Gamma_2]\big)}=\overline{\textrm{Aut}(\Gamma_1)}\wr \overline{\textrm{Aut}(\Gamma_2)}$.
\end{enumerate}
\end{cor}

Similarly as above, given graphs $\Gamma_1, \Gamma_1', \Gamma_2,\Gamma_2'$ with nonempty sets $\textrm{Hom}(\Gamma_1,\Gamma_1')$ and $\textrm{Hom}(\Gamma_2,\Gamma_2')$, we define
\begin{equation}\label{e100}
\textrm{Hom}(\Gamma_1,\Gamma_1')\wr \textrm{Hom}(\Gamma_2,\Gamma_2')
\end{equation}
as the set of all homomorphisms $(\varphi, \beta) : \Gamma_1[\Gamma_2]\to \Gamma_1'[\Gamma_2']$ that are of the form~\eqref{e98}, where $\varphi\in \textrm{Hom}(\Gamma_1,\Gamma_1')$ and $\beta: V(\Gamma_1)\to \textrm{Hom}(\Gamma_2,\Gamma_2')$ is a map. If $\Gamma_i'=\Gamma_i$ for $i=1,2$, we write $\textrm{End}(\Gamma_1)\wr \textrm{End}(\Gamma_2)$ instead of~\eqref{e100} (cf. \cite{knauer,kaschek2009,kaschek2010}).

It is obvious that $\Gamma_1$ and $\Gamma_2$ are cores whenever $\Gamma_1[\Gamma_2]$ is a core. The converse is not true in general.
The following result is proved in~\cite[Theorem~3.11]{knauer}.
\begin{lemma}\label{lemma-knauer}
Let $n$ be a positive integer and $\Gamma_2$ a graph. Then $K_n[\Gamma_2]$ is a core if and only if $\Gamma_2$ is a core.
\end{lemma}
Lemma~\ref{lemma-kaschek2010} is proved in~\cite[Theorem~14]{kaschek2010}.
\begin{lemma}\label{lemma-kaschek2010}
Let  $\Gamma_1$ and $\Gamma_2$ be two graphs, where $\Gamma_1$ is a core, and let $\textrm{core}(\Gamma_1[\Gamma_2])$ be any core of $\Gamma_1[\Gamma_2]$. Then $\textrm{End}\big(\Gamma_1[\Gamma_2]\big)=\textrm{End}(\Gamma_1)\wr \textrm{End}(\Gamma_2)$ if and only if the following assertions are true:
\begin{enumerateB}
\item $\textrm{core}(\Gamma_1[\Gamma_2])=\Gamma_1[\textrm{core}(\Gamma_2)]$, where $\textrm{core}(\Gamma_2)$ is some core of $\Gamma_2$.
\item $S_{\Gamma_1}=\triangle_{\Gamma_1}$ or $\overline{\textrm{core}(\Gamma_2)}$ is connected.
\end{enumerateB}
\end{lemma}

\begin{cor}\label{c8}
Let  $\Gamma_1$ and $\Gamma_2$ be two graphs, where $\Gamma_1[\textrm{core}(\Gamma_2)]$ is a core. If $\Gamma_1$ is self-complementary, then $\textrm{End}\big(\Gamma_1[\Gamma_2]\big)=\textrm{End}(\Gamma_1)\wr \textrm{End}(\Gamma_2)$.
\end{cor}
\begin{proof}
Obviously there exist a homomorphism from $\Gamma_1[\Gamma_2]$ onto $\Gamma_1[\textrm{core}(\Gamma_2)]$. Hence, (i) from Lemma~\ref{lemma-kaschek2010} is satisfied. Since $\Gamma_1[\textrm{core}(\Gamma_2)]$ is a core, the same is true for $\Gamma_1$. Consequently, to end the proof it suffices to prove that $S_{\Gamma_1}=\triangle_{\Gamma_1}$. Suppose the contrary, that is, there are distinct
$u,v\in V(\Gamma_1)$ such that $N_{\Gamma_1}[u]=N_{\Gamma_1}[v]$. Clearly, $u$ and $v$ are adjacent in $\Gamma_1$. Consequently,
$N_{\bar{\Gamma}_1}(u)=N_{\bar{\Gamma}_1}(v)$ and $u,v$ are nonadjacent in $\bar{\Gamma}_1$. The map $\varphi$ on $V(\bar{\Gamma}_1)$ that maps $u$ to $v$ and fixes all other vertices is a nonbijective endomorphism of $\bar{\Gamma}_1$. Since $\bar{\Gamma}_1$ is a core by self-complementarity, we get a contradiction.
\end{proof}

\begin{lemma}\label{zadnja-lex}
Let $\Gamma_1, \Gamma_2$ be graphs, where $\Gamma_1$ vertex-transitive, while $\Gamma_2, \bar{\Gamma}_2$ are both connected. If $\textrm{core}(\Gamma_1)$ is any core of $\Gamma_1$ and $\textrm{core}(\Gamma_1)[\Gamma_2]$ is a core, then
$$\textrm{Hom}\big(\Gamma_1[\Gamma_2],\textrm{core}(\Gamma_1)[\Gamma_2]\big)=\textrm{Hom}\big(\Gamma_1,core(\Gamma_1)\big)\wr \textrm{Hom}(\Gamma_2,\Gamma_2).$$
\end{lemma}
\begin{proof}
Let $\Phi\in \textrm{Hom}\big(\Gamma_1[\Gamma_2],\textrm{core}(\Gamma_1)[\Gamma_2]\big)$ and $v_1\in V(\Gamma_1)$. Since $\Gamma_1$ is vertex-transitive, there exists a subgraph $\Gamma$ in $\Gamma_1$, which is isomorphic to $\textrm{core}(\Gamma_1)$ and such that $v_1\in V(\Gamma)$. The restriction $\Phi|_{V(\Gamma[\Gamma_2])}$ is a homomorphism from $\Gamma[\Gamma_2]$ to $\textrm{core}(\Gamma_1)[\Gamma_2]$. Since $\textrm{core}(\Gamma_1)[\Gamma_2]$ is a core, we deduce that $\Phi|_{V(\Gamma[\Gamma_2])}$ is an isomorphism. By Corollary~\ref{c6}, $\Phi|_{V(\Gamma[\Gamma_2])}=(\varphi,\beta)$ for some $\varphi\in \textrm{Iso}(\Gamma,\textrm{core}(\Gamma_1))$ and a map $\beta: V(\Gamma)\to \textrm{Iso}(\Gamma_2,\Gamma_2)=\textrm{Aut}(\Gamma_2)$. Since $v_1\in V(\Gamma_1)$ is arbitrary, we deduce in particular that
$$\Phi(v_1,v_2)=(\psi(v_1), \beta_{v_1}(v_2))$$
for all $v_1\in V(\Gamma_1)$ and $v_2\in V(\Gamma_2)$, where $\psi: V(\Gamma_1)\to V\big(\textrm{core}(\Gamma_1)\big)$ is some map and $\beta_{v_1}\in \textrm{Aut}(\Gamma_2)$.

We claim that $\psi$ is a graph homomorphism. Let $v_1\sim_{\Gamma_1} v_1'$. Then
\begin{equation}\label{e101}
(\psi(v_1), \beta_{v_1}(v_2))=\Phi(v_1,v_2)\sim\Phi(v_1',v_2')=(\psi(v_1'), \beta_{v_1'}(v_2'))
\end{equation}
for all $v_2,v_2'\in V(\Gamma_2)$. Since $\beta_{v_1}, \beta_{v_1'}$ are automorphisms, we can find $v_2,v_2'$ such that $\beta_{v_1}(v_2)=\beta_{v_1'}(v_2')$. It follows from~\eqref{e101} that $\psi(v_1)\sim_{\Gamma_1} \psi(v_1')$ and $\psi$ is a graph homomorphism. Hence, $\Phi=(\psi,\beta)$, where the map $\beta: V(\Gamma_1)\to \textrm{Aut}(\Gamma_2)$ is defined by $\beta(v_1):=\beta_{v_1}$.
Therefore, $\Phi\in \textrm{Hom}\big(\Gamma_1,core(\Gamma_1)\big)\wr \textrm{Hom}(\Gamma_2,\Gamma_2)$.
\end{proof}
We remark that since $\textrm{core}(\Gamma_1)[\Gamma_2]$ is a core in Lemma~\ref{zadnja-lex}, then $\Gamma_2$ is also a core. Consequently, $\textrm{Hom}(\Gamma_2,\Gamma_2)=\textrm{Aut}(\Gamma_2)$.

\subsection{Complementary prism}

Let $\Gamma$ be a graph with the vertex set $V(\Gamma)=\{v_1,\ldots,v_n\}$ and the edge set $E(\Gamma)$. The \emph{complementary prism} of $\Gamma$ is the graph $\Gamma\bar{\Gamma}$, which is obtained by the disjoint union of $\Gamma$ and its complement $\bar{\Gamma}$, if we connect each vertex in $\Gamma$ to its copy in $\bar{\Gamma}$. More precisely, $V(\Gamma\bar{\Gamma}) = W_1\cup W_2$, where
$$W_1=W_1(\Gamma\bar{\Gamma})=\{(v_1,1),\ldots,(v_n,1)\}\ \textrm{and}\ W_2=W_2(\Gamma\bar{\Gamma})=\{(v_1,2),\ldots,(v_n,2)\},$$
and the edge set $E(\Gamma\bar{\Gamma})$ is given by the union
\begin{align*}
&\Big\{\{(u,1),(v,1)\}: \{u,v\}\in E(\Gamma)\Big\}\\
\cup&\Big\{\{(u,2),(v,2)\}: \{u,v\}\in E(\bar{\Gamma})\Big\}\\
\cup&\Big\{\{(u,1),(u,2)\}: u\in V(\Gamma)\Big\}.
\end{align*}
In particular, $|V(\Gamma\bar{\Gamma})|=2n$ and $|E(\Gamma\bar{\Gamma})|=\binom{n}{2}+n=\binom{n+1}{2}$. The complementary prism generalize the Petersen graph, which is obtained if $\Gamma$ is the pentagon (see Figure~\ref{f1}).
\begin{figure}[h!]
\centering
\includegraphics[width=0.8\textwidth]{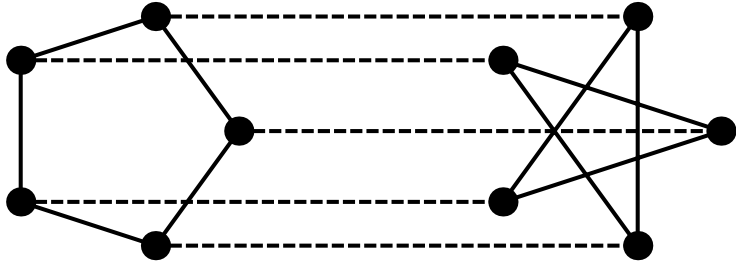}
\caption{The complementary prism $C_5\overline{C_5}$ is the Petersen graph.}
\label{f1}
\end{figure}
We remark that the family of all complementary prisms has a large intersection with the family of all bi-Cayley graphs over abelian groups (cf.~\cite{stefko}).

The following result about the spectrum of~$\Gamma\bar{\Gamma}$ is proved in~\cite[Corollary~3.4]{carvalho2018}.
\begin{lemma}\label{carvalho}
Let $\Gamma$ be a connected $k$-regular graph on $n$ vertices, where $k=\lambda_1\geq \lambda_2\geq\dots\geq \lambda_n$ are the eigenvalues of the adjacency matrix of $\Gamma$. Then the set of all eigenvalues of the adjacency matrix of $\Gamma\bar{\Gamma}$ equals
\begin{align}
\label{e5}&\left\{\frac{n-1\pm \sqrt{(n-1-2k)^2+4}}{2}\right\}\\
\label{e6}&\qquad\qquad\qquad\cup\left\{\frac{-1\pm \sqrt{(2\lambda_i+1)^2+4}}{2} : 2\leq i\leq n\right\}.
\end{align}
\end{lemma}

We apply Corollary~\ref{t1} in the proof of Theorem~\ref{thm-strg}.
\begin{cor}\label{t1}
If $\Gamma$ is as in Lemma~\ref{carvalho}, then the largest and the smallest eigenvalue of $\Gamma \bar{\Gamma}$ equal
\begin{equation}\label{e7}
\frac{n-1+\sqrt{(n-1-2k)^2+4}}{2}
\end{equation}
and
\begin{equation}\label{e8}
\min\left\{\frac{-1-\sqrt{(2\lambda_2+1)^2+4}}{2},\frac{-1-\sqrt{(2\lambda_n+1)^2+4}}{2}\right\},
\end{equation}
respectively.
\end{cor}

\begin{proof}
Obviously, \eqref{e8} is the smallest among all values in \eqref{e5}-\eqref{e6}. To see that \eqref{e7} is the largest value, it suffices to prove that
$$n+\sqrt{(n-1-2k)^2+4}\geq \sqrt{(2\lambda_i+1)^2+4}$$
for all $i\geq 2$. The left-hand side equals $n+\sqrt{m+n^2-2n-4nk}$, where $m=(2k+1)^2+4$, while the right-hand side is bounded above by $\sqrt{m}$. Therefore it suffices to prove that $n+\sqrt{m+n^2-2n-4nk}\geq \sqrt{m}$, which is obvious if $n^2-2n-4nk\geq 0$. On the other hand, if $n^2-2n-4nk<0$, then it suffices to prove the inequality
\begin{align*}
n^2&\geq \big(\sqrt{m}-\sqrt{m+n^2-2n-4nk}\big)^2\\
&=2m+n^2-2n-4nk-2\sqrt{m(m+n^2-2n-4nk)},
\end{align*}
which is equivalent to
$$\sqrt{m(m+n^2-2n-4nk)}\geq m-n-2nk.$$
If the right-hand side is negative, then the inequality is obviously true. Otherwise we can square the inequality and rearrange it to deduce
$n^2(m-4k^2-4k-1)\geq 0$, which is true since $m=4k^2+4k+5$.
\end{proof}

\begin{remark}\label{opomba2}
Let $\Gamma$ be a regular self-complementary graph on $n$ vertices. Then it follows from Corollary~\ref{t1} that the largest and the smallest eigenvalue of $\Gamma\bar{\Gamma}$ are $\frac{n+1}{2}$ and $\frac{-1-\sqrt{(2\lambda_2+1)^2+4}}{2}$, respectively, where $\lambda_2$ is the second eigenvalue of~$\Gamma$. Recall that the second eigenvalue of a graph is an important invariant (cf.~\cite{alon,friedman,brouwer-haemers}).
Lemma~\ref{carvalho} implies that the second eigenvalue of $\Gamma\bar{\Gamma}$ equals $\max\{\frac{n-3}{2}, \frac{-1+\sqrt{(2\lambda_2+1)^2+4}}{2}\}$ if $n>1$. Observe that $\frac{-1+\sqrt{(2\lambda_2+1)^2+4}}{2}>\frac{n-3}{2}$ if and only if $\lambda_2>\frac{\sqrt{n(n-4)}-1}{2}$, which is a value that is slightly smaller than the upper bound $\frac{n-7}{2}-2\cos(\frac{\pi(n-1)}{n})$ for $\lambda_2$ from Remark~\ref{opomba} (the two values are the same if $\Gamma\cong C_5$). It would be interesting to know if there exists a regular/vertex-transitive self-complementary graph $\Gamma$ on $n$ vertices with the second eigenvalue in the bounds $\frac{\sqrt{n(n-4)}-1}{2}<\lambda_2\leq \frac{n-7}{2}-2\cos(\frac{\pi(n-1)}{n})$. If there are no such vertex-transitive self-complementary graphs, then the second eigenvalue of a complementary prism is $\frac{n-3}{2}$ whenever it is vertex-transitive (cf. Corollary~\ref{c1}).
\end{remark}

Lemma~\ref{p6} is proved in~\cite[Theorem~2]{haynes2007}.
\begin{lemma}\label{p6}
Let $\Gamma$ be any graph. Then $\Gamma\bar{\Gamma}$ is a connected graph of diameter at most three. Moreover,
\begin{align*}
\textrm{diam}(\Gamma\bar{\Gamma})=1 &\Longleftrightarrow \Gamma\cong K_1,\\
\textrm{diam}(\Gamma\bar{\Gamma})=2 &\Longleftrightarrow \textrm{diam}(\Gamma)=2=\textrm{diam}(\bar{\Gamma}).
\end{align*}
\end{lemma}

\section{The automorphism group of a complementary prism and the relation to self-complementary and non-Cayley vertex-transitive graphs}\label{section-auto}

In this section we determine the automorphism group of $\Gamma\bar{\Gamma}$ for arbitrary finite simple graph $\Gamma$ (see Theorem~\ref{p3} and Propositions~\ref{p4},~\ref{p5}). As an application we show that $\Gamma\bar{\Gamma}$ is vertex-transitive if and only if $\Gamma$ is vertex-transitive and self-complementary (Corollary~\ref{c1}). Moreover, $\Gamma\bar{\Gamma}$ is not a Cayley graph whenever $\Gamma$ has more than one vertex (Corollary~\ref{c2}).

In order to determine the automorphism group of a complementary prism, we need to consider two subfamilies of graphs separately. Given a (possibly null) graph $\Lambda$ we define $C_5(\Lambda)$ as the graph, which is formed by the pentagon~$C_5$ if we replace one vertex by $\Lambda$ in such way, that the two neighbors of the vertex in $C_5$ are connected to all vertices in $\Lambda$. Further, let \emph{$A$-graph} be the graph on five vertices, where three vertices form a triangle and exactly two of them have another neighbor, which is of degree one. Given a (possibly null) graph $\Lambda$ we define $A(\Lambda)$ as the graph, which is formed by the $A$-graph if we replace the vertex of degree two by $\Lambda$ in such way, that the two neighbors of the vertex in $A$ are connected to all vertices in $\Lambda$ (see Figure~\ref{f2}).
\begin{figure}
\centering
\begin{tabular}{ccc}
\includegraphics[width=0.4\textwidth]{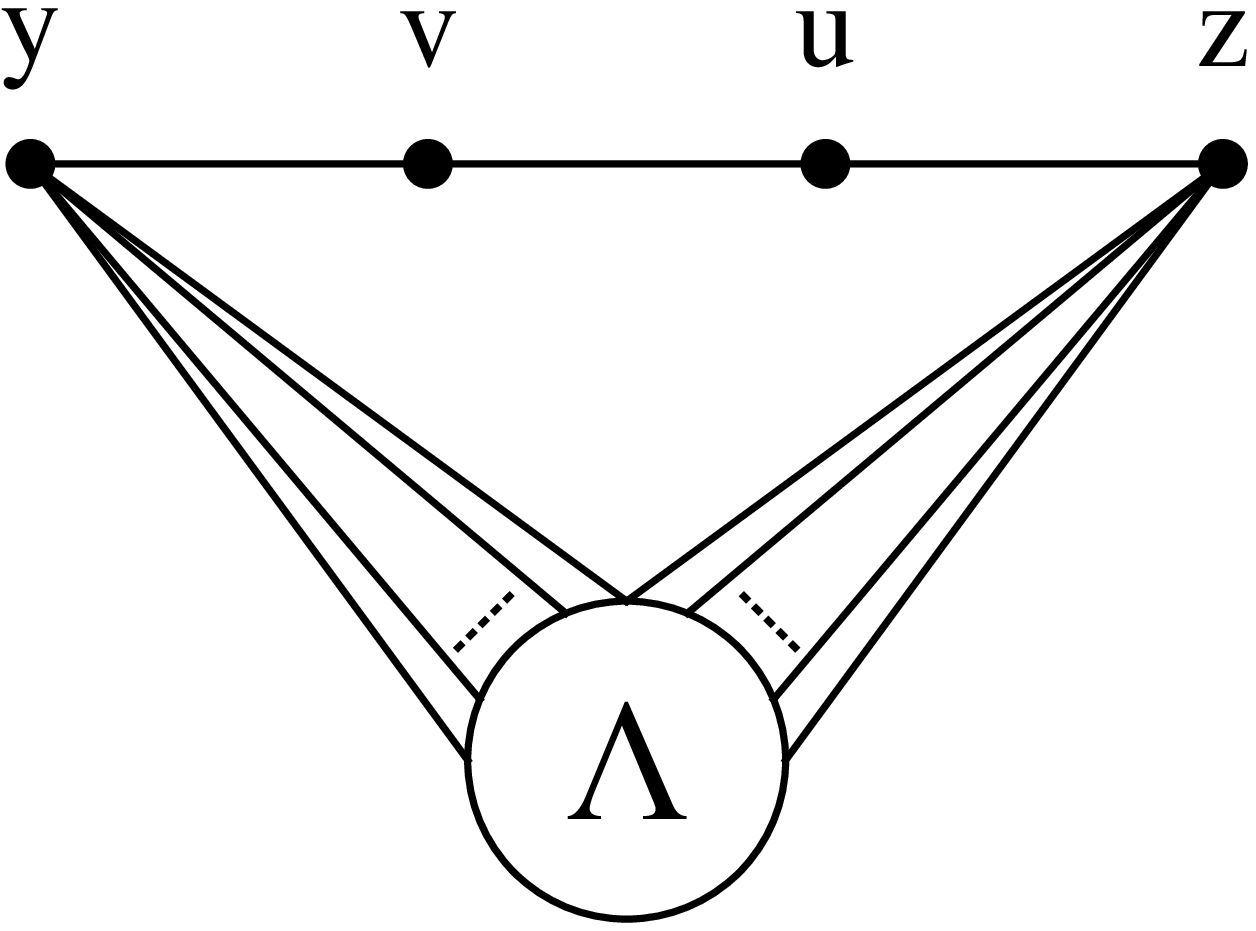}& &
\includegraphics[width=0.4\textwidth]{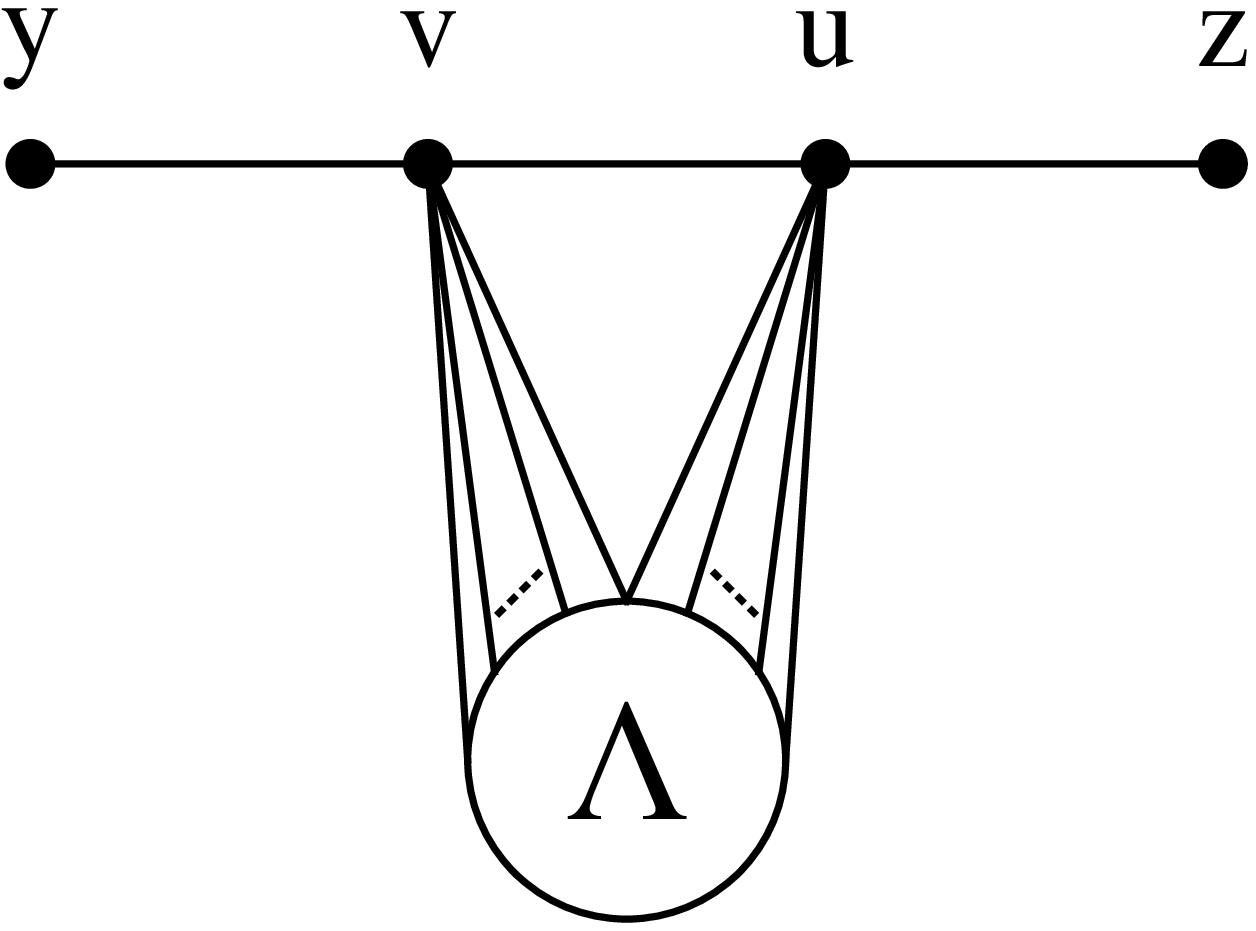}
\end{tabular}
\caption{Graphs $C_5(\Lambda)$ and $A(\Lambda)$ are on the left and on the right, respectively.}
\label{f2}
\end{figure}
Observe that the complements $\overline{C_5(\Lambda)}$ and $\overline{A(\Lambda)}$ are isomorphic to $C_5(\bar{\Lambda})$ and $A(\bar{\Lambda})$, respectively. In particular, $C_5(\Lambda)$ and $A(\Lambda)$ are self-complementary whenever $\Lambda$ is self-complementary. This follows also from the result in~\cite{ruiz} (see also \cite{sc-survey}) and the fact that $C_5$ and $A$-graph are self-complementary. Here we emphasize that graphs $C_5(\Lambda)$, $A(\Lambda)$ are self-complementary even in the degenerate case, where $\Lambda$ is the null graph. Then both $C_5(\Lambda)$ and $A(\Lambda)$ are isomorphic to the path $P_4$. In fact, the degrees of the vertices imply that $C_5(\Lambda)$ and $A(\Lambda')$ are isomorphic if and only if $V(\Lambda)=\emptyset=V(\Lambda')$. Besides, $\textrm{diam}(A(\Lambda))=3$ while $\textrm{diam}(C_5(\Lambda))=2$ unless $\Lambda=\emptyset$.

In the following we temporarily exclude the case $V(\Lambda)=\emptyset$ and the graph $C_5(\Lambda)$ with $|V(\Lambda)|=1$. Then, we can denote the four vertices in $C_5(\Lambda)$ (and in $A(\Lambda)$) that are outside $\Lambda$ by $y,v,u,z$, where $y\sim v\sim u\sim z $ and $\delta(v)=\delta(u)\neq \delta(y)=\delta(z)$. Any automorphism $\varphi\in \textrm{Aut}(\Lambda)$ can be extended to an automorphism of $C_5(\Lambda)$ (or $A(\Lambda)$) in two ways. The first automorphism
\begin{equation}\label{e24}
\widehat{\varphi}_1\ \textrm{fixes all four vertices}\ y,v,u,z.
\end{equation}
The second automorphism $\widehat{\varphi}_2$ satisfies
\begin{equation}\label{e25}
\widehat{\varphi}_2(y)=z,\ \widehat{\varphi}_2(v)=u,\ \widehat{\varphi}_2(u)=v,\  \widehat{\varphi}_2(z)=y.
\end{equation}
Next observe that in the graph $C_5(\Lambda)$ only the vertices in $\Lambda$ have the property that none of their neighbors is of degree two. Similarly, in the graph $A(\Lambda)$ only the vertices in $\Lambda$ have the property that their degree is at least two and their neighbors are not of degree one. Consequently, if $\widehat{\varphi}$ is any automorphism of $C_5(\Lambda)$ (or $A(\Lambda)$), then its restriction to $V(\Lambda)$ is a member of $\textrm{Aut}(\Lambda)$. Formally speaking, if we exclude the case $\Lambda=\emptyset$ and the graph $C_5(\Lambda)$ with $|V(\Lambda)|=1$, then $\textrm{Aut}\big(C_5(\Lambda)\big)$ and $\textrm{Aut}\big(A(\Lambda)\big)$ are both isomorphic to the direct product $\textrm{Aut}(\Lambda)\times \mathbb{Z}_2$, via the map
$$\widehat{\varphi}\mapsto \left\{\begin{array}{ll}(\widehat{\varphi}|_{V(\Lambda)},0) & \textrm{if}\ \widehat{\varphi}\ \textrm{fixes}\ u,v,y,z,\\
(\widehat{\varphi}|_{V(\Lambda)},1) & \textrm{otherwise}.\end{array}\right.$$
If $\Lambda$ is self-complementary, the we deduce in the same way that any antimorphism $\sigma\in \overline{\textrm{Aut}(\Lambda)}$ can be extended to an antimorphism of $C_5(\Lambda)$ (or $A(\Lambda)$) in two ways $\widehat{\sigma}_1$ and $\widehat{\sigma}_2$, by
defining
\begin{equation}\label{e26}
\widehat{\sigma}_1(y)=v,\ \widehat{\sigma}_1(v)=z,\ \widehat{\sigma}_1(u)=y,\ \widehat{\sigma}_1(z)=u
\end{equation}
and
\begin{equation}\label{e27}
\widehat{\sigma}_2(y)=u,\ \widehat{\sigma}_2(v)=y,\ \widehat{\sigma}_2(u)=z,\ \widehat{\sigma}_2(z)=v.
\end{equation}
Conversely, if $\widehat{\sigma}$ is any antimorphism of $C_5(\Lambda)$ (or $A(\Lambda)$), then its restriction to $V(\Lambda)$ is a member of $\overline{\textrm{Aut}(\Lambda)}$. Finally, if $\Gamma=C_5(\Lambda)$, define the map $s_{C_5}: V(\Gamma\bar{\Gamma})\to V(\Gamma\bar{\Gamma})$ by
\begin{align}
\label{e19} s_{C_5}(g,i)&=(g,i)\qquad \big(g\in V(\Lambda), i\in\{1,2\}\big),\\
\nonumber s_{C_5}(y,1)&=(y,1),\ s_{C_5}(v,1)=(y,2),\ s_{C_5}(u,1)=(z,2),\ s_{C_5}(z,1)=(z,1),\\
\nonumber s_{C_5}(y,2)&=(v,1),\ s_{C_5}(v,2)=(u,2),\ s_{C_5}(u,2)=(v,2),\ s_{C_5}(z,2)=(u,1).
\end{align}
It easy to see that $s_{C_5}$ is in fact an automorphism of $\Gamma\bar{\Gamma}$ (cf. Figure~\ref{f3}). Similarly, if $\Gamma=A(\Lambda)$, then the map $s_A: V(\Gamma\bar{\Gamma})\to V(\Gamma\bar{\Gamma})$, defined by,
\begin{align}
\label{e20} s_A(g,i)&=(g,i)\qquad \big(g\in V(\Lambda), i\in\{1,2\}\big),\\
\nonumber s_A(y,1)&=(v,2),\ s_A(v,1)=(v,1),\ s_A(u,1)=(u,1),\ s_A(z,1)=(u,2),\\
\nonumber s_A(y,2)&=(z,2),\ s_A(v,2)=(y,1),\ s_A(u,2)=(z,1),\ s_A(z,2)=(y,2).
\end{align}
is an automorphism of $\Gamma\bar{\Gamma}$ (cf. Figure~\ref{f4}). Moreover, by ignoring \eqref{e19}-\eqref{e20} we get a well defined automorphism of $\Gamma\bar{\Gamma}$ also in the case $V(\Lambda)=\emptyset$. Observe also that the maps \eqref{e24}-\eqref{e25} and \eqref{e26}-\eqref{e27} are the only two automorphisms and antimorphisms, respectively, of the graph $C_5(\emptyset)\cong A(\emptyset)\cong P_4$.

\begin{prop}\label{p4}
Let $\Lambda$ be a (possibly null) graph and let $\Gamma=C_5(\Lambda)$.
\begin{enumerate}
\item If $|V(\Lambda)|=1$, i.e. $\Gamma\bar{\Gamma}$ is the Petersen graph, then the automorphism group $\textrm{Aut}(\Gamma\bar{\Gamma})$ is isomorphic to $S_5$, the symmetric group on five elements.
\item If $\Lambda$ is a (possibly null) self-complementary graph and $|V(\Lambda)|\neq 1$, then the automorphisms of $\Gamma\bar{\Gamma}$ are precisely the maps of the forms
\begin{align}
\label{e28} \Phi(x,i)&=(\widehat{\varphi}(x),i),\\
\nonumber \Phi(x,i)&=(\widehat{\sigma}(x),\bar{i}),\\
\label{e29} \Phi(x,i)&=s_{C_5}(\widehat{\varphi}(x),i),\\
\nonumber \Phi(x,i)&=s_{C_5}(\widehat{\sigma}(x),\bar{i}),
\end{align}
for all $x\in V(\Gamma)$ and $i\in \{1,2\}$, where $\bar{i}\in \{1,2\}\backslash\{i\}$, $\widehat{\varphi}\in \textrm{Aut}(\Gamma)$, $\widehat{\sigma}\in \overline{\textrm{Aut}(\Gamma)}$. The automorphism group $\textrm{Aut}(\Gamma\bar{\Gamma})$ is isomorphic to the semidirect product $$\big(\textrm{Aut}(\Gamma)\cup \overline{\textrm{Aut}(\Gamma)}\big)\rtimes \mathbb{Z}_2.$$
\item If $\Lambda$ is not self-complementary, then the automorphisms of $\Gamma\bar{\Gamma}$ are precisely the maps of the forms~\eqref{e28} and~\eqref{e29}. The automorphism group $\textrm{Aut}(\Gamma\bar{\Gamma})$ is isomorphic to the semidirect product $$\textrm{Aut}(\Gamma)\rtimes \mathbb{Z}_2.$$
\end{enumerate}
\end{prop}
\begin{proof}
Part (i) is well known and it can be found, for example, in \cite[Theorem~4.6]{petersen}.
\begin{figure}[h!]
\centering
\includegraphics[width=0.5\textwidth]{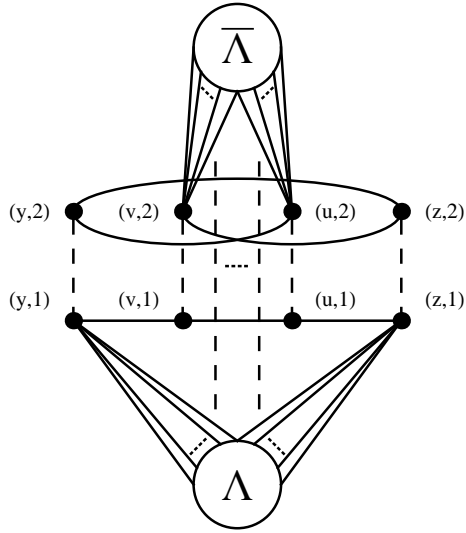}
\caption{Graph $\Gamma\bar{\Gamma}$ if $\Gamma=C_5(\Lambda)$.}
\label{f3}
\end{figure}
To prove (ii) and (iii) let $\Phi\in \textrm{Aut}(\Gamma\bar{\Gamma})$ be arbitrary. Observe that the degrees of vertices in the sets $S_1:=\{(y,1), (z,1), (v,2), (u,2)\}$ and $S_2:=\{(y,2), (z,2), (v,1), (u,1)\}$ all equal $|V(\Lambda)|+2$ and $3$, respectively, which are two different numbers, since $|V(\Lambda)|\neq 1$ (cf.~Figure~\ref{f3}). Further, only the vertices of the form $(g,i)$, with $g\in V(\Lambda)$ and $i\in\{1,2\}$, do not have neighbors of degree three. Hence, $\Phi(S_1)=S_1$, $\Phi(S_2)=S_2$, and
\begin{equation}\label{e23}
\Phi\Big(\big\{(g,i) : g\in V(\Lambda), i\in \{1,2\}\big\}\Big)=(\big\{(g,i) : g\in V(\Lambda), i\in \{1,2\}\big\}.
\end{equation}
Related to the set $S_1$ we next claim that either
\begin{equation}\label{e21}
\Phi\big(\{(y,1),(z,1)\}\big)=\{(y,1),(z,1)\}\quad \textrm{and}\quad \Phi\big(\{(v,2),(u,2)\}\big)=\{(v,2),(u,2)\}
\end{equation}
or
\begin{equation}\label{e22}
\Phi\big(\{(y,1),(z,1)\}\big)=\{(v,2),(u,2)\}\quad \textrm{and}\quad \Phi\big(\{(v,2),(u,2)\}\big)=\{(y,1),(z,1)\}.
\end{equation}
In fact, if $V(\Lambda)\geq 2$, then among the pairs of vertices in $S_1$ only the pairs $\{(y,1),(z,1)\}$ and $\{(v,2),(u,2)\}$ consist of two vertices with at least two common neighbors. If $V(\Lambda)=\emptyset$, then among the pairs of vertices in $S_1$, the property $d_{\Gamma\bar{\Gamma}}\big((y,1),(z,1)\big)=3=d_{\Gamma\bar{\Gamma}}\big((v,2),(u,2)\big)$  is unique to pairs $\{(y,1),(z,1)\}$ and $\{(v,2),(u,2)\}$.
If \eqref{e21} is true, then since vertices of the form $(g,1)$, with $g\in V(\Lambda)$, are adjacent to $(y,1)$ and $(z,1)$ but are nonadjacent to $(v,2)$ and $(u,2)$, it follows from \eqref{e23} that
$$\Phi\Big(\big\{(g,i) : g\in V(\Lambda) \big\}\Big)=\big\{(g,i) : g\in V(\Lambda)\big\}\qquad (i\in\{1,2\}).$$
Consequently, $\Phi(g,i)=(\varphi(g),i)$ for all $g\in V(\Lambda)$ and $i\in\{1,2\}$, where $\varphi$ is some automorphism of $\Lambda$. If \eqref{e22} is true, then we similarly deduce that
$$\Phi\Big(\big\{(g,i) : g\in V(\Lambda) \big\}\Big)=\big\{(g,\bar{i}) : g\in V(\Lambda)\big\}\qquad (i\in\{1,2\})$$
and consequently $\Phi(g,i)=(\sigma(g),\bar{i})$ for all $g\in V(\Lambda)$ and $i\in\{1,2\}$, where $\sigma$ is some antimorphism of $\Lambda$.
In particular, \eqref{e22} is possible only if $\Lambda$ is self-complementary. Finally, observe that $\Phi$ is fully determined by $\varphi$ or $\sigma$ and by its values on elements in the set $S_1$. Hence, by adopting the notation from \eqref{e24}-\eqref{e27}, we can describe $\Phi$ by the table
\begin{center}
\begin{tabular}{|c|c|c|c|c|}
\hline
$\Phi(y,1)$&$\Phi(z,1)$&$\Phi(v,2)$&$\Phi(u,2)$&$\Phi$\rule[-0.5ex]{0pt}{3ex}\\
\hline
\hline
$(y,1)$&$(z,1)$&$(v,2)$&$(u,2)$& $\Phi_1(x,i)=(\widehat{\varphi}_1(x),i)$\rule[-0.5ex]{0pt}{3ex}\\
\hline
$(y,1)$&$(z,1)$&$(u,2)$&$(v,2)$& $\Phi_2(x,i)=s_{C_5}(\widehat{\varphi}_1(x),i)$\rule[-0.5ex]{0pt}{3ex}\\
\hline
$(z,1)$&$(y,1)$&$(v,2)$&$(u,2)$& $\Phi_3(x,i)=s_{C_5}(\widehat{\varphi}_2(x),i)$\rule[-0.5ex]{0pt}{3ex}\\
\hline
$(z,1)$&$(y,1)$&$(u,2)$&$(v,2)$& $\Phi_4(x,i)=(\widehat{\varphi}_2(x),i)$\rule[-0.5ex]{0pt}{3ex}\\
\hline
$(v,2)$&$(u,2)$&$(y,1)$&$(z,1)$& $\Phi_5(x,i)=s_{C_5}(\widehat{\sigma}_2(x),\bar{i})$\rule[-0.5ex]{0pt}{3ex}\\
\hline
$(v,2)$&$(u,2)$&$(z,1)$&$(y,1)$& $\Phi_6(x,i)=(\widehat{\sigma}_1(x),\bar{i})$\rule[-0.5ex]{0pt}{3ex}\\
\hline
$(u,2)$&$(v,2)$&$(y,1)$&$(z,1)$& $\Phi_7(x,i)=(\widehat{\sigma}_2(x),\bar{i})$\rule[-0.5ex]{0pt}{3ex}\\
\hline
$(u,2)$&$(v,2)$&$(z,1)$&$(y,1)$& $\Phi_8(x,i)=s_{C_5}(\widehat{\sigma}_1(x),\bar{i})$\rule[-0.5ex]{0pt}{3ex}\\
\hline
\end{tabular}.
\end{center}
Conversely, all maps of the form $\Phi_1,\ldots,\Phi_8$ are really automorphism of $\Gamma\bar{\Gamma}$. Clearly,
maps $\Phi_5, \Phi_6, \Phi_7, \Phi_8$ are possible only if $\Lambda$ is self-complementary.

To determine the algebraic structure of $\textrm{Aut}(\Gamma\bar{\Gamma})$ assume firstly that $\Lambda$ is self-complementary.
Then the maps of the form $\Phi_1, \Phi_4, \Phi_6, \Phi_7$ form a subgroup $N$ in $\textrm{Aut}(\Gamma\bar{\Gamma})$, which is of index two, and therefore normal. Its elements either map $W_1$ to $W_1$ and $W_2$ to $W_2$ or map $W_1$ to $W_2$ and $W_2$ to $W_1$. On the other hand, the map $s_{C_5}$ does not have this property. Hence, the subgroup $H:=\{\textrm{identity map}, s_{C_5}\}$ intersects $N$ trivially. Therefore $\textrm{Aut}(\Gamma\bar{\Gamma})$ is isomorphic to the semidirect product $N\rtimes H$. Obviously, $H$ is isomorphic to $\mathbb{Z}_2$. The mapping that maps $\widehat{\varphi}\in \textrm{Aut}(\Gamma)$ into $(x,i)\mapsto (\widehat{\varphi}(x),i)$ and $\widehat{\sigma}\in \overline{\textrm{Aut}(\Gamma)}$ into $(x,i)\mapsto (\widehat{\sigma}(x),\bar{i})$ is an isomorphism between groups $\textrm{Aut}(\Gamma)\cup \overline{\textrm{Aut}(\Gamma)}$ and $N$. Hence, $\textrm{Aut}(\Gamma\bar{\Gamma})$ is isomorphic to $\big(\textrm{Aut}(\Gamma)\cup \overline{\textrm{Aut}(\Gamma)}\big)\rtimes \mathbb{Z}_2$, as claimed.

If $\Lambda$ is not self-complementary, then we determine $\textrm{Aut}(\Gamma\bar{\Gamma})$ in the same way. The only exception is that here $N$ consists of maps of the form $\Phi_1$ and $\Phi_4$ and the isomorphism between groups $\textrm{Aut}(\Gamma)$ and $N$ is given by the mapping that maps $\widehat{\varphi}\in \textrm{Aut}(\Gamma)$ into $(x,i)\mapsto (\widehat{\varphi}(x),i)$.
\end{proof}

\begin{prop}\label{p5}
Let $\Lambda$ be a (possibly null) graph and let $\Gamma=A(\Lambda)$.
\begin{enumerate}
\item If $\Lambda$ is a (possibly null) self-complementary graph, then the automorphisms of $\Gamma\bar{\Gamma}$ are precisely the maps of the forms
\begin{align}
\label{e30}\Phi(x,i)&=(\widehat{\varphi}(x),i),\\
\nonumber \Phi(x,i)&=(\widehat{\sigma}(x),\bar{i}),\\
\label{e31}\Phi(x,i)&=s_{A}(\widehat{\varphi}(x),i),\\
\nonumber \Phi(x,i)&=s_{A}(\widehat{\sigma}(x),\bar{i}),
\end{align}
for all $x\in V(\Gamma)$ and $i\in \{1,2\}$, where $\bar{i}\in \{1,2\}\backslash\{i\}$, $\widehat{\varphi}\in \textrm{Aut}(\Gamma)$, $\widehat{\sigma}\in \overline{\textrm{Aut}(\Gamma)}$. The automorphism group $\textrm{Aut}(\Gamma\bar{\Gamma})$ is isomorphic to $$\big(\textrm{Aut}(\Gamma)\cup \overline{\textrm{Aut}(\Gamma)}\big)\rtimes \mathbb{Z}_2.$$

\item If $\Lambda$ is not self-complementary, then the automorphisms of $\Gamma\bar{\Gamma}$ are precisely the maps of the forms~\eqref{e30} and~\eqref{e31}. The automorphism group $\textrm{Aut}(\Gamma\bar{\Gamma})$ is isomorphic to the semidirect product $$\textrm{Aut}(\Gamma)\rtimes \mathbb{Z}_2.$$
\end{enumerate}
\end{prop}
The proof of Proposition~\ref{p5} is almost the same as the proof of Proposition~\ref{p4}. Therefore we sketch only the differences.

\begin{proof}[Sketch of the proof]
Here, the vertices in the sets $S_1=\{(y,1),(z,1),(v,2),(u,2)\}$ and $S_2=\{(y,2),(z,2),(v,1),(u,1)\}$ have degree $2$ and $|V(\Lambda)|+3$, respectively, which are different numbers (cf. Figure~\ref{f4}).
\begin{figure}[h!]
\centering
\includegraphics[width=0.5\textwidth]{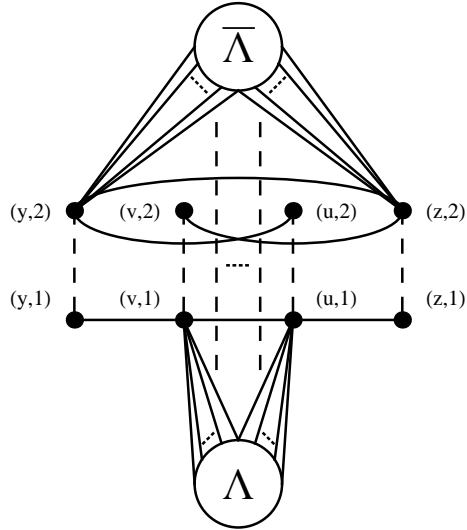}
\caption{Graph $\Gamma\bar{\Gamma}$ if $\Gamma=A(\Lambda)$.}
\label{f4}
\end{figure}
Only the vertices of the form $(g,i)$, with $g\in V(\Lambda)$ and $i\in\{1,2\}$, do not have neighbors of degree two and have degree at least three simultaneously. Hence, any $\Phi\in \textrm{Aut}(\Gamma\bar{\Gamma})$ fixes (setwise) the sets $S_1$, $S_2$, and $\{(g,i) : g\in V(\Lambda), i\in \{1,2\}\big\}$. Related to the set $S_2$ we deduce that either
\begin{equation}\label{e32}
\Phi\big(\{(y,2),(z,2)\}\big)=\{(y,2),(z,2)\}\quad \textrm{and}\quad \Phi\big(\{(v,1),(u,1)\}\big)=\{(v,1),(u,1)\}
\end{equation}
or
\begin{equation}\label{e33}
\Phi\big(\{(y,2),(z,2)\}\big)=\{(v,1),(u,1)\}\quad \textrm{and}\quad \Phi\big(\{(v,1),(u,1)\}\big)=\{(y,2),(z,2)\}.
\end{equation}
In fact, among the pairs of vertices in $S_2$ only the pairs $\{(y,2),(z,2)\}$ and $\{(v,1),(u,1)\}$ consist of two adjacent vertices. If \eqref{e32} is true, then we deduce that $\Phi(g,i)=(\varphi(g),i)$ for all $g\in V(\Lambda)$ and $i\in\{1,2\}$, where $\varphi\in \textrm{Aut}(\Lambda)$. Equality \eqref{e33} is possible only for $\Lambda$ self-complementary, in which case $\Phi(g,i)=(\sigma(g),\bar{i})$ for all $g\in V(\Lambda)$ and $i\in\{1,2\}$, where $\sigma\in \overline{\textrm{Aut}(\Lambda)}$. The automorphism~$\Phi$ is fully determined by $\varphi$ or $\sigma$ and by its values on elements in the set $S_2$ as described by the table
\begin{center}
\begin{tabular}{|c|c|c|c|c|}
\hline
$\Phi(y,2)$&$\Phi(z,2)$&$\Phi(v,1)$&$\Phi(u,1)$&$\Phi$\rule[-0.5ex]{0pt}{3ex}\\
\hline
\hline
$(y,2)$&$(z,2)$&$(v,1)$&$(u,1)$& $\Phi_1(x,i)=(\widehat{\varphi}_1(x),i)$\rule[-0.5ex]{0pt}{3ex}\\
\hline
$(y,2)$&$(z,2)$&$(u,1)$&$(v,1)$& $\Phi_2(x,i)=s_{A}(\widehat{\varphi}_2(x),i)$\rule[-0.5ex]{0pt}{3ex}\\
\hline
$(z,2)$&$(y,2)$&$(v,1)$&$(u,1)$& $\Phi_3(x,i)=s_{A}(\widehat{\varphi}_1(x),i)$\rule[-0.5ex]{0pt}{3ex}\\
\hline
$(z,2)$&$(y,2)$&$(u,1)$&$(v,1)$& $\Phi_4(x,i)=(\widehat{\varphi}_2(x),i)$\rule[-0.5ex]{0pt}{3ex}\\
\hline
$(v,1)$&$(u,1)$&$(y,2)$&$(z,2)$& $\Phi_5(x,i)=s_{A}(\widehat{\sigma}_1(x),\bar{i})$\rule[-0.5ex]{0pt}{3ex}\\
\hline
$(v,1)$&$(u,1)$&$(z,2)$&$(y,2)$& $\Phi_6(x,i)=(\widehat{\sigma}_1(x),\bar{i})$\rule[-0.5ex]{0pt}{3ex}\\
\hline
$(u,1)$&$(v,1)$&$(y,2)$&$(z,2)$& $\Phi_7(x,i)=(\widehat{\sigma}_2(x),\bar{i})$\rule[-0.5ex]{0pt}{3ex}\\
\hline
$(u,1)$&$(v,1)$&$(z,2)$&$(y,2)$& $\Phi_8(x,i)=s_{A}(\widehat{\sigma}_2(x),\bar{i})$\rule[-0.5ex]{0pt}{3ex}\\
\hline
\end{tabular}.
\end{center}
The rest is proved as in Proposition~\ref{p4}, with the exception that here $H=\{\textrm{identity map}, s_{A}\}$.
\end{proof}

To study the automorphism group of $\Gamma\bar{\Gamma}$, where $\Gamma$ is not isomorphic to graphs $C_5(\Lambda)$, $A(\Lambda)$, we need a simple but useful lemma.
\begin{lemma}\label{l4}
Let $\Gamma$ be any graph with $n$ vertices. Each triangle and each $4$-cycle in $\Gamma\bar{\Gamma}$ lies whole in the set $W_1$ or whole in the set $W_2$. Moreover,
\begin{equation}\label{e36}
\delta_{\Gamma\bar{\Gamma}}(v,1)+\delta_{\Gamma\bar{\Gamma}}(v,2)=n+1
\end{equation}
for all $v\in V(\Gamma)$.
\end{lemma}
\begin{proof}
Each vertex in $W_1$ has precisely one neighbor in $W_2$ and vice versa. Since $\delta_{\Gamma}(v)+\delta_{\bar{\Gamma}}(v)=n-1$ for all $v\in V(\Gamma)$, the equality~\eqref{e36} follows. The claim about the triangles and the 4-cycles is obvious.
\end{proof}

\begin{thm}\label{p3}
Let $\Gamma$ be a graph, which is not isomorphic to any of graphs $C_5(\Lambda)$ and $A(\Lambda)$, as $\Lambda$ ranges over all (possibly null) graphs.
\begin{enumerate}
\item If $\Phi$ is an automorphism of $\Gamma\bar{\Gamma}$ and $\Phi(u,i)=(w,j)$ for some distinct $i,j\in\{1,2\}$ and some $u,w\in V(\Gamma)$, then $\Phi(z,1)=(\sigma(z),2)$ and $\Phi(z,2)=(\sigma(z),1)$ for all $z\in V(\Gamma)$, where $\sigma$ is some antimorphism of~$\Gamma$.
\item If $\Phi$ is an automorphism of $\Gamma\bar{\Gamma}$ and $\Phi(u,i)=(w,i)$ for some $i\in\{1,2\}$ and some $u,w\in V(\Gamma)$, then $\Phi(z,j)=(\varphi(z),j)$ for all $z\in V(\Gamma)$ and all $j\in \{1,2\}$, where $\varphi$ is some automorphism of $\Gamma$.
\item The automorphism groups $\textrm{Aut}(\Gamma\bar{\Gamma})$ and $\textrm{Aut}(\Gamma)$ are isomorphic unless $\Gamma$ is self-complementary in which case $\textrm{Aut}(\Gamma\bar{\Gamma})$ is isomorphic to $\textrm{Aut}(\Gamma)\cup \overline{\textrm{Aut}(\Gamma)}$.
\end{enumerate}
\end{thm}

\begin{proof} Let $n=|V(\Gamma)|$. If $n=1$ or $n=2$, then  $\Gamma\bar{\Gamma}$ is the complete graph on two vertices or the path on four vertices, respectively, and the claims are obviously true. Hence, in the sequel we may assume that $n\geq 3$.

(i) Let $\Phi, i,j,u,w$ be as in the claim. Since $\Phi$ is bijective, it follows that $\Phi(W_1)\nsubseteq W_1$ and $\Phi(W_2)\nsubseteq W_2$. Hence, we may assume without loss of generality that $i=1$ and $j=2$, i.e.
\begin{equation}\label{e11}
\Phi(u,1)=(w,2).
\end{equation}
Since the inverse of an antimorphism is an antimorphism, it is irrelevant if prove the conclusion for $\Phi$ or for $\Phi^{-1}$ instead. By Lemma~\ref{l3}, $\Gamma$ or $\bar{\Gamma}$ is connected. We may assume that $\Gamma$ is connected (otherwise we consider $\Phi^{-1}$ instead of $\Phi$). Let $z\in V(\Gamma)$ be any neighbor of $u$ in $\Gamma$. Since $(u,1)\sim (z,1)$, it follows that $(w,2)\sim \Phi(z,1)$. Consequently
\begin{equation}\label{e9}
\Phi(z,1)=(z',2)
\end{equation}
for some $z'\in V(\Gamma)$ or
\begin{equation}\label{e10}
\Phi(z,1)=(w,1).
\end{equation}
We will prove that \eqref{e10} always lead to a contradiction. This will essentially end the proof of part (i). In fact, once we prove \eqref{e9},
we can apply the connectedness of $\Gamma$ and repeat the same argument to deduce that $\Phi(z,1)=(\sigma(z),2)$ for all $z\in V(\Gamma)$, where $\sigma: V(\Gamma)\to V(\Gamma)$ is some map. Since $\Phi$ is bijective, $\sigma$ must be bijective and $\Phi(z,2)=(\sigma'(z),1)$ for all $z\in V(\Gamma)$, where $\sigma': V(\Gamma)\to V(\Gamma)$ is some bijection. Since $\Phi(z,1)\sim \Phi(z,2)$, we must have $\sigma(z)=\sigma'(z)$ for all $z\in V(\Gamma)$. Finally, since $\Phi$ is an automorphism, it follows that $\sigma$ must be an antimorphism. Hence it remains to disprove \eqref{e10}.

Suppose that \eqref{e10} is true. By Lemma~\ref{l4},  the edge $\{(w,1), (w,2)\}$ does not lie on any triangle and on any $4$-cycle in $\Gamma\bar{\Gamma}$. Hence, the same is true for $\{(u,1), (z,1)\}$, and consequently also for the edge $\{u,z\}$ in graph $\Gamma$. It follows from the triangle condition that
\begin{equation}\label{e52}
N_{\bar{\Gamma}}(u)\cup N_{\bar{\Gamma}}(z)=V(\Gamma)\backslash \{u,z\},
\end{equation}
while the sets
$$B_{u}:=N_{\Gamma}(u)\backslash\{z\}\quad \textrm{and}\quad B_{z}:=N_{\Gamma}(z)\backslash\{u\}$$
are disjoint. Since $(u,1)\sim (u,2)$ and $(z,1)\sim (z,2)$, Eq.~$\eqref{e11}$, $\eqref{e10}$ imply that
\begin{equation}\label{e34}
\Phi(u,2)\in W_2\quad \textrm{and}\quad \Phi(z,2)\in W_1.
\end{equation}
Hence, $\Phi(u,2)$ and $\Phi(z,2)$ have at most one neighbor in common. Consequently, the same is true for $(u,2)$ and $(z,2)$, i.e. $$|N_{\bar{\Gamma}}(u)\cap N_{\bar{\Gamma}}(z)|\leq 1.$$
From equations $N_{\Gamma}(u)\cup N_{\bar{\Gamma}}(u)=V(\Gamma)\backslash\{u\}$, $N_{\Gamma}(z)\cup N_{\bar{\Gamma}}(z)=V(\Gamma)\backslash\{z\}$, and~\eqref{e52} we deduce that either $N_{\bar{\Gamma}}(u)\cap N_{\bar{\Gamma}}(z)=\emptyset$, in which case
\begin{align}
\label{e13}B_{u}&=N_{\bar{\Gamma}}(z),\ B_{z}=N_{\bar{\Gamma}}(u),\\
\label{e12}\textrm{and}\ V(\Gamma)&\ \textrm{is partitioned by the sets}\ B_u, B_z, \{u,z\},
\end{align}
or
$N_{\bar{\Gamma}}(u)\cap N_{\bar{\Gamma}}(z)=\{y\}$ for some $y\in V(\Gamma)$, in which case
\begin{align}
\label{e15}B_{u}\cup \{y\}&=N_{\bar{\Gamma}}(z),\ B_{z}\cup\{y\}=N_{\bar{\Gamma}}(u),\\
\label{e14}\textrm{and}\ V(\Gamma)&\ \textrm{is partitioned by the sets}\ B_u, B_z, \{u,z\}, \{y\}.
\end{align}
If \eqref{e13}-\eqref{e12} is true, then
$$\delta(u,1)+\delta(z,1)=|B_u|+|B_z|+|\{u,z\}|+|\{(u,2),(z,2)\}|=n+2,$$
while
$$\delta\big(\Phi(u,1)\big)+\delta\big(\Phi(z,1)\big)=\delta(w,1)+\delta(w,2)=\delta_{\Gamma}(w)+\delta_{\bar{\Gamma}}(w)+2=n+1,$$
a contradiction, since $\Phi$ is an automorphism. Hence, we may assume \eqref{e15}-\eqref{e14}. Since the edge $\{u,z\}$ does not lie on any $4$-cycle in $\Gamma$, it follows that there are no edges in $\Gamma$ with one vertex in $B_u$ and the other vertex in $B_z$. In other words,
$(x_u,2)\sim (x_z,2)$ and consequently
\begin{equation}\label{e35}
\Phi(x_u,2)\sim \Phi(x_z,2)
\end{equation}
for all $x_u\in B_u$ and $x_z\in B_z$. Moreover, \eqref{e15} implies that $\Phi(x_u,2)\sim \Phi(z,2)$ and $\Phi(x_z,2)\sim \Phi(u,2)$ for all $x_u\in B_u$ and $x_z\in B_z$. Consequently, it follows from \eqref{e34} that $\Phi(x_u,2)\in W_1$ for all $x_u\in B_u$ with one possible exception, and  $\Phi(x_z,2)\in W_2$ for all $x_z\in B_z$ with one possible exception. Hence, if $|B_u|\geq 3$ and $|B_z|\geq 2$, or $|B_u|\geq 2$ and $|B_z|\geq 3$, we get a contradiction by~\eqref{e35}, since there is no vertex in $W_2$ that is adjacent to more than one vertex in $W_1$ or vice versa. The remaining possibilities are split into five cases. Regardless of the case, recall from \eqref{e15} and \eqref{e34} that $$W_2\ni \Phi(u,2)\sim \Phi(y,2)\sim \Phi(z,2)\in W_1.$$ Hence Lemma~\ref{l4} implies that either
\begin{equation}\label{e55}
n+1=\delta \big(\Phi(u,2)\big)+\delta\big(\Phi(y,2)\big)=\delta(u,2)+\delta(y,2)
\end{equation}
or
\begin{equation}\label{e56}
n+1=\delta \big(\Phi(y,2)\big)+\delta\big(\Phi(z,2)\big)=\delta(y,2)+\delta(z,2),
\end{equation}
depending on whether $\Phi(y,2)\in W_1$ or $\Phi(y,2)\in W_2$, respectively.

\begin{myenumerate}{Case}
\item Let $|B_u|=2=|B_z|$. If we denote $B_u=\{x_u,\dot{x}_u\}$ and $B_z=\{x_z,\dot{x}_z\}$, then $V(\Gamma)=\{z, u, y, x_u, \dot{x}_u, x_z, \dot{x}_z\}$. Recall from~\eqref{e15} that $\delta(u,2)=4=\delta(z,2)$. Consequently,
    \eqref{e55}-\eqref{e56} imply that $\delta(y,2)=4$. Besides $(u,2)$ and $(z,2)$, the vertex $(y,1)$ is another neighbor of $(y,2)$. It follows that precisely one vertex among $(x_u,2), (\dot{x}_u,2), (x_z,2), (\dot{x}_z,2)$ is a neighbor of $(y,2)$ too. Denote this vertex by $(a,2)$.
    Consequently, the vertex $\Phi(a,2)$ is unique among $$\Phi(x_u,2), \Phi(\dot{x}_u,2), \Phi(x_z,2), \Phi(\dot{x}_z,2)$$
    to be adjacent to $\Phi(y,2)$.

    \hspace{1em} If $\Phi(y,2)\in W_1$, then we claim that $a\in \{x_u,\dot{x}_u\}$. The opposite, i.e. $a\in \{x_z,\dot{x}_z\}$, and \eqref{e15} would imply that vertices $\Phi(y,2), \Phi(u,2), \Phi(a,2)$ form a triangle, which is partly in $W_1$ and partly in $W_2$, a contradiction by Lemma~\ref{l4}. Hence, $a\in \{x_u,\dot{x}_u\}$ and $\Phi(a,2),\Phi(y,2), \Phi(u,2), \Phi(x_z,2)$ form a 4-cycle by \eqref{e35}. Since $\Phi(y,2)\in W_1$ and $\Phi(u,2)\in W_2$, we get a contradiction by Lemma~\ref{l4}.

    \hspace{1em} Similarly, if $\Phi(y,2)\in W_2$, then $a\in \{x_z,\dot{x}_z\}$, since the opposite would imply that $\Phi(y,2), \Phi(z,2), \Phi(a,2)$ form a triangle, which is partly in $W_1$ and partly in $W_2$. Hence, $\Phi(a,2),\Phi(y,2), \Phi(z,2), \Phi(x_u,2)$ form a 4-cycle with $\Phi(y,2)\in W_2$ and $\Phi(z,2)\in W_1$, a contradiction.

\item\label{case2} Let $B_z=\emptyset$. Then $V(\Gamma)=\{u,z,y\}\cup B_u$, $|B_u|=n-3$,  and  \eqref{e15} implies that $N_{\bar{\Gamma}}(z)=B_{u}\cup \{y\}$, $N_{\bar{\Gamma}}(u)=\{y\}$. Let $b$ denote the number of edges between $y$ and $B_u$ in $\Gamma$. Then there are $|B_u|-b$ such edges in $\bar{\Gamma}$ (see Figure~\ref{f5}).
    \begin{figure}[h!]
\centering
\begin{tabular}{ccc}
\includegraphics[width=0.31\textwidth]{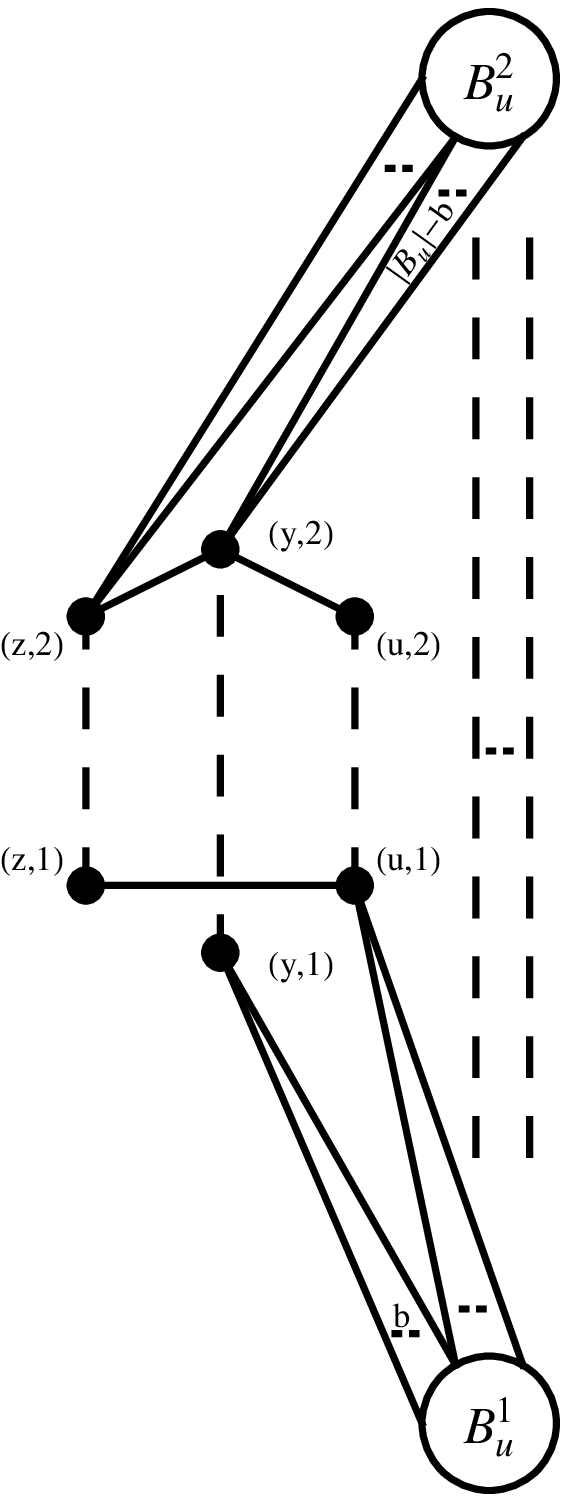}&\phantom{a}\qquad\qquad\qquad\phantom{a} &
\includegraphics[width=0.37\textwidth]{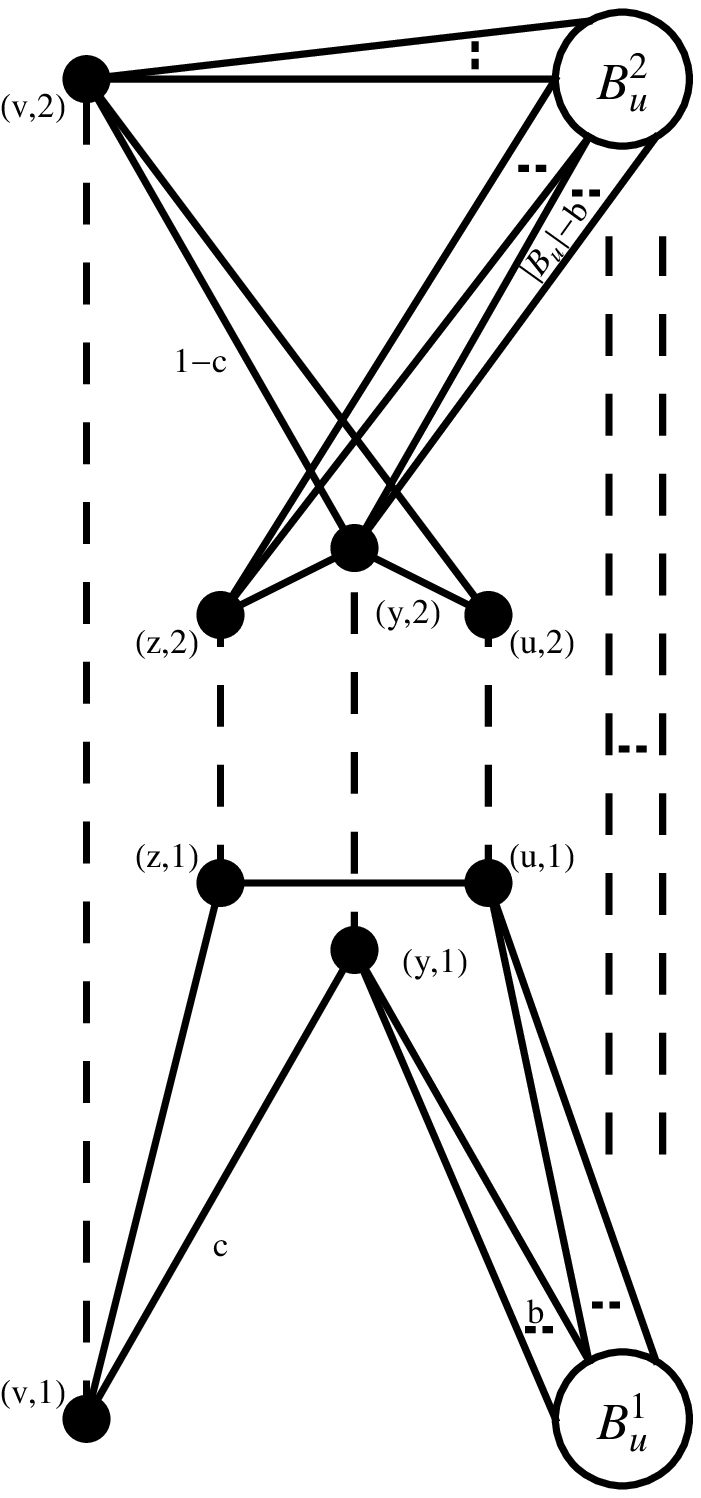}
\end{tabular}
\caption{Graphs from Cases~\ref{case2} and~\ref{case4} are on the left and on the right, respectively. Here, $B_u^{i}=\{(x,i) : x\in B_u\}$ for $i\in\{1,2\}$.}
\label{f5}
\end{figure}
Consequently, \eqref{e55}-\eqref{e56} imply that either
    \begin{equation}\label{e37}
    n+1=2+\big(3+(|B_u|-b)\big)=n+2-b
    \end{equation}
    or
    \begin{equation}\label{e38}
    n+1=\big(3+(|B_u|-b)\big)+(|B_u|+2)=2n-1-b,
    \end{equation}
    depending on whether $\Phi(y,2)\in W_1$ or $\Phi(y,2)\in W_2$, respectively. In case of \eqref{e38} we get $b=n-2>|B_{u}|$, a contradiction. Hence, \eqref{e37} is correct, which yields $b=1$, $\Phi(y,2)\in W_1$, $B_u\neq \emptyset$, and $n\geq 4$. We deduce also that
    $$\delta(y,2)=n-1=\delta(z,2)$$
    and consequently $\delta\big(\Phi(y,2)\big)=n-1=\delta\big(\Phi(z,2)\big)$. Since $\Phi(y,2)$, $\Phi(z,2)$ are both in $W_1$, while
    $$n-1>2=\delta(y,1)=\delta(z,1)\quad \textrm{and}\quad \delta(u,1)=2+|B_u|=n-1,$$
    we deduce that one vertex $(v,1)$ in $W_1$, with $v\in B_u$, must have degree $n-1$. Since $|B_u|=n-3$, this is possible only if, in graph $\Gamma$, vertex $v$ is adjacent to all other vertices in $B_u$, it is the (unique) vertex in $B_u$ adjacent to $y$, and as all vertices in $B_u$ it is adjacent to $u$. To summarize, vertices $y, v, u, z$, in this order, form a path, where $v$ and $u$ are both adjacent to all vertices in $B_u\backslash \{v\}$, while $y$ and $z$ have no neighbors in $B_u\backslash \{v\}$. Hence, $\Gamma$ is isomorphic to $A(\Lambda)$, where $\Lambda$ is induced by the set $B_u\backslash \{v\}$. This is a contradiction by the assumption.

\item Let $B_u=\emptyset$. Then $V(\Gamma)=\{u,z,y\}\cup B_z$, $|B_z|=n-3$,  and  \eqref{e15} implies that $N_{\bar{\Gamma}}(u)=B_{z}\cup \{y\}$, $N_{\bar{\Gamma}}(z)=\{y\}$. Let $b$ denote the number of edges between $y$ and $B_z$ in $\Gamma$. Then there are $|B_z|-b$ such edges in $\bar{\Gamma}$. Equations \eqref{e55}-\eqref{e56} imply that either
    \begin{equation}\label{e53}
    n+1=(|B_z|+2)+\big(3+(|B_z|-b)\big)=2n-1-b
    \end{equation}
    or
    \begin{equation}\label{e54}
    n+1=\big(3+(|B_z|-b)\big)+2=n+2-b,
    \end{equation}
    depending on whether $\Phi(y,2)\in W_1$ or $\Phi(y,2)\in W_2$, respectively. In case of \eqref{e53} we get $b=n-2>|B_{z}|$, a contradiction. Hence, \eqref{e54} is correct, which yields $b=1$ and
    \begin{equation}\label{e57}
    \Phi(y,2)\in W_2.
    \end{equation}
    In particular, $B_z\neq \emptyset$ and $n\geq 4$. Let $v$ be the unique vertex in $B_z$ that is a neighbor of $y$ in $\Gamma$. Then for each $x_z\in B_z\backslash\{v\}$ vertices $\Phi(x_z,2)$, $\Phi(y,2)$, $\Phi(u,2)$ form a triangle. Since $\Phi(u,2)\in W_2$ by \eqref{e34}, Lemma~\ref{l4} implies that
    \begin{equation}\label{e58}
    \Phi(x_z,2)\in W_2
    \end{equation}
    for all $x_z\in  B_z\backslash\{v\}$. The vertex $v$ can not be adjacent to all other vertices from the set $B_z$ in graph $\Gamma$, since this would imply that $\Gamma$ is isomorphic to the graph $A(\Lambda)$ with  $\Lambda$ being induced by the set $B_z\backslash \{v\}$. Hence, there is $x\in B_z$ such that $v\sim_{\bar{\Gamma}} x$. Consequently, vertices $\Phi(v,2)$, $\Phi(x,2)$, $\Phi(u,2)$ form a triangle. Since $\Phi(u,2)\in W_2$, Lemma~\ref{l4} implies that
    \begin{equation}\label{e59}
    \Phi(v,2)\in W_2.
    \end{equation}
    Recall that $\delta\big(\Phi(u,2)\big)=\delta(u,2)=n-1$. Hence, \eqref{e11} and \eqref{e57}-\eqref{e59} imply that all neighbors of $\Phi(u,2)$ are in $W_2$. This is a contradiction, since each vertex in~$W_2$ has one neighbor in $W_1$.

\item\label{case4} Let $|B_z|=1$ and $|B_u|\geq 1$.  Then $V(\Gamma)=\{u,z,y,v\}\cup B_u$, where $B_z=\{v\}$  and  \eqref{e15} implies that $N_{\bar{\Gamma}}(z)=B_{u}\cup \{y\}$, $N_{\bar{\Gamma}}(u)=\{y,v\}$. Observe also that $n\geq 5$. For graph $\Gamma$ let $b$ denote the number of edges between $y$ and $B_u$ and let $c$ be the number of edges between $y$ and $v$ in $\Gamma$. Then
    $$0\leq b\leq |B_u|=n-4,\qquad c\in\{0,1\},$$
    and the corresponding numbers of edges in the complement $\bar{\Gamma}$ are $|B_u|-b=n-4-b$ and $1-c$, respectively (see Figure~\ref{f5}). Eq.~\eqref{e55}-\eqref{e56} imply that either
    \begin{equation}\label{e39}
    n+1=3+\big(3+(|B_u|-b)+(1-c)\big)=n+3-b-c
    \end{equation}
    or
    \begin{equation}\label{e40}
    n+1=\big(3+(|B_u|-b)+(1-c)\big)+|B_u|+2=2n-2-b-c,
    \end{equation}
    depending on whether $\Phi(y,2)\in W_1$ or $\Phi(y,2)\in W_2$, respectively. In case of \eqref{e40}, we deduce that $b+c=n-3$, which is possible only if $b=n-4$ and $c=1$. In this case $\Gamma$ is isomorphic to the graph $C_5(\Lambda)$, where graph $\Lambda$ is induced by the set $B_u$. This is a contradiction by the assumption. Hence, we may assume \eqref{e39}, i.e. $b+c=2$ and $\Phi(y,2)\in W_1$. If $c=0$, then vertices $\Phi(u,2)$, $\Phi(y,2)$, $\Phi(v,2)$ form a triangle that is partly in $W_1$ and partly in $W_2$. Since this is not possible by Lemma~\ref{l4}, we have $c=1=b$. If $n=5$, $\Gamma$ is the pentagon, i.e. $C_5(\Lambda)$, where $\Lambda$ is the graph with one vertex. Since this is ruled out by the assumption, we may assume that $n\geq 6$. Consequently, $y$ has at least one neighbor $s\in B_u$ in $\bar{\Gamma}$ and vertices $\Phi(u,2), \Phi(v,2), \Phi(s,2), \Phi(y,2)$, in this order, form a 4-cycle in $\Gamma\bar{\Gamma}$. Since $\Phi(y,2)\in W_1$ and $\Phi(u,2)\in W_2$ we get in contradiction by Lemma~\ref{l4}.

\item Let $|B_u|=1$ and $|B_z|\geq 1$. Then $V(\Gamma)=\{u,z,y,v\}\cup B_z$, where $B_u=\{v\}$  and  \eqref{e15} implies that $N_{\bar{\Gamma}}(u)=B_{z}\cup \{y\}$, $N_{\bar{\Gamma}}(z)=\{y,v\}$. Moreover, $n\geq 5$. For graph $\Gamma$ let $b$ denote the number of edges between $y$ and $B_z$ and let~$c$ be the number of edges between $y$ and $v$ in $\Gamma$. Then
    $$0\leq b\leq |B_z|=n-4,\qquad c\in\{0,1\},$$
    and the corresponding numbers of edges in the complement $\bar{\Gamma}$ are $|B_z|-b$ and $1-c$, respectively. Equations \eqref{e55}-\eqref{e56} imply that either
    \begin{equation}\label{e60}
    n+1=|B_z|+2+\big(3+(|B_z|-b)+(1-c)\big)=2n-2-b-c
    \end{equation}
    or
    \begin{equation}\label{e61}
    n+1=\big(3+(|B_z|-b)+(1-c)\big)+3=n+3-b-c,
    \end{equation}
    depending on whether $\Phi(y,2)\in W_1$ or $\Phi(y,2)\in W_2$, respectively. In case of \eqref{e60}, we deduce that $b+c=n-3$, which is possible only if $b=n-4$ and $c=1$. In this case $\Gamma$ is isomorphic to graph $C_5(\Lambda)$, where graph $\Lambda$ is induced by the set $B_z$. This is a contradiction by the assumption. Hence, we may assume \eqref{e61}, i.e. $b+c=2$ and
    \begin{equation}\label{e62}
    \Phi(y,2)\in W_2.
    \end{equation}
    Consequently, either
    $$c=0,\ b=2\qquad \textrm{or}\qquad c=1=b.$$
    For each of the $|B_z|-b$ vertices $x_z\in B_z$ that are adjacent to $y$ in $\bar{\Gamma}$ we deduce that $\Phi(x_z,2)$, $\Phi(y,2)$, $\Phi(u,2)$ form a triangle. Consequently, Lemma~\ref{l4} and \eqref{e62} imply that
    \begin{equation}\label{e63}
    \Phi(x_z,2)\in W_2
    \end{equation}
    for any such $x_z$.

    \hspace{1em} Suppose that $c=0$ and $b=2$. Let $x,\dot{x}\in B_z$ be those vertices that are not adjacent to $y$ in $\bar{\Gamma}$. Then $\Phi(x,2)$, $\Phi(v,2)$, $\Phi(y,2)$, $\Phi(u,2)$ and $\Phi(\dot{x},2)$, $\Phi(v,2)$, $\Phi(y,2)$, $\Phi(u,2)$ are both $4$-cycles in $\Gamma\bar{\Gamma}$. Lemma~\ref{l4} and \eqref{e62} imply that
    \begin{equation}\label{e64}
    \Phi(x,2), \Phi(\dot{x},2)\in W_2.
    \end{equation}
    Recall that $\delta\big(\Phi(u,2)\big)=\delta(u,2)=|B_z|+2=n-2$. Hence, \eqref{e11} and \eqref{e62}-\eqref{e64} imply that all neighbors of $\Phi(u,2)$ are in $W_2$. This is a contradiction, since a vertex in $\Gamma\bar{\Gamma}$ can not have all neighbors in $W_2$.

    \hspace{1em} Let $c=1=b$. Let $x\in B_z$ be the vertex that is not adjacent to $y$ in~$\bar{\Gamma}$. If $x$ is adjacent to all vertices from the set $B_z\backslash \{x\}$ in graph $\Gamma$, then $\Gamma$ is isomorphic to the graph $C_5(\Lambda)$, where graph $\Lambda$ is induced by the set $B_z\backslash \{x\}$. This is a contradiction by the assumption. Hence, $x\sim_{\bar{\Gamma}} x'$ for some $x'\in B_z$. Consequently, $\Phi(x,2)$, $\Phi(x',2)$, $\Phi(u,2)$ form a triangle. Since $\Phi(u,2)\in W_2$ by \eqref{e34}, Lemma~\ref{l4} implies that $\Phi(x,2)\in W_2$. Together with \eqref{e11} and \eqref{e62}-\eqref{e63} we deduce that all neighbors of $\Phi(u,2)$ are in $W_2$, a contradiction as above.

    \hspace{1em} The proof of part (i) is completed.
\end{myenumerate}

(ii) Let $\Phi, i,v,w$ be as in the claim. We may assume that $i=1$, i.e.
$\Phi(v,1)=(w,1)$. By part (i) it follows that $\Phi(u,1)=(\varphi(u),1)$ for all $u\in V(\Gamma)$ for some map $\varphi:  V(\Gamma)\to V(\Gamma)$. Since $\Phi$ is bijective, it follows that $\varphi$ is bijective and $\Phi(u,2)=(\varphi'(u),2)$ for all $u\in V(\Gamma)$, where $\varphi': V(\Gamma)\to V(\Gamma)$ is some bijection. Since $\Phi(u,1)\sim \Phi(u,2)$, it follows that $\varphi(u)=\varphi'(u)$ for all $u\in V(\Gamma)$. Since $\Phi$ is an automorphism, it follows that $\varphi$ is an automorphism of $\Gamma$.

(iii) Consider the map $\mathfrak{I}: \textrm{Aut}(\Gamma)\to\textrm{Aut}(\Gamma\bar{\Gamma})$, defined by $\mathfrak{I}(\varphi)=\Phi_{\varphi}$ for all $\varphi\in \textrm{Aut}(\Gamma)$, where $\Phi_{\varphi}(u,i)=(\varphi(u),i)$ for all $u\in V(\Gamma)$ and $i\in\{1,2\}$. Clearly, $\mathfrak{I}$ is an injective group homomorphism. If $\Gamma$ is not self-complementary, then parts (i), (ii) imply that $\mathfrak{I}$ is also surjective, i.e. an isomorphism. If $\Gamma$ is self-complementary, then we can extend $\mathfrak{I}$ to an injective group homomorphism $\textrm{Aut}(\Gamma)\cup \overline{\textrm{Aut}(\Gamma)}\to \textrm{Aut}(\Gamma\bar{\Gamma})$, by defining $\mathfrak{I}(\sigma)=\Phi_{\sigma}$ for all $\sigma\in \overline{\textrm{Aut}(\Gamma)}$, where $\Phi_{\sigma}(u,i)=(\sigma(u),\bar{i})$ for all $u\in V(\Gamma)$ and $i\in\{1,2\}$. Here $\bar{i}\in \{1,2\}\backslash \{i\}$. By parts (i), (ii), the extended $\mathfrak{I}$ is surjective, and therefore it is a group isomorphism.
\end{proof}

Recall that $|\overline{\textrm{Aut}(\Gamma)}|=|\textrm{Aut}(\Gamma)|$ for any self-complementary graph $\Gamma$. Consequently, Propositions~\ref{p4} and \ref{p5}, Theorem~\ref{p3}, and the well known fact that the automorphism group of the pentagon is the dihedral group with 10 elements, while $|S_5|=120$, imply the following result.

\begin{cor}
Let $\Gamma$ be any graph. Then the number $$\frac{|\textrm{Aut}(\Gamma\bar{\Gamma})|}{|\textrm{Aut}(\Gamma)|}$$ equals:
\begin{enumerate}
\item 12 if $\Gamma$ is the pentagon;
\item 4 if $\Gamma$ is isomorphic to either $A(\Lambda)$ or $C_5(\Lambda)$, where $\Lambda$ is a (possibly null) self-complementary graph, and where we assume that $|V(\Lambda)|\neq 1$ in the $C_5(\Lambda)$ case;
\item 2 if $\Gamma$ is a self-complementary graph that is not included in (i) and (ii) or $\Gamma$ is isomorphic to either $A(\Lambda)$ or $C_5(\Lambda)$, where $\Lambda$ is not self-complementary;
\item 1 otherwise.
\end{enumerate}
\end{cor}
\begin{remark}
Observe that the map
$\mathfrak{I}: \textrm{Aut}(\Gamma)\to\textrm{Aut}(\Gamma\bar{\Gamma})$, defined by $\mathfrak{I}(\varphi)=\Phi_{\varphi}$, where   $\Phi_{\varphi}(u,i)=(\varphi(u),i)$; $(u\in V(\Gamma), i\in\{1,2\})$ is an injective group homomorphism. In this sense $\textrm{Aut}(\Gamma)$ can be understood as a subgroup of $\textrm{Aut}(\Gamma\bar{\Gamma})$. It turns out that only in the case of the pentagon, it is not a normal subgroup. In all other cases $\textrm{Aut}(\Gamma\bar{\Gamma})/\mathfrak{I}(\textrm{Aut}(\Gamma))$ is a quotient group, which is trivial in case (iv) and isomorphic to $\mathbb{Z}_2$ in case (iii). In case (ii) it turns out that the quotient group is isomorphic to the Klein four-group $\mathbb{Z}_2\times \mathbb{Z}_2$.
\end{remark}

\begin{remark}
Observe that among the graphs $C_5(\Lambda)$ and $A(\Lambda)$ the pentagon is unique to be regular. Hence, the assumption in Theorem~\ref{p3} is valid for all other regular graphs.
\end{remark}

Clearly, if $\Gamma$ has $n$ vertices, then $\Gamma\bar{\Gamma}$ is regular if and only if $\Gamma$ is $(\frac{n-1}{2})$-regular (see also~\cite[Theorem~3.6]{carvalho2018}). In this case $\Gamma\bar{\Gamma}$ is $(\frac{n+1}{2})$-regular. Next we consider the vertex-transitive property.

\begin{cor}\label{c1}
Let $\Gamma$ be any graph. Then $\Gamma\bar{\Gamma}$ is vertex-transitive if and only if~$\Gamma$ is vertex-transitive and self-complementary.
\end{cor}
\begin{proof}
Clearly, the claim is true if $\Gamma$ is the pentagon, i.e. if $\Gamma\bar{\Gamma}$ is the Petersen graph. In the sequel we assume that  $\Gamma$ is not the pentagon.

Assume that $\Gamma$ is vertex-transitive and self-complementary. Pick any two vertices $(v,i)$ and $(w,j)$ in $\Gamma\bar{\Gamma}$, where $v,w\in V(\Gamma)$ and $i,j\in\{1,2\}$.

Suppose firstly that $i=j$. By the assumption there exists an automorphism $\varphi$ of $\Gamma$ such that $\varphi(v)=w$. Since $\varphi$ is also an automorphism of $\bar{\Gamma}$ it follows that the map $\Phi(u,t)=(\varphi(u),t)$, defined for all $u\in V(\Gamma)$ and $t\in\{1,2\}$, is an automorphism of $\Gamma\bar{\Gamma}$ that maps $(v,i)$ into $(w,j)$.

Suppose now that $i\neq j$. By the assumption there exist an antimorphism $\sigma : \Gamma\to \bar{\Gamma}$ and an
automorphism $\varphi$ of $\Gamma$ such that $\varphi(\sigma(v))=w$. Then the map $\Phi(u,t)=(\varphi(\sigma(u)),\bar{t})$, defined for all $u\in V(\Gamma)$ and $t\in\{1,2\}$, where $\bar{t}\in\{1,2\}\backslash\{t\}$, is an automorphism of $\Gamma\bar{\Gamma}$ that maps $(v,i)$ into $(w,j)$.

Conversely, assume that $\Gamma\bar{\Gamma}$ is vertex-transitive. In particular, it is regular, which implies that $\Gamma$ is regular. Hence, the condition in Theorem~\ref{p3} is satisfied. Let $v,w\in V(\Gamma)$ be arbitrary. By the assumption there exist $\Phi,\Psi\in \textrm{Aut}(\Gamma\bar{\Gamma})$ such that $\Phi(v,1)=(w,1)$ and $\Psi(v,1)=(w,2)$. By parts (ii) and (i) in Theorem~\ref{p3}, there are $\varphi\in \textrm{Aut}(\Gamma)$ and $\sigma\in \overline{\textrm{Aut}(\Gamma)}$ such that
$\Phi(u,t)=(\varphi(u),t)$ and $\Psi(u,t)=(\sigma(u),\bar{t})$ for all $u\in V(\Gamma)$ and $t\in \{1,2\}$.
In particular, $\Gamma$ is self-complementary and
$$(\varphi(v),1)=\Phi(v,1)=(w,1).$$
Therefore $\varphi(v)=w$ and $\Gamma$ is vertex-transitive.
\end{proof}

Recall that each Cayley graph is vertex-transitive, while the converse is not true. This is witnessed also by vertex-transitive complementary prisms as we show below. We remark that a corresponding result for a bipartite double cover, which is defined similarly as a complementary prism, is different~\cite{ademir}.

\begin{cor}\label{c2}
If $\Gamma$ is a graph with $|V(\Gamma)|>1$, then $\Gamma\bar{\Gamma}$ is not a Cayley graph.
\end{cor}

\begin{proof}
Suppose that  $\Gamma\bar{\Gamma}$ is a Cayley graph. By Lemma~\ref{sabid} there exists a subgroup $G$ in $\textrm{Aut}(\Gamma\bar{\Gamma})$ that acts transitively on $V(\Gamma\bar{\Gamma})$ and $|G|=|V(\Gamma\bar{\Gamma})|$. Since each Cayley graph is vertex-transitive, Corollary~\ref{c1} implies that $\Gamma$ is vertex-transitive and self-complementary. Since the Petersen graph is not a Cayley graph, it follows that $\Gamma$ is not the pentagon. Consequently, the assumptions of Theorem~\ref{p3} are satisfied. Pick an arbitrary vertex $v_0\in V(\Gamma)$. Since $G$ acts transitively on $V(\Gamma\bar{\Gamma})$, there exists $\Psi\in G$ such that $\Psi(v_0,1)=(v_0,2)$. By part~(i) of Theorem~\ref{p3} there exists $\sigma\in \overline{\textrm{Aut}(\Gamma)}$ such that
\begin{equation*}\label{e43}
\Psi(x,i)=(\sigma(x),\bar{i})
\end{equation*}
for all $x\in V(\Gamma)$ and $i\in\{1,2\}$. In particular, $\sigma(v_0)=v_0$ and $\Psi^2\in G$ fixes the vertex $(v_0,1)$. Since $G$ acts transitively on vertices and satisfies $|G|=|V(\Gamma\bar{\Gamma})|$, there are no other elements in $G$ that fix $(v_0,1)$. Hence, $\Psi^2$ is the identity map, and therefore $(v,i)=\Psi(\Psi(v,i))=(\sigma^2(v),i)$ for all $v\in V(\Gamma)$ and $i\in\{1,2\}$. It follows that $\sigma^2$ is the identity map on $V(\Gamma)$, which contradicts Lemma~\ref{order_antim}.
\end{proof}

Corollaries~\ref{c1} and~\ref{c2} combine two seemingly unrelated but rich research areas in graph theory: 1) non-Cayley vertex-transitive graphs and 2) vertex-transitive self-complementary graphs. The claim about being seemingly unrelated is backed by the fact that the first known construction of a non-Cayley vertex-transitive self-complementary graph appeared only in 2001 by Li and Praeger~\cite{praeger}.

The study of (general) non-Cayley vertex-transitive graphs has a longer tradition. Backed by the results in \cite{alspach1982,alspach1979,frucht,godsil1980,marusic1985,mckay}, in 1983 Maru\v{s}i\v{c}~\cite{marusic1983} proposed a problem to characterize all \emph{Cayley numbers}, i.e., positive integers $n$ such that each vertex-transitive graph on $n$ numbers is a Cayley graph. This problem fostered a lot of research (see for example \cite{pisanski,ted2008,tedpablo,AMC,istvan,klavdija2020,praeger1998,praegerJCTB,adam,klavdija2012,seress,scapellato,praeger1994,praeger1996,praeger1993JCTB,seress1998} and the references therein) and to the best of our knowledge the problem remains open in a few special cases.

The first results about the order of vertex-transitive self-complementary graphs dates back to Sachs~\cite{sachs} who was interested in the order of circulant graphs. Zelinka~\cite{zelinka} was probably the first to study the order of (general) vertex-transitive self-complementary graphs. Parallelly, Kotzig~\cite{kotzig} proposed a list of problems on regular self-complementary graphs, including a problem on the order. All mentioned motivated a lot of research related to vertex-transitive self-complementary graphs (see for example \cite{alspach1999,ted2004,rosa,jajcay,Li1997,Li2014,muzychuk,peisert,rao,suprunenko,zhang} and the references therein). In particular, Rao~\cite{rao} constructed a vertex-transitive self-complementary graph of any order $n$ with the prime decomposition of the form
$$n=p_1^{a_1}\cdots p_{s}^{a_s},$$
where $p_i^{a_i}\equiv 1 (\textrm{mod}~4)$ for all $1\leq i\leq s$. In 1999 Muzychuk~\cite{muzychuk} proved that there are no vertex-transitive self-complementary graphs with different orders (see also the survey~\cite{beezer}).

By Corollaries~\ref{c1} and~\ref{c2}, $\Gamma\bar{\Gamma}$ is a vertex-transitive non-Cayley graph whenever $\Gamma$ is a vertex-transitive self-complementary graph with more than one vertex. Hence, in this way we obtain vertex-transitive non-Cayley graphs of order
\begin{equation}\label{e45}
2p_1^{a_1}\cdots p_{s}^{a_s},\ \textrm{where}\ p_i^{a_i}\equiv 1 (\textrm{mod}~4)\ \textrm{for all}\ i,
\end{equation}
and not all $a_i$ are zero. Recall from Lemma~\ref{p6} that these graphs are connected. Actually, the diameter of any such graph is two, since the diameter of any vertex-transitive self-complementary graph is two as well (see for example~\cite{sc-survey}). One way to obtain such  graphs $\Gamma\bar{\Gamma}$ is by choosing $\Gamma$ to be a lexicographic product of graphs $\Gamma_1,\ldots,\Gamma_s$, where each $\Gamma_i$ is the Paley graph on $p_i^{a_i}$ vertices (see~\cite{rao,beezer}). Further, it is known from number theory~(see for example \cite[p.~142]{dudley}), that each positive integer is a sum of two squares unless the prime decomposition contains a prime congruent to $3$ $(\textrm{mod}~4)$ to an odd power. Since the sum of two squares is an odd number if and only if precisely one of the two squares is odd, we deduce that the values in~\eqref{e45} are the same as the values
\begin{equation}\label{e46}
2\Big((2i)^2+(2j+1)^2\Big)\quad \textrm{with}\quad  i,j\in \{0,1,2,\ldots\}\ \textrm{and}\ (i,j)\neq (0,0).
\end{equation}
Note that the values~\eqref{e45}-\eqref{e46} are not new non-Cayley numbers. In fact, observe that each number in \eqref{e45}-\eqref{e46} is either a multiple of
\begin{equation}\label{e47}
2p,\ \textrm{where}\  p\ \textrm{is a prime with}\ p\equiv 1\ (\textrm{mod}~4)
\end{equation}
or
\begin{equation}\label{e48}
2p^2\ \textrm{where}\  p\ \textrm{is a prime with}\ p\equiv 3\ (\textrm{mod}~4).
\end{equation}
An example of non-Cayley vertex-transitive graph of order~\eqref{e47} is the generalized Petersen graph $G(2p,i)$, where $i\in\{0,1,\ldots,p-1\}$ is the odd integer that is the square root of $-1$ in the field $\mathbb{Z}_p$. These graphs were studied in~\cite{frucht} (see also~\cite{alspach1979,marusic1983,nedela,lovrecic}). Non-Cayley vertex-transitive graphs of order~\eqref{e48} were included in~\cite[Theorem~2,(c)]{praeger1994}. By considering a disjoint union of multiple copies of such graphs one obtain a non-Cayley vertex-transitive graph of order~\eqref{e45}-\eqref{e46}. Clearly, such graph is disconnected and consequently different from a complementary prism $\Gamma\bar{\Gamma}$. The disjoint union can also be replaced by some other graph operations, which yield connected graphs. Roughly speaking, if $\Gamma'$ is a connected vertex-transitive non-Cayley graph, and $\Gamma''$ is a connected vertex-transitive graph, then the lexicographic products $\Gamma''[\Gamma']$ and $\Gamma''[\overline{\Gamma'}]$ are connected vertex-transitive graphs, which are often non-Cayley. For example, if $\Gamma'$ is a connected graph and $m\geq 2$ is an integer, then $K_m[\overline{\Gamma'}]$ is vertex-transitive and non-Cayley if and only if $\Gamma'$ is vertex-transitive and non-Cayley. Here we skip the details and the proofs of these facts, since we do not rely on them. The bottom line is that even these graph operations cannot produce vertex-transitive non-Cayley graphs of the form $\Gamma\bar{\Gamma}$, as it follows from Proposition~\ref{p7}.

\begin{prop}\label{p7}
Let $\Gamma$ be any graph. Then $\Gamma\bar{\Gamma}$ is not isomorphic to a lexicographic product $\Gamma_1[\Gamma_2]$, where both
$\Gamma_1$ and $\Gamma_2$ have at least two vertices.
\end{prop}
\begin{proof}
Suppose that $\varphi$ is an isomorphism from $\Gamma_1[\Gamma_2]$ onto $\Gamma\bar{\Gamma}$, where $|V(\Gamma_i)|=:n_i\geq 2$ for both $i\in \{1,2\}$. Since $\Gamma\bar{\Gamma}$ is connected by Lemma~\ref{p6}, we deduce that $\Gamma_1$ is connected. For each $u\in V(\Gamma_1)$ denote $V(\Gamma_2)_u=\{(u,v): v\in V(\Gamma_2)\}$ and observe that
\begin{equation}\label{e110}
|\varphi(V(\Gamma_2)_{u})|\geq 2\qquad (u\in V(\Gamma_1)).
\end{equation}
We claim that for each such  $u$ we have
\begin{equation}\label{e49}
\varphi(V(\Gamma_2)_u)\subseteq W_1=W_1(\Gamma\bar{\Gamma})\quad \textrm{or}\quad \varphi(V(\Gamma_2)_u)\subseteq W_2=W_2(\Gamma\bar{\Gamma}).
\end{equation}

Assume that this is not true, i.e., there is $u\in V(\Gamma_1)$ such that
\begin{equation}\label{e50}
\varphi(V(\Gamma_2)_u)\cap W_1\neq \emptyset\quad \textrm{and}\quad \varphi(V(\Gamma_2)_u)\cap W_2\neq \emptyset.
\end{equation}
Since $n_1\geq 2$, there exists $u'\in V(\Gamma_1)$ that is adjacent to $u$ in $\Gamma_1$. Then
\begin{equation}\label{e51}
\textrm{each vertex in}\ \varphi(V(\Gamma_2)_{u'})\ \textrm{is adjacent to each vertex in}\ \varphi(V(\Gamma_2)_u).
\end{equation}
By~\eqref{e50}, there exist $w,w'\in V(\Gamma)$ such that  $(w,1),(w',2)\in \varphi(V(\Gamma_2)_u)$. Since vertices of the form $(w,1),(w',2)$ have at most one common neighbor, we get in contradiction with~\eqref{e51} and~\eqref{e110}.

Hence, the claim~\eqref{e49} is correct. Since $\varphi$ is surjective and $\Gamma_1$ is connected, there are adjacent $u,u'\in V(\Gamma_1)$ such that $\varphi(V(\Gamma_2)_u)\subseteq W_1$ and $\varphi(V(\Gamma_2)_{u'})\subseteq W_2$. Then \eqref{e51} still holds, but because of~\eqref{e110} it contradicts the fact that a vertex in $W_1$ is adjacent to precisely one vertex in $W_2$ and vice versa.
\end{proof}

\section{Hamiltonian properties of $\Gamma\bar{\Gamma}$ and its Cheeger number}

In this short section we provide few observations regarding Hamiltonian paths/cycles and the Cheeger number in a complementary prism.
Hamiltonicity of $\Gamma\bar{\Gamma}$ from complexity point of view was studied in~\cite{cp11}. Here our main focus is on the case, where $\Gamma\bar{\Gamma}$ is vertex-transitive. Recall that Hamiltonian properties of vertex-transitive graphs were extensively studied in the past (see for example~\cite{lovasz-problem,bermond,marusic1983b,babai-handbook,draganklavdija2009} and the references therein), while $K_1\overline{K_1}$ and $C_5\overline{C_5}$ are among the five known vertex-transitive graphs without a Hamiltonian cycle.

\begin{prop}\label{p8}
If $\Gamma$ is a regular self-complementary graph, then $\Gamma\bar{\Gamma}$ has a Hamiltonian path.
\end{prop}
\begin{proof}
Let $n=|V(\Gamma)|$. If $n=1$, then  $\Gamma\bar{\Gamma}$ is the complete graph on two vertices and the claim is obviously true. Let $n>1$. By Lemma~\ref{regularscHamiltonskost} there is a Hamiltonian cycle $v_1,v_2,\ldots,v_n,v_1$ in $\Gamma$. Similarly, there is a Hamiltonian cycle $w_1,w_2,\ldots,w_n,w_1$ in $\bar{\Gamma}$. By rotating the vertices in the second cycle we can achieve that $w_1=v_n$. Then $$(v_1,1), \ldots, (v_n,1),(w_1,2),\ldots, (w_n,2)$$ is a Hamiltonian path in  $\Gamma\bar{\Gamma}$.
\end{proof}

We remark that it follows from \cite[Theorem~1.1]{kuhn} and Lemma~\ref{DAM} that graphs from Proposition~\ref{p8} have a Hamiltonian cycle if they have sufficiently many vertices. Next we consider Hamiltonian-connectivity.

\begin{lemma}\label{l7}
Let $\Gamma$ be a graph on $n>1$ vertices such that both $\Gamma$ and $\bar{\Gamma}$ are Hamiltonian-connected. Then $\Gamma\bar{\Gamma}$ is Hamiltonian-connected.
\end{lemma}
\begin{proof}
From the assumptions it follows that $n>2$. Pick arbitrary vertices $(u,i)$ and $(v,j)$ in  $\Gamma\bar{\Gamma}$.

Suppose firstly that $i=1=j$. Pick a Hamiltonian path $$u=:u_1\sim_{\Gamma} u_2\sim_{\Gamma}\cdots \sim_{\Gamma} u_n:=v$$
in $\Gamma$ that joins vertices $u,v$, and a Hamiltonian path in $\bar{\Gamma}$ that joins vertices $u_1,u_2$, that is,
$$u_1=:w_1\sim_{\bar{\Gamma}} w_2\sim_{\bar{\Gamma}}\cdots \sim_{\bar{\Gamma}} w_n:=u_2.$$ Then
$$(u,1)=(u_1,1)\sim (w_1,2)\sim \cdots \sim (w_n,2)\sim (u_2,1)\sim \cdots \sim (u_n,1)=(v,1)$$
 is a Hamiltonian path in $\Gamma\bar{\Gamma}$ that start in $(u,i)$ and ends in $(v,j)$.

 If $i=2=j$, we obtain the required Hamiltonian path in a symmetrical way. Let $\{i,j\}=\{1,2\}$. We may assume that $i=1$ and $j=2$. Pick any vertex $z\in V(\Gamma)\backslash\{u,v\}$ and Hamiltonian paths of the form
 $$u=:z_1\sim_{\Gamma} z_2\sim_{\Gamma}\cdots \sim_{\Gamma} z_n:=z$$
 and
$$ z=:v_1\sim_{\bar{\Gamma}} v_2\sim_{\bar{\Gamma}}\cdots \sim_{\bar{\Gamma}} v_n:=v$$
in $\Gamma$ and $\bar{\Gamma}$, respectively. Then
$$(u,1)=(z_1,1)\sim \cdots \sim (z_n,1)\sim (v_1,2)\sim \cdots \sim (v_n,2)=(v,2)$$
 is a Hamiltonian path in $\Gamma\bar{\Gamma}$ that start in $(u,i)$ and ends in $(v,j)$.
\end{proof}

\begin{cor}\label{p9}
Let $\Gamma$ be a self-complementary graph on $n>5$ vertices, which is vertex-transitive or strongly regular. Then  $\Gamma\bar{\Gamma}$ is Hamiltonian-connected.
\end{cor}
\begin{proof}
The claim follows immediately from Lemmas~\ref{l7} and~\ref{hamconnected}.
\end{proof}
Corollaries~\ref{p9} and~\ref{c1} imply that $\Gamma\bar{\Gamma}$ is Hamiltonian-connected whenever it is vertex-transitive and non-isomorphic to the Petersen graph and to $K_2$.

In Mohar's paper~\cite{mohar2} it is mentioned that the Cheeger number of the Petersen graph is 1. As we prove below, this is true for each complementary prism~$\Gamma\bar{\Gamma}$, where $\Gamma\ncong K_2, \overline{K_2}$ is regular. More generally, we have the following.
\begin{prop}\label{cheeger}
If $\Gamma$ is a graph on $n$ vertices, then $h(\Gamma\bar{\Gamma})=1$ unless either $\Gamma$ or $\bar{\Gamma}$ has an edge $\{u,v\}$ such that $\delta(u)=1$ and $\delta(v)=n-1$ in which case  $h(\Gamma\bar{\Gamma})=\frac{n-1}{n}$.
\end{prop}
\begin{proof}
Let $\{S,T\}$ be any partition of $V(\Gamma\bar{\Gamma})$ such that $1\leq |S|\leq |T|$. In particular, $|S|\leq n$.

If we select $\{S,T\}=\{W_1,W_2\}$, we deduce that $\frac{e(S,T)}{|S|}=1$ and hence $h(\Gamma\bar{\Gamma})\leq 1$. Further, if there exists an edge $\{u,v\}\in E(\Gamma)$ such that $\delta_{\Gamma}(u)=1$ and $\delta_{\Gamma}(v)=n-1$, then the selection $S=\{(u,1)\}\cup \{(w,2) : w\in V(\Gamma)\backslash \{v\}\}$ and $T=S^c$ implies that $h(\Gamma\bar{\Gamma})\leq \frac{n-1}{n}$, since the edges in $\Gamma\bar{\Gamma}$ that contribute to $e(S,T)$ are precisely $\{(u,1),(v,1)\}$ and $\{(w,1),(w,2)\}$ with $w\in V(\Gamma)\backslash\{u,v\}$.
Symmetrically, if there exists an edge $\{u,v\}\in E(\bar{\Gamma})$ such that $\delta_{\bar{\Gamma}}(u)=1$ and $\delta_{\bar{\Gamma}}(v)=n-1$, then the inequality $h(\Gamma\bar{\Gamma})\leq \frac{n-1}{n}$ is deduced by selecting $S=\{(u,2)\}\cup \{(w,1) : w\in V(\Gamma)\backslash \{v\}\}$.

To deduce the reversed inequalities, write $S=S_1\cup S_2$, where $S_i=\{(v,i): v\in V_i\}\subseteq W_i$ and $V_i\subseteq V(\Gamma)$ for $i\in \{1,2\}$. Consider the partition $$\{V_1\backslash V_2, V_2\backslash V_1,V_1\cap V_2, V(\Gamma)\backslash (V_1\cup V_2)\}$$ of $V(\Gamma)$.
For each pair of distinct vertices $u,v\in V(\Gamma)$ we have either $u\sim_{\Gamma} v$ or $u\sim_{\bar{\Gamma}} v$. If we consider pairs, where
$u\in V_1\cap V_2$ and $v\in V(\Gamma)\backslash (V_1\cup V_2)$ or
$u\in V_1\backslash V_2$ and $v\in V_2\backslash V_1$, together with edges $\{(w,1),(w,2)\}$ with $w\in (V_1\backslash V_2)\cup (V_2\backslash V_1)$,
we deduce that
\begin{equation}\label{e105}
\frac{e(S,T)}{|S|}\geq \frac{|V_1\backslash V_2|+|V_2\backslash V_1|+|V_1\cap V_2|\cdot|V(\Gamma)\backslash (V_1\cup V_2)|+|V_1\backslash V_2|\cdot |V_2\backslash V_1|}{|V_1\backslash V_2|+|V_2\backslash V_1|+2 |V_1\cap V_2|}.
\end{equation}
Consequently,
$$\frac{e(S,T)}{|S|}\geq 1\ \textrm{if either}\ |V_1\cap V_2|=0\ \textrm{or}\ |V(\Gamma)\backslash (V_1\cup V_2)|\geq 2.$$

Suppose now that $|V_1\cap V_2|\geq 1$ and $|V(\Gamma)\backslash (V_1\cup V_2)|\in \{0,1\}$. Since $$n\geq |S|=|V_1\backslash V_2|+|V_2\backslash V_1|+2 |V_1\cap V_2|=n+|V_1\cap V_2|-|V(\Gamma)\backslash (V_1\cup V_2)|,$$
we deduce that $1\leq |V_1\cap V_2|\leq |V(\Gamma)\backslash (V_1\cup V_2)|\leq 1$, that is,
$$V_1\cap V_2=\{u\},\ V(\Gamma)\backslash (V_1\cup V_2)=\{v\},\ \textrm{and}\ |S|=n$$
for some $u,v\in V(\Gamma)$.
Moreover, the right-hand side of \eqref{e105} simplifies into
\begin{equation}\label{e106}
\frac{|V_1\backslash V_2|+|V_2\backslash V_1|+1+|V_1\backslash V_2|\cdot |V_2\backslash V_1|}{|V_1\backslash V_2|+|V_2\backslash V_1|+2},
\end{equation}
which is smaller than 1 if and only if
\begin{align}
\label{e108}|V_1\backslash V_2|=0\quad &\textrm{and}\quad |V_2\backslash V_1|=n-2\\
\nonumber &\textrm{or}\\
\label{e109}|V_2\backslash V_1|=0\quad &\textrm{and}\quad |V_1\backslash V_2|=n-2.
\end{align}
In these two cases \eqref{e106} equals $\frac{n-1}{n}$ and therefore $h(\Gamma\bar{\Gamma})\geq \frac{n-1}{n}$. Furthermore, since $|S|=n$, we have
$\frac{e(S,T)}{|S|}<1$ if and only if
\begin{equation}\label{e107}
e(S,T)=n-1.
\end{equation}
In the case of~\eqref{e108}, the equation~\eqref{e107} is equivalent to relations $u\nsim_{\Gamma} w$ and $v\nsim_{\bar{\Gamma}} w$ for all $w\in V_2\backslash V_1$. Since $u\sim_{\Gamma} v$ or $u\sim_{\bar{\Gamma}} v$, we get $\delta_{\Gamma}(u)=1$, $\delta_{\Gamma}(v)=n-1$ or $\delta_{\bar{\Gamma}}(v)=1$, $\delta_{\bar{\Gamma}}(u)=n-1$, respectively. Symmetrically, in the case of~\eqref{e109}, \eqref{e107} implies either $\{u,v\}\in E(\Gamma)$, $\delta_{\Gamma}(v)=1$, $\delta_{\Gamma}(u)=n-1$ or $\{u,v\}\in E(\bar{\Gamma})$, $\delta_{\bar{\Gamma}}(u)=1$, $\delta_{\bar{\Gamma}}(v)=n-1$, which ends the proof.
\end{proof}

\section{The core of $\Gamma\bar{\Gamma}$}\label{cores}

In this section we study the core of graph $\Gamma\bar{\Gamma}$. Its highlights are Theorems~\ref{thm-strg} and~\ref{thm-vt} that deal with a self-complementary graph $\Gamma$ that is strongly regular or vertex-transitive, respectively. Our first result is crucial and provides substantial information about $\textrm{core}(\Gamma\bar{\Gamma})$ for general graph $\Gamma$.
\begin{lemma}\label{lemma-core}
Let $\Gamma$ be any graph that is not isomorphic to $K_2$ and $\overline{K_2}$. If $\textrm{core}(\Gamma\bar{\Gamma})$ is any core of $\Gamma\bar{\Gamma}$, then one of the following five possibilities is true.
\begin{enumerate}
\item Graph $\Gamma\bar{\Gamma}$ is a core.
\item All vertices of $\textrm{core}(\Gamma\bar{\Gamma})$ are contained in $W_1$, in which case
$$\textrm{core}(\Gamma\bar{\Gamma})\cong\textrm{core}(\Gamma).$$
\item All vertices of $\textrm{core}(\Gamma\bar{\Gamma})$ are contained in $W_2$, in which case
$$\textrm{core}(\Gamma\bar{\Gamma})\cong\textrm{core}(\bar{\Gamma}).$$
\item There exists a partition $\{V_1,V_2,V_3\}$ of the vertex set $V(\Gamma)=V(\bar{\Gamma})$, where the sets $V_1$ and $V_2$ are nonempty, with the following properties.
\begin{enumerate}
\item The vertex set of $\textrm{core}(\Gamma\bar{\Gamma})$ equals
\begin{equation}\label{e82}
\{(v,2): v\in V_1\cup V_2\}\cup \{(v,1): v\in V_2\}.
\end{equation}
\item In $\Gamma$, there is no edge with one endpoint in $V_1$ and the other endpoint in $V_2$.
\item In $\Gamma$, each vertex in $V_3$ is adjacent to at most one vertex in $V_2$.
\item In $\Gamma$, there exists a vertex in $V_2$ that has no neighbors in $V_3$.
\item If $\Psi$ is any retraction from $\Gamma\bar{\Gamma}$ onto $\textrm{core}(\Gamma\bar{\Gamma})$, then
\begin{align}
\label{e76} &\Psi(\{(v,2) : v\in V_3\})\subseteq \{(v,2) : v\in V_1\},\\
\label{e77}&\Psi(\{(v,1) : v\in V_3\})\subseteq \{(v,2) : v\in V_1\cup V_2\},\\
\label{e78}&\Psi(\{(v,1) : v\in V_1\})\subseteq \{(v,2) : v\in V_1\cup V_2\}.
\end{align}
\end{enumerate}
\item There exists a partition $\{V_1,V_2,V_3\}$ of the vertex set $V(\Gamma)=V(\bar{\Gamma})$, where the sets $V_1$ and $V_2$ are nonempty,  with the following properties.
\begin{enumerate}
\item The vertex set of $\textrm{core}(\Gamma\bar{\Gamma})$ equals $$\{(v,1): v\in V_1\cup V_2\}\cup \{(v,2): v\in V_2\}.$$
\item In $\bar{\Gamma}$, there is no edge with one endpoint in $V_1$ and the other endpoint in $V_2$.
\item In $\bar{\Gamma}$, each vertex in $V_3$ is adjacent to at most one vertex in $V_2$.
\item In $\bar{\Gamma}$, there exists a vertex in $V_2$ that has no neighbors in $V_3$.
\item If $\Psi$ is any retraction from $\Gamma\bar{\Gamma}$ onto $\textrm{core}(\Gamma\bar{\Gamma})$, then
\begin{align*}
&\Psi(\{(v,1) : v\in V_3\})\subseteq \{(v,1) : v\in V_1\},\\
&\Psi(\{(v,2) : v\in V_3\})\subseteq \{(v,1) : v\in V_1\cup V_2\},\\
&\Psi(\{(v,2) : v\in V_1\})\subseteq \{(v,1) : v\in V_1\cup V_2\}.
\end{align*}
\end{enumerate}
\end{enumerate}
\end{lemma}

\begin{proof}
The vertex set of $\textrm{core}(\Gamma\bar{\Gamma})$ can be written as $$\{(u,1) : u\in U_1\}\cup \{(u,2): u\in U_2\},$$
where $U_1$, $U_2$ are two subsets of $V(\Gamma)$ and at least one of them is nonempty. If $U_2=\emptyset$, then the vertices of $\textrm{core}(\Gamma\bar{\Gamma})$ are contained in $W_1$ and any retraction of $\Gamma\bar{\Gamma}$ onto $\textrm{core}(\Gamma\bar{\Gamma})$ is a homomorphism from $\Gamma\bar{\Gamma}$ to $\Gamma$. Since $\Gamma$ is isomorphic to an induced subgraph in $\Gamma\bar{\Gamma}$ we have also the obvious homomorphism in the other direction. From Lemma~\ref{lemma-izomorfnostjeder} we infer that  $\textrm{core}(\Gamma\bar{\Gamma})\cong \textrm{core}(\Gamma)$, i.e. (ii) is true. If $U_1=\emptyset$, then we deduce that (iii) is true in the same way. In the sequel we assume that $U_1\neq \emptyset\neq U_2$. Moreover, Lemmas~\ref{p6} and~\ref{lemma-core-povezanost} imply that $\textrm{core}(\Gamma\bar{\Gamma})$ is connected. Hence, $U_1\cap U_2\neq \emptyset$. We now separate three cases.
\begin{myenumerate}{Case}
\item Let $U_1=U_2$. If both sets equal $V(\Gamma)$, then (i) is true. We next show that the inequality in $U_1=U_2\neq V(\Gamma)$ leads to a contradiction.

    \hspace{1em} Firstly, if $|U_1|=1$ , then  $|V(\Gamma)\backslash U_1|=1$ is ruled out by the assumption, since~$\Gamma$ is not isomorphic to $K_2$ and $\overline{K_2}$. Suppose that $|U_1|=1$ and $|V(\Gamma)\backslash U_1|\geq 2$. Pick arbitrary distinct vertices $v,v'\in V(\Gamma)\backslash U_1$ and denote the unique vertex in $U_1$ by $u$. Then either $v\sim_{\Gamma} v'$ or  $v\sim_{\bar{\Gamma}} v'$. In the former case,  $u$ cannot be $\Gamma$-adjacent to both $v$ and $v'$. In fact, the opposite would imply that $\Gamma$ contains a triangle, while $\textrm{core}(\Gamma\bar{\Gamma})\cong K_2$, a contradiction by Lemma~\ref{lemma-chromatic}. Hence, $\Gamma\bar{\Gamma}$ contains one of the following three $5$-cycles
\begin{align*}
&(u,1), (v,1), (v',1), (v',2), (u,2), (u,1),\\
&(u,1), (v',1), (v,1), (v,2), (u,2), (u,1),\\
&(u,2), (v,2), (v,1), (v',1), (v',2), (u,2),
\end{align*}
which implies a contradiction by Lemma~\ref{lemma-chromatic} in the same way as above. If $v\sim_{\bar{\Gamma}} v'$ we get a contradiction in a symmetrical way.

\hspace{1em} Next, observe that $|U_1|=2$ is not possible, since it would imply that $\textrm{core}(\Gamma\bar{\Gamma})$ is isomorphic to a path on four vertices, which is not a core.

\hspace{1em} Finally, assume that $|U_1|\geq 3$. Pick arbitrary $v\in V(\Gamma)\backslash U_1$ and any retraction~$\Psi$ from $\Gamma\bar{\Gamma}$ onto $\textrm{core}(\Gamma\bar{\Gamma})$. Suppose that $\Psi(v,1)=(u,1)$ for some $u\in U_1=U_2$. Then $(v,1)\nsim (u,1)$, i.e. $(v,2)\sim (u,2)$ and consequently, $(u,2)\sim \Psi(v,2)\sim \Psi(v,1)=(u,1)$, which is not possible, since $(u,1)$ and $(u,2)$ do not have common neighbors. Hence,
\begin{equation}\label{e74}
\Psi(v,1)=(u,2)
\end{equation}
for some $u\in U_1=U_2$. In symmetrical way we deduce that $\Psi(v,2)=(\dot{u},1)$ for some $\dot{u}\in U_1=U_2$. Since $\Psi(v,1)\sim \Psi(v,2)$ it follows that $\dot{u}=u$, i.e.
\begin{equation}\label{e75}
\Psi(v,2)=(u,1).
\end{equation}
Since $|U_1|\geq 3$, we can find distinct $\ddot{u},\dddot{u}\in U_1$ such that either $(\ddot{u},1)\sim (v,1)\sim (\dddot{u},1)$ or $(\ddot{u},2)\sim (v,2)\sim (\dddot{u},2)$. Consequently, either
$$(\ddot{u},1)\sim \Psi(v,1)\sim (\dddot{u},1)\quad \textrm{or}\quad  (\ddot{u},2)\sim \Psi(v,2)\sim (\dddot{u},2),$$
which contradicts either~\eqref{e74} or~\eqref{e75}.

\item\label{case2b} Let $U_2\backslash U_1\neq \emptyset$. We prove that in this case (iv) is true. Denote $V_1:=U_2\backslash U_1$ and $V_2:=U_1\cap U_2$. Then $V_1\neq \emptyset\neq V_2$ and $V_1\cap V_2=\emptyset$. Pick any retraction $\Psi$ from $\Gamma\bar{\Gamma}$ onto $\textrm{core}(\Gamma\bar{\Gamma})$.

    \hspace{1em} Suppose that (b) is not true, i.e. there are vertices $v_1\in V_1$ and $v_2\in V_2$ such that $v_1\sim_{\Gamma} v_2$.  Since $(v_2,1)$ and $(v_1,2)$ are vertices in $\textrm{core}(\Gamma\bar{\Gamma})$ and $(v_1,2)\sim (v_1,1)\sim (v_2,1)$ it follows that
    $$(v_1,2)\sim \Psi(v_1,1)\sim (v_2,1).$$ This is a contradiction, since $\Psi(v_1,1)\neq (v_1,1)$. In fact, $\Psi(v_1,1)$ is a vertex in $\textrm{core}(\Gamma\bar{\Gamma})$ while $(v_1,1)$ is not. Hence, the claim in (b) is correct.

    \hspace{1em} Next we prove~\eqref{e78}. If $v_1\in V_1$, then $(v_1,1)$ is not a vertex in $\textrm{core}(\Gamma\bar{\Gamma})$. Consequently $(v_1,1)\neq \Psi(v_1,1)\sim \Psi(v_1,2)=(v_1,2)$, which implies that
    \begin{equation}\label{e79}
    \Psi(v_1,1)\in \{(u,2) : u\in U_2\},
    \end{equation}
    i.e.~\eqref{e78}.

    \hspace{1em} To prove the claim (a), assume the contrary, i.e. $U_1\backslash U_2\neq \emptyset$. Then we show as in the previous two paragraphs that there are no edges in $\bar{\Gamma}$ with one endpoint in $V_1':=U_1\backslash U_2$ and the other endpoint in $V_2$, and for arbitrary $v_1'\in V_1'$ we have \begin{equation}\label{e80}
    \Psi(v_1',2)\in \{(u,1): u\in U_1\}.
    \end{equation}
    Let $v_1\in V_1$ be arbitrary. Recall that vertices $(v_1,1)$ and $(v_1',2)$ are not contained in  $\textrm{core}(\Gamma\bar{\Gamma})$, that is, $(v_1,1)\notin \{(u,1) : u\in U_1\}$ and $(v_1',2)\notin \{(u,2) : u\in U_2\}$. Since either
    $$\Psi(v_1,1)\sim \Psi(v_1',1)=(v_1',1)\quad \textrm{or}\quad \Psi(v_1',2)\sim \Psi(v_1,2)=(v_1,2),$$
    we get in contradiction with \eqref{e79} or \eqref{e80}, respectively. Hence, $U_1\backslash U_2=\emptyset$ and the claim (a) is true.

    \hspace{1em} Let $V_3:=V(\Gamma)\backslash U_2$. Then $\{V_1,V_2,V_3\}$ is a partition of $V(\Gamma)$.

    \hspace{1em} To prove (c) assume the contrary, that is, there exist $v_3\in V_3$ and distinct $v_2,v_2'\in V_2$ such that $v_2 \sim_{\Gamma} v_3\sim_{\Gamma} v_2'$. Then $(v_2,1)\sim (v_3,1)\sim (v_2',1)$ and consequently $W_1 \ni (v_2,1)=\Psi(v_2,1) \sim \Psi (v_3,1)\sim\Psi(v_2',1)=(v_2',1)\in W_1$. Therefore $\Psi (v_3,1)\in W_1$. By (a) it follows that $$\Psi (v_3,1)=(v_2'',1)$$ for some $v_2''\in V_2\backslash \{v_2,v_2'\}$. Consequently, $(v_2'',1)\nsim (v_3,1)$, i.e. $(v_2'',2)\sim (v_3,2)\sim (v_3,1)$ and therefore $(v_2'',2)\sim \Psi(v_3,2)\sim (v_2'',1)$, which is not possible, as $(v_2'',1)$ and $(v_2'',1)$ do not have common neighbors. Hence, (c) is true.

    \hspace{1em} Suppose that the inclusion \eqref{e76} is not true. Then by (a) there exist $v_3\in V_3$ and $v_2\in V_2$ such that $\Psi(v_3,2)=(v_2,i)$ for some
    $i\in \{1,2\}$. Assume firstly that $i=2$. Then, $(v_3,2)\nsim (v_2,2)$, i.e. $(v_3,1)\sim (v_2,1)$. Consequently
    $$(v_2,1)=\Psi(v_2,1)\sim \Psi(v_3,1)\sim \Psi(v_3,2)=(v_2,2),$$
    which is not possible. Assume now that $i=1$, that is,
    \begin{equation}\label{e84}
    \Psi(v_3,2)=(v_2,1).
    \end{equation}
    It follows from (c) that $(v_3,2)$ is adjacent to all vertices of the form $(v,2)$ where $v\in V_2$, with one possible exception. Since $\Psi(v,2)=(v,2)$ and $\Psi(v_3,2)\in W_1$, it follows that $|V_2|\in \{1,2\}$. Moreover, if $|V_2|=2$, then
    \begin{equation}\label{e83}
    (v_3,2)\ \textrm{is adjacent to precisely one vertex of the form}\ (v,2),\ v\in V_2.
    \end{equation}
    Suppose that $|V_2|=2$, i.e. $V_2=\{v_2,v\}$ for some $v\neq v_2$. Then \eqref{e84}-\eqref{e83} imply
    $$(v_3,2)\sim (v_2,2)\ \textrm{and}\ (v_3,1)\sim (v,1),$$
    since the opposite would yield $(v_2,1)=\Psi(v_3,2)\sim \Psi(v,2)=(v,2)$, which is not true. Consequently,
    $$(v_2,1)=\Psi(v_3,2)\sim \Psi(v_3,1)\sim \Psi(v,1)=(v,1).$$
    This is a contradiction, since $(v_2,1)$ and $(v,1)$ do not have common neighbors in $\textrm{core}(\Gamma\bar{\Gamma})$ as they are its unique vertices that are contained in $W_1$. On the other hand if $|V_2|=1$, i.e. $V_2=\{v_2\}$, then the graph with the vertex set~\eqref{e82} is not a core. In fact, in such a case a map $\dot{\Psi}$, which satisfies $\dot{\Psi}(v_2,1)=(v_1,2)$ for arbitrary $v_1\in V_1$ and fixes all other vertices of $\textrm{core}(\Gamma\bar{\Gamma})$ is a non-bijective endomorphism of $\textrm{core}(\Gamma\bar{\Gamma})$. Hence the inclusion \eqref{e76} is correct.

    \hspace{1em} If~\eqref{e77} is not true, then there are $v_3\in V_3$, $v_2\in V_2$ such that $\Psi(v_3,1)=(v_2,1)$. Since $\Psi(v_3,2)\sim \Psi(v_3,1)$, we get in contradiction with \eqref{e76}, and~\eqref{e77} is correct.

    \hspace{1em} It remains to prove the statement (d). Suppose that each vertex $v_2\in V_2$ has at least one $\Gamma$-neighbor in the set $V_3$. Fix one of them and denote it by $f(v_2)$. Choose an arbitrary retraction $\Psi$ from $\Gamma\bar{\Gamma}$ onto $\textrm{core}(\Gamma\bar{\Gamma})$. Define a map $\ddot{\Psi}$ on the vertex set of $\textrm{core}(\Gamma\bar{\Gamma})$, which is described in~\eqref{e82}, by
    \begin{align*}
    \ddot{\Psi}(v,i):=\left\{\begin{array}{ll}(v,i) & \textrm{if}\ v\in V_1\cup V_2\ \textrm{and}\ i=2,\\
    \Psi(f(v),2) & \textrm{if}\ v\in V_2\ \textrm{and}\ i=1.\end{array}\right.
    \end{align*}
    We claim that $\ddot{\Psi}$ is a non-bijective endomorphism of $\textrm{core}(\Gamma\bar{\Gamma})$. This gives the desired contradiction that proves (d). Pick any $v_2\in V_2$. Then~\eqref{e76} implies that $\Psi(f(v_2),2)=(v_1,2)$ for some $v_1\in V_1$. Consequently
    $$\ddot{\Psi}(v_2,1)=\Psi(f(v_2),2)=(v_1,2)=\ddot{\Psi}(v_1,2),$$
    which shows that the map $\ddot{\Psi}$ is not bijective. Moreover, since
    $$(v_2,1)=\Psi(v_2,1)\sim \Psi(f(v_2),1),$$
    it follows from~\eqref{e77} that each $v_2\in V_2$ satisfies
    \begin{equation}\label{e85}
    \Psi(f(v_2),1)=(v_2,2).
    \end{equation}
    To prove that $\ddot{\Psi}$ is an endomorphism, suppose that $(v,i)\sim (u,j)$ are two adjacent vertices in $\textrm{core}(\Gamma\bar{\Gamma})$. If $i=2=j$, then $\ddot{\Psi}(v,i)=(v,i)\sim (u,j)=\ddot{\Psi}(u,j)$. If $\{i,j\}=\{1,2\}$, say $i=1$ and $j=2$, then $v=u\in V_2$ and~\eqref{e85} implies that
    $$\ddot{\Psi}(u,j)=\ddot{\Psi}(v,2)=(v,2)=\Psi(f(v),1)\sim \Psi(f(v),2)=\ddot{\Psi}(v,1)=\ddot{\Psi}(v,i).$$
    If $i=1=j$, then $v,u\in V_2$, and~\eqref{e85} implies that $$\Psi(f(v),1)=(v,2)\nsim (u,2)=\Psi(f(u),1).$$ Hence,
    $(f(v),1)\nsim (f(u),1)$, i.e. $(f(v),2)\sim (f(u),2)$. Consequently,
    $$\ddot{\Psi}(v,i)=\ddot{\Psi}(v,1)=\Psi(f(v),2)\sim \Psi(f(u),2)=\ddot{\Psi}(u,1)=\ddot{\Psi}(u,j),$$
    which shows that $\ddot{\Psi}$ is an endomorphism of $\textrm{core}(\Gamma\bar{\Gamma})$.

\item Let $U_1\backslash U_2\neq \emptyset$. Then we proceed symmetrically as in Case~\ref{case2b} to deduce that (v) is true.\qedhere
\end{myenumerate}
\end{proof}

\begin{remark}
The set $V_3$ in Lemma~\ref{lemma-core} is allowed to be empty. In this case either $\Gamma$ or $\bar{\Gamma}$ is disconnected.
\end{remark}
\begin{exa}\label{ex:misteriozen}
The possibilities~(iv),(v) in Lemma~\ref{lemma-core} can occur even if the graph $\Gamma$ is regular. For example, if $\Gamma$ is a (disjoint) union of a triangle and a pentagon, then the core of $\Gamma\bar{\Gamma}$ is the graph, which is induced by all vertices in~$W_2$ together with the triangle part of $W_1$. The verification of this fact is left to the reader. On the other hand it seems mysteriously complicated to provide an example of a regular graph $\Gamma$, where both $\Gamma$ and $\bar{\Gamma}$ are connected and the possibility~(iv) or (v) occur for $\textrm{core}(\Gamma\bar{\Gamma})$. Below we provide an example of such graph $\Gamma$ on 505 vertices that is 194-regular. Consider the complement of the complete graph on 99 vertices, $\overline{K_{99}}$, the graph $\overline{K(10,4)} - e$ that we described in \eqref{e86}-\eqref{e87}, and the Cayley graph $\textrm{Cay}(\mathbb{F}_{49}\times \mathbb{F}_4, S)$, where $\mathbb{F}_{49}$ and $\mathbb{F}_{4}=\{0,1,\imath,1+\imath\}$ are finite fields with 49 and 4 elements, respectively, $\mathbb{F}_{49}\times \mathbb{F}_4$ is the corresponding additive group, and $S=\{(x,y): x\in \mathbb{F}_{49}^{\ast}, y\in\{0,1\}\}$. Here, $\mathbb{F}_{49}^{\ast}=\mathbb{F}_{49}\backslash \{0\}$. The graph $\textrm{Cay}(\mathbb{F}_{49}\times \mathbb{F}_4, S)$ has two connected components, namely
$${\cal C}_1=\{(x,y): x\in \mathbb{F}_{49}, y\in \{0,1\}\}\ \textrm{and}\ {\cal C}_2=\{(x,y): x\in \mathbb{F}_{49}, y\in \{\imath,1+\imath\}\}.$$
Fix any $(x_1,y_1)\in {\cal C}_1$ and $(x_2,y_2)\in {\cal C}_2$ and number the vertices $\{w_1,\ldots,w_{196}\}$ of $\textrm{Cay}(\mathbb{F}_{49}\times \mathbb{F}_4, S)$ is such way that $w_{195}=(x_1,y_1)$ and $w_{196}=(x_2,y_2)$. Further, number the vertices $\{v_1,\ldots,v_{99}\}$ of $\overline{K_{99}}$  in some order and let $u_1,u_2$ be the two endpoints of the deleted edge $e$ in $\overline{K(10,4)} - e$. We form the graph $\Gamma$ from the disjoint union of graphs $\overline{K_{99}}$, $\overline{K(10,4)} - e$, $\textrm{Cay}(\mathbb{F}_{49}\times \mathbb{F}_4, S)$ by adding the following edges:
\begin{itemize}
\item $\{w_{195},u_1\}$, $\{w_{196},{u_2}\}$,
\item $\{w_{195},v_i\}$, $\{w_{196},v_i\}$ where $i\in \{1,\ldots,97\}$,
\item $\{w_j,v_{98}\}$, $\{w_j,v_{99}\}$ where $j\in \{1,\ldots,194\}$,
\item $\{v_i,w_{s}\}$ where $i\in \{1,\ldots,97\}$ and $s\in \{1,\ldots,194\}\backslash\{i,i+97\}$.
\end{itemize}
Observe that the graph $\Gamma$ is 194-regular. The connectedness of $\overline{K(10,4)} - e$ imply the connectedness of $\Gamma$. The connectedness of $\bar{\Gamma}$ is obvious.  We next prove that the vertices of a core of $\Gamma\bar{\Gamma}$ are contained in the set $X_1\cup X_2$, where
\begin{align*}
X_1&=\left\{(u,1) : u\in V\big(\overline{K(10,4)} - e\big)\right\},\\
X_2&=\left\{(u,2) : u\in V\big(\overline{K(10,4)} - e\big)\right\}\cup \{(v_1,2),\ldots,(v_{99},2)\},
\end{align*}
and none of the possibilities (i), (ii), (iii) of Lemma~\ref{lemma-core} can occur. Observe that the sets
$$C=\{(x,0) : x\in \mathbb{F}_{49}\}\quad \textrm{and}\quad I=\{(0,y) : y\in \mathbb{F}_{4}\}$$
are a clique and an independent set in $\textrm{Cay}(\mathbb{F}_{49}\times \mathbb{F}_4, S)$, respectively. In fact, by Lemma~\ref{clique-coclique} their orders are the largest possible. Consequently, the maps $p_1: (x,y)\mapsto (x,0)$ and $p_2: (x,y)\mapsto (0,y)$ are endomorphisms of $\textrm{Cay}(\mathbb{F}_{49}\times \mathbb{F}_4, S)$ and $\overline{\textrm{Cay}(\mathbb{F}_{49}\times \mathbb{F}_4, S)}$ onto their maximum cliques, respectively (compare with~\cite[Proposition~1]{JCTA}). Let $f_1: C\to\{v_1,\ldots,v_{49}\}$ and $f_2: I\to\{v_{50},v_{51},v_{52},v_{53}\}$ be any bijections and consider the map $\Psi$ on $\Gamma\bar{\Gamma}$, which fixes the vertices in $X_1\cup X_2$ and satisfies
\begin{align*}
\Psi(w_j,1)&=\left\{\begin{array}{ll}\big(f_1(p_1(w_j)),2\big)& \textrm{if}\ j\in\{1,\ldots,194\},\\
(u_1,2)& \textrm{if}\ j=195,\\
(u_2,2)& \textrm{if}\ j=196,\\
\end{array}\right.\\
\Psi(w_j,2)&=\left\{\begin{array}{ll}\big(f_2(p_2(w_j)),2\big)& \textrm{if}\ j\notin\{50,\ldots,53,147,\ldots,150\},\\
(v_{54},2)& \textrm{if}\ j=50,\\
(v_{55},2)& \textrm{if}\ j=51,\\
(v_{56},2)& \textrm{if}\ j=52,\\
(v_{57},2)& \textrm{if}\ j=53,\\
(v_{58},2)& \textrm{if}\ j=147,\\
(v_{59},2)& \textrm{if}\ j=148,\\
(v_{60},2)& \textrm{if}\ j=149,\\
(v_{61},2)& \textrm{if}\ j=150,
\end{array}\right.,\\
\Psi(v_i,1)&=\left\{\begin{array}{ll}(v_{62},2)&\qquad\quad\ \; \textrm{if}\ i\neq 62,\\
(v_{63},2)&\qquad\quad\ \; \textrm{if}\ i=62.
\end{array}\right.
\end{align*}
It is straightforward to check that $\Psi$ is a retraction. Since it is not bijective, the claim (i) in Lemma~\ref{lemma-core} is not true. Moreover, Remark~\ref{remark} implies that  $$\textrm{core}(\Gamma\bar{\Gamma})\cong\textrm{core}(\langle X_1\cup X_2\rangle),$$ where $\langle X_1\cup X_2\rangle$ is the subgraph, which is induced by the set~$X_1\cup X_2$. If the claim (ii) in Lemma~\ref{lemma-core} is true, then
$$\textrm{core}(\Gamma)\cong\textrm{core}(\langle X_1\cup X_2\rangle)$$
and consequently Lemmas~\ref{lemma-izomorfnostjeder}, \ref{lemma-clique}, and \eqref{e86}-\eqref{e87} imply that
\begin{align*}
101&\leq \omega(K_{99})+\omega(K(10,4) + e)=\omega(\langle X_2\rangle)\\
&\leq \omega(\langle X_1\cup X_2\rangle)=\omega(\Gamma)=\omega(\overline{K(10,4)} - e)\leq 84,
\end{align*}
a contradiction.
Next observe that $\Psi$ maps $W_2$ into $\langle X_2\rangle$. Hence, Lemma~\ref{lemma-chromatic} implies that
\begin{equation}\label{e88}
\chi(\bar{\Gamma})\leq \chi(\langle X_2\rangle).
\end{equation}
If the claim (iii) in Lemma~\ref{lemma-core} is true, then
$$\textrm{core}(\bar{\Gamma})\cong\textrm{core}(\langle X_1\cup X_2\rangle)$$
and consequently Lemmas~\ref{lemma-izomorfnostjeder}, \ref{lemma-chromatic},  \eqref{e86}-\eqref{e87}, and \eqref{e88} imply that
\begin{align*}
104&\leq \chi(\overline{K(10,4)} - e)= \chi(\langle X_1\rangle)
\leq \chi(\langle X_1\cup X_2\rangle)\\
&=\chi(\bar{\Gamma})\leq \chi(\langle X_2\rangle)=\chi(K_{99})+\chi(K(10,4) + e)=103,
\end{align*}
a contradiction.\hfill$\Box$
\end{exa}

\begin{cor}\label{c3}
Let $\Gamma$ be any graph on $n$ vertices that is $\big(\frac{n-1}{2}\big)$-regular. Then only the statements (i), (ii), (iii) in Lemma~\ref{lemma-core} are possible.
\end{cor}
\begin{proof}
Suppose the claim (iv) from Lemma~\ref{lemma-core} is true. Then $V\big(\textrm{core}(\Gamma\bar{\Gamma})\big)$ equals~\eqref{e82}. Let $n_i=|V_i|$ for $i=1,2,3$. Then $n_1,n_2\geq 1$. By statement (b), each vertex in $V_1$ is adjacent to each vertex in $V_2$, viewed in the graph $\bar{\Gamma}$, which is  $\big(\frac{n-1}{2}\big)$-regular. Hence, $n_2\leq \frac{n-1}{2}$. Consequently, each vertex $v\in V_2$ has some neighbor $f(v)\in V_3$, viewed in graph $\Gamma$. In fact, the opposite would imply that~$v$ has at least $n_1+n_3=n-n_2\geq n-\frac{n-1}{2}=\frac{n+1}{2}$ neighbors in $\bar{\Gamma}$, which is not possible. By statement (c), the map $f: V_2\to V_3$ is injective. Hence,
\begin{equation}\label{e89}
n_3\geq n_2.
\end{equation}
Moreover, each vertex $v\in V_2$ has zero and at most $n_2-1$ $\Gamma$-neighbors in $V_1$ and $V_2$, respectively. Consequently, it has at least $\frac{n-1}{2}-(n_2-1)$ $\Gamma$-neighbors in $V_3$. By (c) it follows that $n_3\geq \big(\frac{n-1}{2}-(n_2-1)\big)n_2$. Since $n=n_1+n_2+n_3$, we deduce that
\begin{equation}\label{e90}
(n_2-n_1-1)n_2\geq (n_2-2)n_3.
\end{equation}
Further observe that $$n_2\geq 2,$$ since the equality $n_2=1$ would imply that the graph with the vertex set~\eqref{e82} is not a core. We claim also that
\begin{equation}\label{e91}
n_1\geq 2.
\end{equation}
If $n_1=1$, then we infer from the inclusion~\eqref{e76} that in graph $\bar{\Gamma}$ there are no edges inside the set $V_3$ and there are no edges connecting the vertex in $V_1$ with some vertex in $V_3$. Consequently, we deduce by the statement (b) in Lemma~\ref{lemma-core} (iv) that a vertex $f(v)\in V_3$ as above, has in graph $\Gamma$ (at least) $$(n_3-1)+1+1\geq \left(\frac{n-1}{2}-1\right)+1+1=\frac{n+1}{2}$$
neighbors. This is a contradiction that proves the bound~\eqref{e91}. Consequently, the equality $n_2=2$ is not possible by~\eqref{e90}-\eqref{e91}, while the inequality $n_2>2$ transforms \eqref{e90} into $$\frac{n_2-n_1-1}{n_2-2} n_2\geq n_3,$$ which contradicts \eqref{e89}.

In the same way we see that the claim (v) in Lemma~\ref{lemma-core} is not true.
\end{proof}

\begin{exa}\label{exa1}
In Corollary~\ref{c3} the possibility (ii) or (iii) may occur even if a $(\frac{n-1}{2})$-regular graph is self-complementary.
Consider the graph $\Gamma$ on the vertex set $\{1,2,\ldots,13\}$ with the adjacency matrix as in Figure~\ref{f10}. The map $\sigma: \Gamma\to \bar{\Gamma}$, which maps vertices $1,2,\ldots,13$ into
$$1,8,9,10,11,12,13,5,4,3,2,7,6,$$
respectively, is  an antimorphism. Any core of the graph $\Gamma\bar{\Gamma}$ is isomorphic to the complete graph on five vertices. For example, vertices $$(1,1), (2,1), (3,1), (4,1), (5,1)$$ form such a core $C$, yielding the possibility (ii). In fact, if we denote
\begin{align*}
U_1&=\{(1,1), (8,1), (9,1), (10,1), (4,2), (5,2), (11,2)\},\\
U_2&=\{(2,1), (6,1), (7,1), (1,2), (3,2)\},\\
U_3&=\{(3,1), (11,1), (2,2), (8,2)\},\\
U_4&=\{(4,1), (12,1), (7,2), (9,2), (13,2)\},\\
U_5&=\{(5,1), (13,1), (6,2), (10,2), (12,2)\},
\end{align*}
then the map $\Psi$ on $V(\Gamma\bar{\Gamma})$, which maps vertices from the set $U_i$ into $(i,1)$, for all $i\in \{1,2,3,4,5\}$, is a retraction onto $C$. If we replace edges $\{3,9\}$, $\{4,10\}$ in graph $\Gamma$ by edges $\{3,10\}$, $\{4,9\}$, we obtain a graph $\Gamma'$, which has the same properties as $\Gamma$. Namely, $\sigma$ is an antimorphism of $\Gamma'$ and $\Psi$ is a retraction onto the clique $C$. Both graphs $\Gamma$, $\Gamma'$ were found with the help of the database~\cite{repository2}.\hfill$\Box$
\begin{figure}
\centering
\begin{tabular}{c}
{\tiny $
\left(
\begin{array}{ccccccccccccc}
 0 & 1 & 1 & 1 & 1 & 1 & 1 & 0 & 0 & 0 & 0 & 0 & 0 \\
 1 & 0 & 1 & 1 & 1 & 0 & 0 & 1 & 1 & 0 & 0 & 0 & 0 \\
 1 & 1 & 0 & 1 & 1 & 0 & 0 & 1 & 1 & 0 & 0 & 0 & 0 \\
 1 & 1 & 1 & 0 & 1 & 0 & 0 & 0 & 0 & 1 & 1 & 0 & 0 \\
 1 & 1 & 1 & 1 & 0 & 0 & 0 & 0 & 0 & 1 & 1 & 0 & 0 \\
 1 & 0 & 0 & 0 & 0 & 0 & 0 & 1 & 1 & 1 & 1 & 1 & 0 \\
 1 & 0 & 0 & 0 & 0 & 0 & 0 & 1 & 1 & 1 & 1 & 0 & 1 \\
 0 & 1 & 1 & 0 & 0 & 1 & 1 & 0 & 0 & 0 & 0 & 1 & 1 \\
 0 & 1 & 1 & 0 & 0 & 1 & 1 & 0 & 0 & 0 & 0 & 1 & 1 \\
 0 & 0 & 0 & 1 & 1 & 1 & 1 & 0 & 0 & 0 & 0 & 1 & 1 \\
 0 & 0 & 0 & 1 & 1 & 1 & 1 & 0 & 0 & 0 & 0 & 1 & 1 \\
 0 & 0 & 0 & 0 & 0 & 1 & 0 & 1 & 1 & 1 & 1 & 0 & 1 \\
 0 & 0 & 0 & 0 & 0 & 0 & 1 & 1 & 1 & 1 & 1 & 1 & 0
\end{array}
\right)
$}\\
\\
\includegraphics[width=0.8\textwidth]{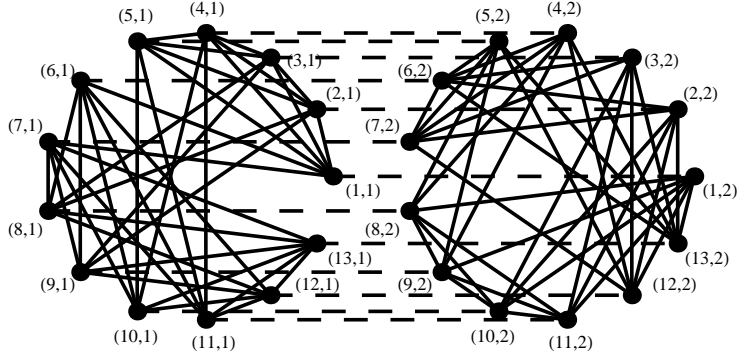}
\end{tabular}
\caption{The adjacency matrix of the graph $\Gamma$ in Example~\ref{exa1} and the graph $\Gamma\bar{\Gamma}$.}
\label{f10}
\end{figure}
\end{exa}

\begin{cor}\label{c4}
Let $\Gamma$ be any graph on $n$ vertices, which is not isomorphic to $K_2$, $\overline{K_2}$, $P_3$, and $\overline{P_3}$. If
$\textrm{core}(\Gamma\bar{\Gamma})$ is regular, then only the statements (i), (ii), (iii) in Lemma~\ref{lemma-core} are possible.
\end{cor}
\begin{proof}
Suppose the claim (iv) from Lemma~\ref{lemma-core} is true. Let $\bar{\Gamma}_1,\Gamma_2,\bar{\Gamma}_2$ be the subgraphs in $\Gamma\bar{\Gamma}$, which are induced by the vertex sets
\begin{equation}\label{e92}
\{(v,2): v\in V_1\},\quad\{(v,1): v\in V_2\},\quad \{(v,2): v\in V_2\},
\end{equation}
respectively. Since $\textrm{core}(\Gamma\bar{\Gamma})$ is regular, it follows from \eqref{e82} and the statement~(b) in Lemma~\ref{lemma-core} (iv) that all three graphs $\bar{\Gamma}_1,\Gamma_2,\bar{\Gamma}_2$ are regular. If $\bar{k}_1, n_1$ and $\bar{k}_2, n_2$ are the degree and the order of $\bar{\Gamma}_1$ and $\bar{\Gamma}_2$, respectively, then it follows from~\eqref{e82} that the degree of a vertex in $\Gamma\bar{\Gamma}$ equals
$$\bar{k}_1+n_2=(n_2-\bar{k}_2-1)+1=\bar{k}_2+n_1+1,$$
depending on the part~\eqref{e92} of the vertex set that it belongs. From the first equality we deduce that $\bar{k}_1+\bar{k}_2=0$ and hence $\bar{k}_1=0=\bar{k}_2$. The second equality yields $n_2=n_1+1$. Consequently, $\bar{\Gamma}_1\cong \overline{K_{n_1}}$, $\Gamma_2\cong K_{n_1+1}$, $\bar{\Gamma}_2\cong \overline{K_{n_1+1}}$. If $\Psi$ is a retraction from $\Gamma\bar{\Gamma}$ onto $\textrm{core}(\Gamma\bar{\Gamma})$, then a vertex in the set $\{(v,1): v\in V_1\}$ cannot be mapped into $\{(v,2): v\in V_1\}$ by $\Psi$, since $\bar{\Gamma}_1\cong \overline{K_{n_1}}$. Consequently, it follows from~\eqref{e78} that $n_1=1$, $\textrm{core}(\Gamma\bar{\Gamma})$ is isomorphic to a pentagon, and $$\Psi(v_1,1)=(v_2,2)$$ for some $v_2\in V_2$, where  $V_1=\{v_1\}$. Since $\Gamma\ncong\overline{P_3}$, it follows that $V_3\neq \emptyset$. Let $w\in V_3$. From~\eqref{e76} we deduce that
$\Psi(w,2)=(v_1,2)$. In particular, $(v_1,2)\nsim (w,2)$, i.e. $(v_1,1)\sim (w,1)\sim (w,2)$. Consequently, $ (v_2,2), \Psi(w,1), (v_1,2)$ form a triangle in the pentagon $\textrm{core}(\Gamma\bar{\Gamma})$, a contradiction.

In the same way we see that the claim (v) in Lemma~\ref{lemma-core} can be true only if~$\Gamma$ is isomorphic to the path on three vertices.
\end{proof}

\begin{thm}\label{thm-strg}
If $\Gamma$ is a strongly regular self-complementary graph, then $\Gamma\bar{\Gamma}$ is a core.
\end{thm}
\begin{proof}
The claim is obvious if $n:=|V(\Gamma)|=1.$ Let $n>1$. By Corollary~\ref{c3} only the statements (i), (ii), (iii) in Lemma~\ref{lemma-core} could be possible. Since $\Gamma$ is self-complementary it suffices to prove that (iii) leads to a contradiction. Hence, assume that $\textrm{core}(\Gamma\bar{\Gamma})\cong\textrm{core}(\bar{\Gamma})$. By Lemma~\ref{lemma-izomorfnostjeder} there exists a homomorphism between $\Gamma\bar{\Gamma}$ and $\bar{\Gamma}$. Hence, Lemma~\ref{lemma-theta} implies that $\vartheta\big(\overline{\Gamma\bar{\Gamma}}\big)\leq \vartheta(\Gamma)$ and  Lemma~\ref{lemma-thetaproduct} yields
\begin{equation}\label{e93}
\frac{2n}{\vartheta(\Gamma\bar{\Gamma})}\leq \vartheta(\Gamma).
\end{equation}
Recall that graphs $\Gamma$ and $\Gamma\bar{\Gamma}$ are regular of degree $\frac{n-1}{2}$ and $\frac{n+1}{2}$, respectively. Moreover, by Lemma~\ref{strg} and Corollary~\ref{t1}, their smallest eigenvalues equal $\frac{-\sqrt{n}-1}{2}$ and $\frac{-\sqrt{n+4}-1}{2}$, respectively. Hence, if we apply Lemma~\ref{lemma-thetabound} on both sides of~\eqref{e93} we deduce that
$$\frac{2n}{\frac{2n}{1-\frac{\frac{n+1}{2}}{\frac{-\sqrt{n+4}-1}{2}}}}\leq \frac{n}{1-\frac{\frac{n-1}{2}}{\frac{-\sqrt{n}-1}{2}}},$$
which simplifies into
$$n+1\leq (\sqrt{n}-1)(\sqrt{n+4}+1).$$
However, the inequality between the geometric and the arithmetic mean yields
\begin{align*}
(\sqrt{n}-1)(\sqrt{n+4}+1)&=\sqrt{n(n+4)}+\sqrt{n}-\sqrt{n+4}-1\\
&<\sqrt{n(n+4)}-1\\
&\leq \frac{n+(n+4)}{2}-1=n+1,
\end{align*}
a contradiction.
\end{proof}

\begin{cor}\label{c5}
Let $\Gamma$ be a vertex-transitive self-complementary graph. Then only the statements (i), (ii), (iii) in Lemma~\ref{lemma-core} are possible.
\end{cor}
\begin{proof}
The claim follows directly from Corollary~\ref{c3}. The same claim follows also from Corollary~\ref{c4}, Lemma~\ref{lemma-core-vt}, and the fact that the order of a regular self-complementary graph equals 1 modulo 4.
\end{proof}

\begin{lemma}\label{lema-core-not-complete}
If $\Gamma$ is a vertex-transitive self-complementary graph on $n>1$ vertices, then $\textrm{core}(\Gamma\bar{\Gamma})$ is not a complete graph.
\end{lemma}
\begin{proof}
We need to prove that $\chi(\Gamma\bar{\Gamma})>\omega(\Gamma\bar{\Gamma})$. Since $n>1$ and $\Gamma$ is both self-complementary and vertex-transitive, Lemma~\ref{clique-coclique} implies that
$$\omega(\Gamma\bar{\Gamma})=\max\{\alpha(\Gamma),\omega(\Gamma)\}=\omega(\Gamma)\leq \sqrt{n}.$$
Let $I$ be any independent set in $\Gamma\bar{\Gamma}$. Then $I$ is a disjoint union of sets $I_1\subseteq W_1$ and $I_2\subseteq W_2$. If we write $I_i=\{(u,i) : u\in J_i\}$ for $i\in \{1,2\}$,  where $J_1, J_2\subseteq V(\Gamma)$, then
\begin{equation}\label{e94}
J_1\cap J_2=\emptyset.
\end{equation}
Since $J_1$ and $J_2$ are independent set and a clique in $\Gamma$, respectively, we have $|J_1|\leq \alpha(\Gamma)=\omega(\Gamma)\leq \sqrt{n}$ and $|J_2|\leq \omega(\Gamma)=\sqrt{n}$, while Lemma~\ref{clique-coclique} and \eqref{e94} imply that $|J_1|\cdot |J_2|<n$. Hence,
$|I|=|J_1|+|J_2|< 2\sqrt {n}$, and therefore
$$\chi(\Gamma\bar{\Gamma})\geq \frac{|V(\Gamma\bar{\Gamma})|}{\alpha(\Gamma\bar{\Gamma})}>\frac{2n}{2\sqrt{n}}=\sqrt{n}\geq \omega(\Gamma\bar{\Gamma}).\qedhere$$
\end{proof}

\begin{thm}\label{thm-vt}
Let $\Gamma$ be a vertex-transitive self-complementary graph. If $\Gamma$ is either a core or its core is a complete graph, then $\Gamma\bar{\Gamma}$ is a core.
\end{thm}
\begin{proof}
By Corollary~\ref{c5}, only the statements (i), (ii), (iii) in Lemma~\ref{lemma-core} are possible. Suppose that (iii) is correct, that is,
\begin{equation}\label{e95}
V(\textrm{core}(\Gamma\bar{\Gamma}))\subseteq W_2
\end{equation}
and
\begin{equation}\label{e96}
\textrm{core}(\Gamma\bar{\Gamma})\cong \textrm{core}(\bar{\Gamma}).
\end{equation}
If the core of $\Gamma$ is a complete graph, then the self-complementarity and \eqref{e96} yields a contradiction by Lemma~\ref{lema-core-not-complete}. Suppose now that $\Gamma$ is a core. Then the self-complementarity and \eqref{e96} imply
that $\textrm{core}(\Gamma\bar{\Gamma})\cong \bar{\Gamma}$. Hence, \eqref{e95} yields
\begin{equation}\label{e97}
V(\textrm{core}(\Gamma\bar{\Gamma}))=W_2.
\end{equation}
Let
$\psi_1(v)=(v,1)$; $(v\in V(\Gamma))$ be the canonical isomorphism between $\Gamma$ and the subgraph in $\Gamma\bar{\Gamma}$, which is induced by the set $W_1$. Similarly, let $\psi_2(v)=(v,2)$; $(v\in V(\Gamma))$ be the canonical isomorphism between $\bar{\Gamma}$ and the subgraph induced by $W_2$. If $\Psi$ is any retraction from $\Gamma\bar{\Gamma}$ onto $\textrm{core}(\Gamma\bar{\Gamma})$, and $\sigma$ is any antimorphism between $\bar{\Gamma}$ and $\Gamma$, then the composition $\sigma \circ \psi_2^{-1}\circ (\Psi|_{W_1})\circ \psi_1$ is an endomorphism of $\Gamma$. Since $\Gamma$ is a core, the restriction $\Psi|_{W_1}$ is an isomorphism between the subgraphs in $\Gamma\bar{\Gamma}$ that are induced by the sets $W_1$ and $W_2$, respectively. Consequently $\psi_2^{-1}\circ (\Psi|_{W_1})\circ \psi_1: \Gamma\to\bar{\Gamma}$ is an antimorphism. By Lemma~\ref{fiksna_tocka}, there exists $v\in V(\Gamma)$ such that $\big(\psi_2^{-1}\circ (\Psi|_{W_1})\circ \psi_1\big)(v)=v$. Consequently, $\Psi(v,1)=(\Psi|_{W_1})(v,1)=(v,2)$. Since $\Psi$ is a retraction, \eqref{e97} implies that $\Psi(v,2)=(v,2)$. Since $\{(v,1), (v,2)\}$ is an edge in $\Gamma\bar{\Gamma}$, we have a contradiction.

In the same way we see that (ii) in Lemma~\ref{lemma-core} is not possible.
\end{proof}
\begin{remark}
The proof of Theorem~\ref{thm-vt} can be viewed as an additional proof of Theorem~\ref{thm-strg} if the following modifications are made: Corollary~\ref{c5} is replaced by Corollary~\ref{c3}, and the analog of Lemma~\ref{lema-core-not-complete} for strongly regular self-complementary graphs is proved by the application of~\cite[Corolarry 3.8.6 and Theorem~3.8.4]{godsil_cambridge} instead of Lemma~\ref{clique-coclique}. Finally, a nontrivial result from~\cite{roberson-strgreg} is applied, which states that a primitive strongly regular graph is either a core or its core is complete.
\end{remark}

Recall from the introduction that many `nice' graphs are either cores or their cores are complete. Hence it is expected that many vertex-transitive self-complementary graphs fulfil the assumptions in Theorem~\ref{thm-vt}. Consequently, we state the following open problem.

\begin{openproblem}\label{op}
Does there exist a vertex-transitive self-complementary graph $\Gamma$  such that $\Gamma\bar{\Gamma}$ is not a core?
\end{openproblem}

Note that edge-transitive self-complementary graphs are also arc-transitive (see~\cite{sc-survey}). Since self-complementary graphs are connected by Lemma~\ref{l3}, it follows that any edge-transitive self-complementary graph is also vertex-transitive.
However, such graphs are always strongly regular (cf.~\cite{sc-survey}) and therefore Theorem~\ref{thm-strg} implies that $\Gamma\bar{\Gamma}$ is a core. Despite the orders of vertex-transitive self-complementary graphs were fully determined by Muzychuk~\cite{muzychuk}, there is a major gap between the understanding of vertex-transitive self-complementary graphs and the understanding of edge-transitive self-complementary graphs. In fact, the later were completely characterized by Peisert~\cite{peisert}. Recall also that the first non-Cayley vertex-transitive self-complementary graph was constructed only in 2001~\cite{praeger}, and the construction is highly nontrivial. We believe that all these facts indicate that Open Problem~\ref{op} may be challenging. Below we consider the only vertex-transitive self-complementary graphs we are aware of, which are neither cores nor their cores are complete graphs (i.e. they do not satisfy the assumption in Theorem~\ref{thm-vt}). They do not solve Open Problem~\ref{op}.

\begin{prop}\label{prop-lex1}
Let $\Gamma_1,\Gamma_2$ be vertex-transitive self-complementary graphs. If $\textrm{core}(\Gamma_1)$ is complete, $\Gamma_2$ is a core, and $\Gamma=\Gamma_1[\Gamma_2]$, then $\Gamma\bar{\Gamma}$ is a core.
\end{prop}
\begin{proof}
Recall that the lexicographic product $\Gamma$ is vertex-transitive and self-complementary. By Corollary~\ref{c5} only possibilities (i), (ii), (iii) in Lemma~\ref{lemma-core} may occur. Suppose that (iii) is correct, that is, $V(\textrm{core}(\Gamma\bar{\Gamma}))\subseteq W_2$ and $\textrm{core}(\Gamma\bar{\Gamma})\cong \textrm{core}(\bar{\Gamma})$. Let $\phi: \bar{\Gamma}_1\to K_m$ be any endomorphism onto a complete core of $\bar{\Gamma}_1$. Then the map $\bar{\Gamma}=\bar{\Gamma}_1[\bar{\Gamma}_2]\to K_m[\bar{\Gamma}_2]$, defined by $(v_1,v_2)\mapsto (\phi(v_1),v_2)$ for all $v_1\in V(\bar{\Gamma}_1),v_2\in V(\bar{\Gamma}_2)$, is a graph homomorphism. By Lemma~\ref{lemma-knauer} it follows that
$$\textrm{core}(\Gamma\bar{\Gamma})\cong \textrm{core}(\bar{\Gamma})=K_m[\bar{\Gamma}_2].$$
Let $\psi: \textrm{core}(\Gamma\bar{\Gamma})\to K_m[\bar{\Gamma}_2]$ be any isomorphism and let $\Phi: \Gamma\bar{\Gamma}\to \textrm{core}(\Gamma\bar{\Gamma})$ be any endomorphism. Let $\psi_1, \psi_2$ be defined as in the proof of Theorem~\ref{thm-vt}. Then $f_2:=\psi \circ (\Phi|_{W_2})\circ \psi_2\in \textrm{Hom}(\bar{\Gamma}_1[\bar{\Gamma}_2],K_m[\bar{\Gamma}_2])$. Similarly, if $\sigma : \Gamma\to \bar{\Gamma}$ is any antimorphism and $f_1:=\psi \circ (\Phi|_{W_1})\circ \psi_1$, then $f_1\circ \sigma^{-1}\in \textrm{Hom}(\bar{\Gamma}_1[\bar{\Gamma}_2],K_m[\bar{\Gamma}_2])$. By Lemma~\ref{zadnja-lex} and Corollary~\ref{c7} we deduce that
\begin{equation*}
f_2=(\varphi_2,\beta_2)\quad \textrm{and}\quad f_1=(\varphi_1,\beta_1)\circ (\sigma_1,\gamma),
\end{equation*}
where
\begin{align*}
&\varphi_1, \varphi_2\in \textrm{Hom}(\bar{\Gamma}_1,K_m),\\
&\beta_1,\beta_2: V(\bar{\Gamma}_1)\to \textrm{Hom}(\bar{\Gamma}_2,\bar{\Gamma}_2)=\textrm{Aut}(\bar{\Gamma}_2),\\
&\sigma_1\in \overline{\textrm{Aut}(\Gamma_1)},\\
&\gamma : V(\Gamma_1)\to \overline{\textrm{Aut}(\Gamma_2)}.
\end{align*}
That is,
\begin{align}
\label{e102}f_2(v_1,v_2)&=\big(\varphi_2(v_1),\beta_2(v_1)(v_2)\big),\\ \label{e103}f_1(v_1,v_2)&=\Big((\varphi_1\circ\sigma_1)(v_1),\big(\beta_1(\sigma_1(v_1))\circ\gamma(v_1)\big)(v_2)\Big)
\end{align}
for all $v_1\in V(\Gamma_1),v_2\in V(\Gamma_2)$. Pick any $v\in V(K_m)$. By Corollary~\ref{lemma-preimage}, $\varphi_2^{-1}(v)$ is an independent set in $\bar{\Gamma}_1$ of order $\alpha(\bar{\Gamma}_1)$, $(\varphi_1\circ\sigma_1)^{-1}(v)$ is a clique in~$\bar{\Gamma}_1$ of order $\alpha(\Gamma_1)=\omega(\bar{\Gamma}_1)$, and $\alpha(\bar{\Gamma}_1)\omega(\bar{\Gamma}_1)=|V(\Gamma_1)|$. By Lemma~\ref{clique-coclique}, there exists $$v_1\in \varphi_2^{-1}(v)\cap (\varphi_1\circ\sigma_1)^{-1}(v).$$ Since $\big(\beta_1(\sigma_1(v_1))\circ\gamma(v_1)\big)^{-1} \circ \beta_2(v_1)\in \overline{\textrm{Aut}(\Gamma_2)}$, Lemma~\ref{fiksna_tocka} yields a vertex $v_2\in V(\Gamma_2)$ such that $$\big(\beta_1(\sigma_1(v_1))\circ\gamma(v_1)\big)(v_2)=\beta_2(v_1)(v_2).$$ Consequently,
$f_1(v_1,v_2)=f_2(v_1,v_2)$ by \eqref{e102}-\eqref{e103}. Therefore
\begin{align*}
\Phi\big((v_1,v_2),1\big)&=(\psi^{-1}\circ f_1\circ \psi_1^{-1})\big((v_1,v_2),1\big)\\
&=(\psi^{-1}\circ f_2\circ \psi_2^{-1})\big((v_1,v_2),2\big)=\Phi\big((v_1,v_2),2\big),
\end{align*}
which is a contradiction, since $\{\big((v_1,v_2),1\big),\big((v_1,v_2),2\big)\}$ is an edge in $\Gamma\bar{\Gamma}$.

In the same way we see that (ii) in Lemma~\ref{lemma-core} is not possible.
\end{proof}

If we swap the assumptions regarding $\Gamma_1$ and $\Gamma_2$ in Proposition~\ref{prop-lex1}, then we are able to deduce the same conclusion under the additional condition that $\Gamma_1[\textrm{core}(\Gamma_2)]$ is a core.
\begin{prop}\label{prop-lex2}
Let $\Gamma_1,\Gamma_2$ be vertex-transitive self-complementary graphs. If $\textrm{core}(\Gamma_2)$ is complete, $\Gamma_1[\textrm{core}(\Gamma_2)]$ is a core, and $\Gamma=\Gamma_1[\Gamma_2]$, then $\Gamma\bar{\Gamma}$ is a~core.
\end{prop}
Since Propositions~\ref{prop-lex1} and~\ref{prop-lex2} have similar proofs, we sketch only the main differences.
\begin{proof}[Sketch of the proof]
Denote $\textrm{core}(\bar{\Gamma}_2)=:K_m$. Similarly as in the proof of Proposition~\ref{prop-lex1} we deduce that the condition (iii) in Lemma~\ref{lemma-core} would yield
$$\textrm{core}(\Gamma\bar{\Gamma})\cong \textrm{core}(\bar{\Gamma})=\bar{\Gamma}_1[K_m].$$ Here, the only difference is the application of Lemma~\ref{lemma-knauer}, which is replaced by the assumption that $\bar{\Gamma}_1[K_m]$ is a core. Let $\psi: \textrm{core}(\Gamma\bar{\Gamma})\to \bar{\Gamma}_1[K_m]$ be any isomorphism, and define $\Phi, \psi_1, \psi_2, f_1, f_2, \sigma$ as in the proof of Proposition~\ref{prop-lex1}.
Then $f_2, f_1\circ \sigma^{-1}\in \textrm{Hom}(\bar{\Gamma}_1[\bar{\Gamma}_2],\bar{\Gamma}_1[K_m])$. Hence, these maps may be interpreted as members of $\textrm{End}(\bar{\Gamma}_1[\bar{\Gamma}_2])$, which equals $\textrm{End}(\bar{\Gamma}_1)\wr \textrm{End}(\bar{\Gamma}_2)$ by Corollary~\ref{c8}.
Note that since $\Gamma_1[\textrm{core}(\Gamma_2)]$ is a core, the graph $\bar{\Gamma}_1\cong \Gamma_1$ is a core as well. Therefore
$f_2=(\varphi_2,\beta_2)$ and $f_1=(\varphi_1,\beta_1)\circ (\sigma_1,\gamma)$, where
\begin{align*}
&\varphi_1, \varphi_2\in \textrm{End}(\bar{\Gamma}_1)=\textrm{Aut}(\bar{\Gamma}_1),\\
&\beta_1,\beta_2: V(\bar{\Gamma}_1)\to \textrm{End}(\bar{\Gamma}_2),\\
&\sigma_1\in \overline{\textrm{Aut}(\Gamma_1)},\\
&\gamma : V(\Gamma_1)\to \overline{\textrm{Aut}(\Gamma_2)},
\end{align*}
and the image of $\beta_i(v_1)$ is in $V(K_m)$ for all $v_1\in V(\bar{\Gamma}_1)$ and $i=1,2$. Clearly, \eqref{e102}-\eqref{e103} is still true.
Since $\varphi_2^{-1}\circ \varphi_1\circ \sigma_1\in \overline{\textrm{Aut}(\Gamma_1)}$, Lemma~\ref{fiksna_tocka} yields a vertex $v_1\in V(\Gamma_1)$ such that
$$(\varphi_1\circ\sigma_1)(v_1)=\varphi_2(v_1).$$
Let $v\in V(K_m)$ be any vertex. By Corollary~\ref{lemma-preimage}, $(\beta_2(v_1))^{-1}(v)$ is an independent set in $\bar{\Gamma}_2$ of order $\alpha(\bar{\Gamma}_2)$, $\big(\beta_1(\sigma_1(v_1))\circ\gamma(v_1)\big)^{-1}(v)$ is a clique in $\bar{\Gamma}_2$ of order $\alpha(\Gamma_2)=\omega(\bar{\Gamma}_2)$, and $\alpha(\bar{\Gamma}_2)\omega(\bar{\Gamma}_2)=|V(\bar{\Gamma}_2)|$. By Lemma~\ref{clique-coclique} there exists $$v_2\in (\beta_2(v_1))^{-1}(v) \cap \big(\beta_1(\sigma_1(v_1))\circ\gamma(v_1)\big)^{-1}(v).$$ Consequently, \eqref{e102}-\eqref{e103}  imply that $f_1(v_1,v_2)=f_2(v_1,v_2)$ and we deduce the same contradiction as in the proof of Proposition~\ref{prop-lex1}.
\end{proof}

\begin{exa}
Paley graph $\textrm{Paley}(q)$ is a core if $q$ is not a square, while its core is complete if $q$ is a square~\cite[Proposition~3.3]{cameron}. Hence, $\Gamma_1=\textrm{Paley}(q_1)$ and $\Gamma_2=\textrm{Paley}(q_2)$ satisfy the assumptions in Proposition~\ref{prop-lex1} whenever $q_1$ is a square and $q_2$ is not. In the reversed order, graphs $\Gamma_1=\textrm{Paley}(q_2)$ and $\Gamma_2=\textrm{Paley}(q_1)$ satisfy the assumptions in Proposition~\ref{prop-lex2} at least for $q_2=5$. In fact, $\textrm{Paley}(5)\cong C_5$ and $\Gamma_1[\textrm{core}(\Gamma_2)]\cong C_5[K_{\sqrt{q_1}}]$ is a core by~\cite[Theorem~3.11]{knauer}.\hfill$\Box$
\end{exa}

Clearly, if vertex-transitive self-complementary graphs $\Gamma_1$ and $\Gamma_2$ have both complete cores, then the same is true for $\Gamma=\Gamma_1[\Gamma_2]$, and therefore $\Gamma\bar{\Gamma}$ is a core by Theorem~\ref{thm-vt}. In view of Open Problem~\ref{op},
the following three combinations in the lexicographic product $\Gamma_1[\Gamma_2]$ of vertex-transitive self-complementary graphs should be also addressed:
\begin{itemize}
\item $\Gamma_1$ is a core, $\textrm{core}(\Gamma_2)=K_m$, and $\Gamma_1[K_m]$ is neither a core nor its core is complete (if such graphs exist);
\item $\Gamma_1$ and $\Gamma_2$ are both cores, and $\Gamma_1[\Gamma_2]$ is neither a core nor its core is complete (if such graphs exist);
\item lexicographic products with more than two factors.
\end{itemize}
However, they seem quite challenging. In fact, in general we only know that $\textrm{core}(\Gamma_1[\Gamma_2])=\Gamma_1'[\textrm{core}(\Gamma_2)]$, where $\Gamma_1'$ is a subgraph in $\Gamma_1$~(see~\cite{hahn}).\bigskip

\noindent{\bf Acknowledgements and comments.} Graph $\Gamma\bar{\Gamma}$ was denoted by $\Gamma\equiv\bar{\Gamma}$ in an earlier draft of this paper. The author used the same notation also in his talk presentation\cprotect\footnote{Slides are available at \verb|https://8ecm.si/system/admin/abstracts/presentations/|\\
\verb|000/000/098/original/Orel_8ECM_2021.pdf?1626466439|. Accessed October 2, 2021.} at the 8th European Congress of Mathematics. Later during 8ECM he attended
the talk of Paula Carvalho, where he learned that graph $\Gamma\bar{\Gamma}$ was introduced in~\cite{haynes2007} and is known as the `complementary prism'. He would like to thank Paula Carvalho also for letting him know that Lemmas~\ref{carvalho} and \ref{p6} were already proved in \cite{carvalho2018} and~\cite{haynes2007}, respectively.

This work is supported in part by the Slovenian Research Agency (research program P1-0285
and research projects N1-0140, N1-0208 and N1-0210).

\end{document}